\documentclass[11pt]{article}
\usepackage{url}

\usepackage{makeidx}  
\usepackage{graphicx}
\usepackage{amsmath}
\usepackage{amsthm}
\usepackage{amssymb}
\usepackage{amsfonts}
\usepackage{amscd}
\usepackage[matrix,arrow,curve]{xy}
\usepackage{mathabx}

\begin{document}

\pagestyle{headings}  

\newtheorem{theorem}{Theorem}[section]
\newtheorem{lemma}[theorem]{Lemma}
\newtheorem{proposition}[theorem]{Proposition}
\newtheorem{prop}[theorem]{Proposition}
\newtheorem{corollary}[theorem]{Corollary}
\newtheorem{cor}[theorem]{Corollary}

\newtheorem*{theorem*}{Theorem}

\theoremstyle{remark}

\theoremstyle{definition}
\newtheorem{problem}[theorem]{Problem}
\newtheorem{definition}[theorem]{Definition}
\newtheorem{remark}[theorem]{Remark}
\newtheorem{q}[theorem]{Question}
\newtheorem{question}[theorem]{Question}
\newtheorem*{acknowledgment}{Acknowledgment}

\newcommand\hide[1]{\commented{gray}{Hidden:}{#1}}
\renewcommand\hide[1]\empty


\title{On ultrafilter extensions of first-order models 
       \\and ultrafilter interpretations}
\author{Nikolai L.~Poliakov
         \footnote{National Research University 
          Higher School of Economics, Moscow.},
        Denis~I.~Saveliev
\footnote{Russian Academy of Sciences, 
          Institute for Information Transmission Problems. 
          }
          }

\date{}

\maketitle

\newcommand{\dom}{ {\mathop{\mathrm {dom\,}}\nolimits} }
\newcommand{\ran}{ {\mathop{\mathrm{ran\,}}\nolimits} }

\newcommand{\app}{ {\mathop{\mathrm {app}}\nolimits} }
\newcommand{\ext}{ {\mathop{\mathrm {ext}}\nolimits} }
\newcommand{\cur}{ {\mathop{\mathrm {cur}}\nolimits} }

\newcommand{\rel}{ {\mathop{\mathrm {rel}}\nolimits} }
\newcommand{\fnc}{ {\mathop{\mathrm {fnc}}\nolimits} }

\newcommand{\inter}{ {\mathop{\mathrm {int}}\nolimits} }
\newcommand{\cl}{ {\mathop{\mathrm {cl}}\nolimits} }
\newcommand{\lcl}{ {\mathop{\mathrm {lcl\,}}\nolimits} }
\newcommand{\rcl}{ {\mathop{\mathrm {rcl\,}}\nolimits} }
\newcommand{\cof}{ {\mathop{\mathrm {cof\,}}\nolimits} }
\newcommand{\add}{ {\mathop{\mathrm {add\,}}\nolimits} }
\newcommand{\sat}{ {\mathop{\mathrm {sat\,}}\nolimits} }
\newcommand{\tc}{ {\mathop{\mathrm {tc\,}}\nolimits} }
\newcommand{\unif}{ {\mathop{\mathrm {unif\,}}\nolimits} }
\newcommand{\uhr}{\!\upharpoonright\!}
\newcommand{\lra}{ {\leftrightarrow} }
\newcommand{\ot}{ {\mathop{\mathrm {ot\,}}\nolimits} }
\newcommand{\ol}{\overline}
\newcommand{\cnc}{ {^\frown} }
\newcommand{\image}{\/``\,}
\newcommand{\scc}{\beta\!\!\!\!\beta}
\newcommand{\wh}{\widehat}
\newcommand{\wt}{\widetilde}
\newcommand{\inn}{\mathrm{in\,}}
\newcommand{\id}{\mathrm{id}}

\newcommand{\RC}{\mathrm{RC}}
\newcommand{\RCO}{\mathrm{RCO}}
\newcommand{\RCl}{\mathrm{RCl}}
\newcommand{\RO}{\mathrm{RO}}
\newcommand{\RClop}{\mathrm{RClop}}
\newcommand{\RegCl}{\mathrm{RegCl}}
\newcommand{\RegO}{\mathrm{RegO}}
\newcommand{\Cl}{\mathrm{Cl}}
\newcommand{\Clop}{\mathrm{Clop}}

\begin{abstract}
There exist two known types of ultrafilter 
extensions of first-order models, both in 
a~certain sense canonical. One of them~\cite{Goranko} 
comes from modal logic and universal algebra, 
and in fact goes back to~\cite{Jonsson Tarski}. 
Another one~\cite{Saveliev,Saveliev(inftyproc)} 
comes from model theory and algebra of ultrafilters, 
with ultrafilter extensions of 
semigroups~\cite{Hindman Strauss} 
as its main precursor. By a~classical fact of 
general topology, the space of ultrafilters over 
a~discrete space is its largest compactification. 
The main result 
of~\cite{Saveliev,Saveliev(inftyproc)}, which 
confirms a~canonicity of this extension, generalizes 
this fact to discrete spaces endowed with an 
arbitrary first-order structure. An analogous result 
for the former type of ultrafilter extensions was 
obtained in~\cite{Saveliev(2 concepts)}. Results 
of such kind are referred to as extension theorems.

After a~brief introduction, 
we offer a~uniform approach to both types of 
extensions based on the idea to extend the extension 
procedure itself. We propose a~generalization of 
the standard concept of first-order interpretations 
in which functional and relational symbols are 
interpreted rather by ultrafilters over sets of 
functions and relations than by functions and 
relations themselves, and define ultrafilter models 
with an appropriate semantics for them. 
We provide two specific operations which 
turn ultrafilter models into ordinary models, 
establish necessary and sufficient conditions under 
which the latter are the two canonical ultrafilter 
extensions of some ordinary models, and obtain  
a~topological characterization of ultrafilter models. 
We generalize a~restricted version 
of the extension theorem to ultrafilter models. 
To formulate the full version, we propose a~wider 
concept of ultrafilter models with their semantics 
based on limits of ultrafilters, and show that 
the former concept can be identified, in a~certain 
way, with a~particular case of the latter; moreover, 
the new concept absorbs the ordinary concept 
of models. We provide two more specific operations 
which turn ultrafilter models in the narrow sense 
into ones in the wide sense, 
and establish necessary and sufficient conditions 
under which ultrafilter models in the wide sense are 
the images of ones in the narrow sense under these 
operations, and also are two canonical ultrafilter
extensions of some ordinary models. Finally, 
we establish three full versions of the extension 
theorem for ultrafilter models in the wide sense.

The results of the first three sections of this paper 
were partially announced in~\cite{Poliakov Saveliev}.
\end{abstract}

{\let\thefootnote\relax}

{\let\thefootnote\relax}


\section{Introduction}

In this section, we recall main definitions and facts  
concerning ultrafilter extensions of arbitrary maps, 
relations, and first-order models. All results mentioned 
here are established in various previous papers, so 
we omit their proofs. The section provides also some 
(of necessity incomplete) historical information.

Given a~set~$X$, let $\scc X$ be the set of 
ultrafilters over~$X$. As usual, we let 
$X\subseteq\scc X$ by identifying each $x\in X$
with the principal ultrafilter given by~$x$.
Fix a~first-order language and consider 
an arbitrary model~$\mathfrak A$ of the language:
$$
\mathfrak A=(X,F,\ldots,R,\ldots)
$$ 
with the universe~$X$, operations $F,\ldots\,$, 
and relations $R,\ldots$~.

\begin{definition}\label{def: abstract u e}
An ({\it abstract}) {\it ultrafilter extension} 
of~$\mathfrak A$ is any model~$\mathfrak A'$ 
in the same language of form 
$$
\mathfrak A'=
(\scc X,F\,',\ldots,R\,',\ldots)
$$ 
with the universe $\scc X$ and operations 
$F\,',\ldots$ and relations $R\,',\ldots$ 
on~$\scc X$ that extend $F,\ldots$ and $R,\ldots$\,, 
respectively. 
\end{definition}

There are essentially {\it two\/} known ways to 
extend relations by ultrafilters, and {\it one\/} 
to extend maps. Particular instances of these 
extensions were discovered by various authors in 
different time and different areas, often without 
a~knowledge of parallel studies in adjacent areas. 
It is convenient to describe these extensions in 
topological terms.

Recall that $\scc X$~carries a~natural topology 
generated by basic open sets 
$$
\wt A=\{\mathfrak u\in\scc X:A\in\mathfrak u\}
$$ 
for all $A\subseteq X$. Easily, the sets are also closed, 
so the space~$\scc X$ is zero-dimensional. Moreover, 
$\scc X$~is compact, Hausdorff, extremally disconnected 
(the closure of any open set is open), and the largest 
compactification of the discrete space~$X$. This means 
that $X$~is dense in~$\scc X$ and every (trivially 
continuous) map~$h$ of~$X$ into any compact Hausdorff 
space~$Y$ uniquely extends to a~continuous map~$\wt h$ 
of~$\scc X$ into~$Y$:
$$
\xymatrix{
&\scc X\,
\ar@{-->}^{\wt{h}\quad}[drr]&&
\\
&X\,
\ar[u]
\ar[rr]^{h}
&&\,Y&
}
$$
by letting for all $\mathfrak u\in\scc X$,
$$
\wt{h}(\mathfrak u)=y
\;\;\text{where}\;\;
\{y\}=
\bigcap_{A\in\mathfrak u}\cl_Y\,h\image A.
$$
(As usual, $\cl_S\,B$~is the closure of~$B$ in~$S$, 
and $f\image B$ is the image of~$B$ under~$f$.) 
The largest compactification of Tychonoff spaces, 
usually referred to as the {\it Stone--\v{C}ech 
compactification}, was discovered independently 
by \v{C}ech~\cite{Cech} and M.~Stone~\cite{Stone}; 
then Wallman~\cite{Wallman} did the same 
for $T_1$~spaces (by using ultrafilters on 
lattices of closed sets); see 
\cite{Hindman Strauss,Comfort Negrepontis,Engelking} 
for more information. 


The ultrafilter extension of a~unary relation~$R$ on
a~set~$X$ is exactly the basic (cl)open set~$\wt R$, 
and the ultrafilter extension of a~unary map 
$F:X\to Y$, where $Y$~is a~compact Hausdorff space 
(for operations~$F$ on~$X$ we let $Y=\scc X$ 
as $X\subseteq\scc X$), is exactly its continuous 
extension~$\wt F$.
Thus in the unary case, the procedure gives classical 
objects known since 1930s. As for maps and relations 
of greater arities, several instances of their 
ultrafilter extensions were discovered only in 1960s.

\newpage

\vskip+1em
\noindent
\textbf{\textit{Ultrafilter extensions of maps.}}
Studying ultraproducts, Kochen~\cite{Kochen} and 
Frayne, Morel, and Scott~\cite{Frayne Morel Scott} 
considered a~``multiplication" of ultrafilters, 
which actually is the ultrafilter extension of 
the $n$-ary operation of taking $n$-tuples. They 
shown that the successive iteration of ultrapowers by 
ultrafilters $\mathfrak u_1,\ldots,\mathfrak u_n$ is 
isomorphic to a~single ultrapower by their ``product". 
This has leaded to the general construction of iterated 
ultrapowers, invented by Gaifman and elaborated by 
Kunen, which has become common in model theory and 
set theory (see~\cite{Chang Keisler,Kanamori}).

Ultrafilter extensions of semigroups appeared in~1960s 
as subspaces of function spaces. To the best of our 
knowledge, the first explicit construction of the 
semigroup that is the ultrafilter extension of 
a~group is due to Ellis~\cite{Ellis}; he also proved 
the existence of idempotents in compact Hausdorff 
semigroups with one-sided continuity~\cite{Ellis 58}. 
In~1970s Galvin and Glazer applied these facts to give 
an easy proof of what now known as Hindman's Finite 
Sums Theorem; the key idea was to use ultrafilters 
that are idempotent w.r.t.~the extended operation. 
Then the method was developed by Bergelson, Blass, 
van~Douwen, Hindman, Protasov, Strauss, and many others, 
and provided numerous Ramsey-theoretic applications 
in number theory, algebra, topological dynamics, 
and ergodic theory.  
The book~\cite{Hindman Strauss} is a~comprehensive 
treatise of this area, with an historical information. 
This technique was recently applied for obtaining 
analogous results for certain non-associative algebras 
(see~\cite{Saveliev(Hindman),Saveliev(idempotents)}).

Ultrafilter extensions of arbitrary $n$-ary maps 
have been introduced independently in recent works 
by Goranko~\cite{Goranko} and
Saveliev~\cite{Saveliev,Saveliev(inftyproc)}. 

\begin{definition}\label{def: u e map}
For a~map $F:X_1\times\ldots\times X_n\to Y$, 
the extended map 
$\wt F:\scc X_1\times\ldots\times\scc X_n\to\scc Y$ 
is defined by letting
\begin{gather*}
\wt F(\mathfrak u_1,\ldots,\mathfrak u_n)=
\\
\bigl\{A\subseteq Y:
\{x_1\in X_1:
\ldots
\{x_n\in X_n:
F(x_1,\ldots,x_n)\in A
\}\in\mathfrak u_n
\ldots
\}\in\mathfrak u_1
\bigr\}.
\end{gather*}
\end{definition}

One can simplify this cumbersome notation by 
introducing {\it ultrafilter quantifiers\/}.  
For every ultrafilter~$\mathfrak u$ over a~set~$X$ 
and formula $\varphi(x,\ldots)$ with parameters 
$x,\ldots$ valuated over~$X$, let 
$$
(\forall^{\,\mathfrak u}x)\,\varphi(x,\ldots)
\;\;\text{mean}\;\;
\{x:\varphi(x,\ldots)\}\in\mathfrak u.
$$
In fact, such quantifiers are a~special kind 
of second-order quantifiers:
$(\forall^{\,\mathfrak u}x)$ is equivalent to
$(\forall A\in\mathfrak u)(\exists x\in A)$, 
and also (since $\mathfrak u$~is ultra) to 
$(\exists A\in\mathfrak u)(\forall x\in A)$. 
Note also that ultrafilter quantifiers are self-dual, 
i.e.~$\forall^{\,\mathfrak u}$ and 
$\exists^{\,\mathfrak u}$ coincide; 
they generally do not commute with each other, i.e. 
$(\forall^{\,\mathfrak u}x)(\forall^{\,\mathfrak v}y)$ 
and 
$(\forall^{\,\mathfrak v}y)(\forall^{\,\mathfrak u}x)$ 
are generally not equivalent; and if $\mathfrak u$~is 
the principal ultrafilter given by $a\in X$ then 
$(\forall^{\,\mathfrak u}x)\varphi(x,\ldots)$
is reduced to $\varphi(a,\ldots)$.

Now the definition above can be rewritten as follows:
\begin{align*}
\wt F(\mathfrak u_1,\ldots,\mathfrak u_n)=
\bigl\{A\subseteq Y:
(\forall^{\,\mathfrak u_1}x_1)
\ldots
(\forall^{\,\mathfrak u_n}x_n)\;
F(x_1,\ldots,x_n)\in A
\bigr\}.
\end{align*}

The map~$\wt F$ can be also described as 
the composition of the ultrafilter extension 
of taking $n$-tuples, which maps 
$\scc X_1\times\ldots\times\scc X_n$ into $\scc(X_1\times\ldots\times X_n)$, and 
the continuous extension of $F$ considered 
as a~unary map, which maps 
$\scc(X_1\times\ldots\times X_n)$ into~$\scc Y$.

Not many properties of original maps are preserved 
under their ultrafilter extensions. Specific 
identities preserved under ultrafilter extensions 
(e.g.~associativity is so while commutativity 
and idempotency are not) are described 
in \cite{Saveliev(inftyproc)}, Theorem~5.3.


\vskip+1em
\noindent
\textbf{\textit{Ultrafilter extensions of relations.}}
One type of ultrafilter extensions of relations 
goes back to a~seminal paper by J{\'o}nsson and
Tarski~\cite{Jonsson Tarski} where they have been 
appeared implicitly, in terms of representations of 
Boolean algebras with operators. For binary relations, 
their representation theory was rediscovered in 
modal logic by Lemmon~\cite{Lemmon} who credited 
much of this work to Scott (see footnote~6 on p.~204); 
see also~\cite{Lemmon Scott}. Goldblatt and
Thomason~\cite{Goldblatt Thomason} (where Section~2 was 
entirely due to Goldblatt) used this to characterize 
modal definability; the term ``ultrafilter extension'' 
has been coined probably in the subsequent work by 
van~Benthem~\cite{van Benthem} (for modal definability 
see also~\cite{van Benthem 88,Venema,Blackburn et al}). 
Later Goldblatt~\cite{Goldblatt} considered the 
extension of $n$-ary relations in the context of 
universal algebra and model theory.

The following definition is equivalent to one appeared 
in \cite{Goranko} (or \cite{Jonsson Tarski,Goldblatt}):

\begin{definition}\label{def: model u e relation}
For a~relation $R\subseteq X_1\times\ldots\times X_n$, 
the extended relation 
$R^*\subseteq\scc X_1\times\ldots\times\scc X_n$ 
is defined by letting
\begin{gather*}
R^*(\mathfrak u_1,\ldots,\mathfrak u_n)
\;\;\text{iff}\;\;
\\
(\forall A_1\in\mathfrak u_1)
\ldots
(\forall A_n\in\mathfrak u_n)
(\exists x_1\in A_1)
\ldots
(\exists x_n\in A_n)\;
R(x_1,\ldots,x_n).
\end{gather*}
\end{definition}

The first-order formulas corresponding to 
so-called canonical modal formulas (e.g.~to 
all Sahlqvist formulas) are preserved under 
passing from $\mathfrak A$ to~$\mathfrak A^*$, 
provided $\mathfrak A$ is a~model of a~relational 
language (see~\cite{van Benthem,Blackburn et al}).

Another type of ultrafilter extensions 
of $n$-ary relations has been recently discovered 
in~\cite{Saveliev,Saveliev(inftyproc)}:

\begin{definition}\label{def: modal u e relation}
For a~relation $R\subseteq X_1\times\ldots\times X_n$, 
the extended relation 
$\wt R\subseteq\scc X_1\times\ldots\times\scc X_n$ 
is defined by letting
\begin{gather*}
\wt R(\mathfrak u_1,\ldots,\mathfrak u_n)
\;\;\text{iff}\;\;
\\
\bigl\{x_1\in X_1:
\ldots
\{x_n\in X_n:
R(x_1,\ldots,x_n)
\}\in\mathfrak u_n
\ldots
\bigr\}\in\mathfrak u_1.
\end{gather*}
\end{definition}

Rewriting this via ultrafilter quantifiers, 
we get an easier formulation:
\begin{align*}
\wt R(\mathfrak u_1,\ldots,\mathfrak u_n)
\;\;\text{iff}\;\;
(\forall^{\,\mathfrak u_1}x_1)
\ldots
(\forall^{\,\mathfrak u_n}x_n)\;
R(x_1,\ldots,x_n).
\end{align*}
By decoding ultrafilter quantifiers, this also 
can be rewritten by
\begin{gather*}
\wt R(\mathfrak u_1,\ldots,\mathfrak u_n)
\;\;\text{iff}\;\;
\\
(\forall A_1\in\mathfrak u_1)
(\exists x_1\in A_1)
\ldots
(\forall A_n\in\mathfrak u_n)
(\exists x_n\in A_n)\;
R(x_1,\ldots,x_n),
\end{gather*}
whence it clearly follows that one of 
the two relations is included into another: 
$$\wt R\subseteq R^*.$$ 
If $R$~is a~unary relation, both extensions, 
$\wt R$ and~$R^*$, coincide with the basic open 
set given by~$R$ (and with $\cl_{\scc X}R$, 
the closure of $R$ in the space~$\scc X$). 
If a~binary relation~$R$ is functional, then $R^*$ 
(but not~$\wt R$) coincides with the above-defined  
extension of~$R$ considered as a~unary map; this 
does not work for relations of bigger arities. 
An easy instance of the $\,\wt{\;}\;$-extensions, 
where $R$~are linear orders, was studied 
in~\cite{Saveliev(orders)}.

A~systematic comparative study of both extensions 
(for binary~$R$) is undertaken 
in~\cite{Saveliev(2 concepts)}. 
In particular, it is shown there that the ${}^*\,$- 
and the $\,\wt{\;}\;$-extensions have a~dual 
character w.r.t.~relation-algebraic operations: 
the ${}^*\,$-extension commutes with composition 
and inversion but not Boolean operations except 
for union, while the
$\,\wt{\;}\;$-extension commutes with all Boolean 
operations but neither composition nor inversion. 
Also \cite{Saveliev(2 concepts)}~provides topological 
characterizations of $\wt R$ and~$R^*$ in terms of
appropriate closure operations and in terms of 
Vietoris-type topologies (regarding $R$ as 
multi-valued maps).


\vskip+1em
\noindent
\textbf{\textit{Ultrafilter extensions of models.}}
Ultrafilter extensions of arbitrary first-order 
models were defined and studied for the first time
independently in~\cite{Goranko} and in~\cite{Saveliev} 
with two distinct versions of extended relations: 
Goranko considered models with the ${}^*\,$-extensions of
relations and Saveliev with their $\,\wt{\;}\;$-extensions. 
Here we shall consider both types of extensions; for 
a~given model~$\mathfrak A$ denote them by $\mathfrak A^*$ 
and $\wt{\,\mathfrak A\,}$, respectively:%
\footnote{
Another notation was used in~\cite{Goranko},  
where $\mathfrak A^*$ was denoted 
by~${\mathbf U}(\mathfrak A)$, and 
in~\cite{Saveliev,Saveliev(inftyproc),Saveliev Shelah},  
where $\wt{\,\mathfrak A\,}$ was denoted 
by~$\scc\,\mathfrak A$.
}

\begin{definition}\label{def: u e model}
For an arbitrary model 
$\mathfrak A=(X,F,\ldots,R,\ldots)$ we let
\begin{align*}
\mathfrak A^*=
\bigl(\scc X,\wt F,\ldots,R^*,\ldots\bigr)
\;\;\text{ and }\;\;
\wt{\,\mathfrak A\,}=
\bigl(\scc X,\wt F,\ldots,\wt R,\ldots\bigr).
\end{align*}
\end{definition}

Since for any relation~$R$ we have $\wt R\subseteq R^*$, 
the following observation is obvious:

\begin{theorem}\label{homo-of-model-to-modal}
For any model~$\mathfrak A$ with the universe~$X$ 
the identity map on~$\scc X$ is a~homomorphism 
of $\wt{\,\mathfrak A\,}$ onto~$\mathfrak A^*$:
$$
\xymatrix{
\wt{\,\mathfrak A\,}\ar^{\id}[rr]&&\mathfrak A^*
\\
&\mathfrak A\ar[ul]\ar[ur]&
}
$$
\end{theorem}

Therefore, all positive formulas satisfied 
in $\wt{\,\mathfrak A\,}$ are also satisfied 
in~$\mathfrak A^*$.


It follows from above mentioned facts that 
ultrafilter extensions are not elementary, 
except for certain degenerate cases. Even 
universal formulas are not preserved under 
these extensions, as seen from the example 
of a~semigroup $\mathfrak A$ without idempotents: 
the semigroup $\wt{\,\mathfrak A\,}$ does have 
an idempotent by Ellis' theorem. On the other hand, idempotents in $\wt{\,\mathfrak A\,}$ is a~key 
tool in obtaining various deep combinatorial 
results about the extended~$\mathfrak A$, 
most of which have no known alternative 
(i.e.~not using ultrafilter extensions) proofs 
(see~\cite{Hindman Strauss}). 
More generally,
some complex (typically, not first-order) 
assertions about the original model have 
counterparts about its ultrafilter extension 
which are easier to formulate and to prove; 
so, in a~sense, the non-elementarity 
of ultrafilter extensions can be their advantage 
in studying the extended models.%
\footnote{
Compare this with non-standard extensions, 
also used to prove assertions about the 
extended model, which {\em are} elementary; 
it is unclear, however, whether this technique 
produces as many results with no known 
alternative proofs as the technique based on 
ultrafilter extensions does. Interestingly, 
a~recent paper \cite{DiNasso Baglini} combines 
both techniques to obtain results in number theory. 
}


The following theorem has been appeared
in~\cite{Saveliev} and called the First Extension 
Theorem in~\cite{Saveliev(inftyproc)}:

\begin{theorem}\label{modelFET}
Let $\mathfrak A$ and~$\mathfrak B$ be two models 
of the same signature. If $h$~is a~homomorphism 
between $\mathfrak A$ and~$\mathfrak B$, then 
the continuous extension~$\wt h$ is a~homomorphism 
between $\wt{\,\mathfrak A\,}$ 
and~$\wt{\,\mathfrak B\,}$:
$$
\xymatrix{
&\wt{\,\mathfrak A\,}
\ar@{-->}^{\wt h}[rr]
&&\wt{\,\mathfrak B\,}&
\\
&\,\mathfrak A\,
\ar[rr]^{h}
\ar[u]\,
&&\mathfrak B\,
\ar[u]&
}
$$
\end{theorem}

Theorem~\ref{modelFET} on the $\,\wt{\;}\;$-extensions 
is a~precise counterpart of Theorem~\ref{modalFET} 
on the ${}^*\,$-extensions, a~principal result
of~\cite{Goranko}:

\begin{theorem}\label{modalFET}
Let $\mathfrak A$ and~$\mathfrak B$ be two models 
of the same signature. If $h$~is a~homomorphism 
between $\mathfrak A$ and~$\mathfrak B$, then 
the continuous extension~$\wt h$ is a~homomorphism 
between $\mathfrak A^*$ and~$\mathfrak B^*$:
$$
\xymatrix{
&\;\mathfrak A^*
\ar@{-->}^{\wt h}[rr]
&&\;\mathfrak B^*&
\\
&\mathfrak A
\ar[rr]^{h}
\ar[u]
&&\mathfrak B
\ar[u]&
}
$$
\end{theorem}

Both theorems remain true for isomorphic embeddings 
and some other model-theoretic interrelations 
(see \cite{Goranko,Saveliev,Saveliev(inftyproc)}). 
On the other hand, it was shown in~\cite{Saveliev Shelah} 
that Theorem~\ref{modelFET} does not hold for elementary 
embeddings, moreover, the ultrafilter extensions of 
a~model and its elementary submodel do not need to be 
elementarily equivalent.

Theorem~\ref{modelFET} is actually a~particular case of 
a~much stronger result of~\cite{Saveliev}, called the 
Second Extension Theorem in~\cite{Saveliev(inftyproc)}. 
To formulate this, we need the following concepts 
introduced in~\cite{Saveliev}.

\begin{definition}\label{def: right cont etc}
Let $X_1,\ldots,X_n,Y$ be topological spaces, and let 
$A_1\subseteq X_1,\ldots,A_{n-1}\subseteq X_{n-1}$.
An $n$-ary function $F:X_1\times\ldots\times X_n\to Y$ 
is {\it right continuous w.r.t.\/}~$A_1,\ldots,A_{n-1}$
iff for each~$i$, $1\le i\le n$, and every
$a_1\in A_1,\ldots,a_{i-1}\in A_{i-1}$ and
$x_{i+1}\in X_{i+1},\ldots,x_n\in X_n$, 
the unary map
$$
x\mapsto
F(a_1,\ldots,a_{i-1},x,x_{i+1},\ldots,x_n)
$$
of~$X_i$ into~$Y$ is continuous. 
An $n$-ary relation $R\subseteq X_1\times\ldots\times X_n$ 
is {\it right open\/} ({\it right closed\/}, 
{\it right clopen\/}, etc.)~{\it w.r.t.\/}
$A_1,\ldots,A_{n-1}$ iff for each~$i$, $1\le i\le n$, 
and every $a_1\in A_1,\ldots,a_{i-1}\in A_{i-1}$ 
and $x_{i+1}\in X_{i+1},\ldots,x_n\in X_n$, 
the set 
$$
\bigl\{x\in X_i:
R(a_1,\ldots,a_{i-1},x,x_{i+1},\ldots,x_n)
\bigr\}
$$
is open (closed, clopen, etc.)~in~$X_i$. 
\end{definition}


Theorem~\ref{modeltopology} 
(\cite{Saveliev,Saveliev(inftyproc)}) describes
topological properties of the $\,\wt{\;}\;$-extensions 
and serves as a~base of Theorem~\ref{modelSET}, 
the Second Extension Theorem 
of~\cite{Saveliev(inftyproc)}. 
(A~very particular instance of the latter theorem, in 
which the models under consideration are semigroups, has 
been appeared in~\cite{Berglund et al}, Theorem~4.5.3.)

\begin{theorem}\label{modeltopology}
Let $\mathfrak A$~be a~model. In the 
extension~$\wt{\,\mathfrak A\,}$, 
all operations are right continuous 
and all relations right clopen 
w.r.t.~the universe of~$\mathfrak A$.
\end{theorem}

\begin{theorem}\label{modelSET}
Let $\mathfrak A$ and $\mathfrak C$ be two models
of the same signature, $h$~a~homomorphism of 
$\mathfrak A$ into~$\mathfrak C$, and let 
$\mathfrak C$~be endowed with a~compact Hausdorff 
topology in which all operations are right continuous, 
and all relations are right closed, w.r.t.~the image 
of the universe of~$\mathfrak A$ under~$h$. Then 
$\wt{\,h\,}$~is a~homomorphism of 
$\wt{\,\mathfrak A\,}$ into~$\mathfrak C$:
$$
\xymatrix{
&\wt{\,\mathfrak A\,}
\ar@{-->}^{\wt{h}\;\;}[drr]&&
\\
&\mathfrak A
\ar[u]
\ar[rr]^{\;h}
&&{\;\mathfrak C}&
}
$$
\end{theorem}

Theorem~\ref{modelFET} (for homomorphisms) 
easily follows: take $\wt{\,\mathfrak B}\,$ as 
such a~$\mathfrak C$. The main meaning of 
Theorem~\ref{modelSET} is that it generalizes 
the mentioned classical \v{C}ech--Stone result to 
the case when the underlying discrete space~$X$
carries an arbitrary first-order structure.


A~natural question is whether the ${}^*\,$-extensions 
are also canonical in a~similar sense. The answer is 
positive; two following theorems are counterparts of 
Theorems \ref{modeltopology} and~\ref{modelSET}, 
respectively (essentially both have been proved 
in~\cite{Saveliev(2 concepts)}). Recall that a~set 
is {\it regular closed\/} iff it is the closure of 
an open set.

\begin{theorem}\label{modaltopology}
Let $\mathfrak A$~be a~model. In the 
extension~$\mathfrak A^*$, all relations are regular 
closed, namely, the closures of the relations 
in~$\mathfrak A$ (while all operations are right 
continuous w.r.t.~the universe of~$\mathfrak A$ 
as before).
\end{theorem}

\begin{theorem}\label{modalSET}
Let $\mathfrak A$ and $\mathfrak C$ be two models of 
the same signature, $h$~a~homomorphism of $\mathfrak A$
into~$\mathfrak C$, and let $\mathfrak C$~be 
endowed with a~compact Hausdorff topology in which 
all operations are right continuous w.r.t.~the image 
of the universe of~$\mathfrak A$ under~$h$, and 
all relations are closed. Then $\wt{\,h\,}$~is
a~homomorphism of $\mathfrak A^*$ into~$\mathfrak C$. 
$$
\xymatrix{
&\;\mathfrak A^*\,
\ar@{-->}^{\wt{h}\;\;}[drr]&
\\
&\mathfrak A\,
\ar[u]
\ar[rr]^{\;h}
&&\,\mathfrak C&
}
$$
\end{theorem}

Similarly, Theorem~\ref{modalFET} (for homomorphisms) 
follows from Theorem~\ref{modalSET}. The latter also
generalizes the Stone--\v{C}ech result for discrete 
spaces to discrete models but with a~narrow class of 
target models~$\mathfrak C$: having relations rather
closed than right closed in Theorem~\ref{modelSET}.
In the sequel, we shall refer to 
Theorems \ref{modelFET} and~\ref{modalFET} 
as the {\it First Extension Theorems}, and to 
stronger Theorems \ref{modelSET} and~\ref{modalSET} 
as the {\it Second Extension Theorems}, for the 
${}^*\,$- and $\,\wt{\;}\;$-types of ultrafilter
extensions, respectively. Let us point out that in 
all these extension theorems the converse implication 
``if $\wt h$~is a~homomorphism of an ultrafilter
extension~$\mathfrak A'$ then $h$~is a~homomorphism 
of~$\mathfrak A$\,'' is also true but trivial since 
$\mathfrak A$ is a~submodel of~$\mathfrak A'$. 
We note also that the Second Extension Theorems 
are based on an ``abstract extension theorem'' 
describing certain conditions on models, their 
submodels, homomorphisms, and topological properties, 
under which such a~homomorphism lifts from such 
a~submodel to the whole model. The theorem will be 
used in our paper, too; we shall formulate it later on 
(Theorem~\ref{abstract-ET}).


We end this introductory section with topological 
characterizations of both types of ultrafilter 
extensions of relations and ultrafilter extensions 
of maps into discrete spaces and into compact 
Hausdorff spaces.

\begin{theorem}\label{topological-char}
Let $X_1,\ldots,X_n,Y$ be discrete spaces, 
$Z$~a~compact Hausdorff space, and let the sets 
$\scc X_1,\ldots,\scc X_n,Y$ be endowed with 
the standard topology on ultrafilters, and 
$\scc X_1\times\ldots\times\scc X_n$ with 
the usual product topology. Then
\begin{enumerate}
\item[(i)]
$Q\subseteq\scc X_1\times\ldots\times\scc X_n$ 
is $\wt{R}$ for some 
$R\subseteq X_1\times\ldots\times X_n$ iff 
$Q$~is right clopen w.r.t.~$X_1,\ldots,X_{n-1}$;
\item[(ii)]
$Q\subseteq\scc X_1\times\ldots\times\scc X_n$ 
is $R^*$ for some 
$R\subseteq X_1\times\ldots\times X_n$ iff 
$Q$~is regular closed;
\item[(iii)]
$G:\scc X_1\times\ldots\times\scc X_n\to\scc Y$
is $\wt{F}$ for some 
$F:X_1\times\ldots\times X_n\to Y$ iff 
$G$~is right continuous and right open 
w.r.t.~$X_1,\ldots,X_{n-1}$;
\item[(iv)]
$G:\scc X_1\times\ldots\times\scc X_n\to Z$ 
is $\wt{H}$ for some 
$H:X_1\times\ldots\times X_n\to Z$ iff 
$G$~is right continuous w.r.t.~$X_1,\ldots,X_{n-1}$.
\end{enumerate} 
Moreover, all the four extension operations: 
$R\mapsto\wt{R}$, $R\mapsto R^*$, 
$F\mapsto\wt{F}$, $H\mapsto\wt{H}$, 
are bijections. 
\end{theorem}

This theorem shows that Theorems \ref{modeltopology}
and~\ref{modaltopology} in fact {\it characterize\/} 
the ${}^*\,$- and $\,\wt{\;}\;$-extensions via their
topological properties (and the same will follow 
from Theorems \ref{e-and-E-as-ultraextensions} 
and~\ref{e-and-E-topology} later). 


The subsequent text is organized as follows. 

In Section~2, we develop a~topological technique 
that allows us to define an ultrafilter extension 
of the procedure of ultrafilter extension itself. 
This is a~key tool for our article. Based on it, we
provide a~uniform approach to both types of ultrafilter 
extensions of relations (Theorem~\ref{modal-via-ext}), 
and furthermore, in Section~3, we define an ultrafilter
interpretation of first-order syntax, under which
functional and relational symbols are interpreted rather
by ultrafilters over sets of functions and relations 
than by their elements. We define ultrafilter models 
using ultrafilter evaluations of variables and 
ultrafilter interpretations and an appropriate 
semantics for them. We provide two specific operations, 
$e$ and~$E$, which turn ultrafilter models into ordinary 
ones, establish necessary and sufficient conditions 
under which the latter are two canonical ultrafilter 
extensions of some ordinary models 
(Theorem~\ref{e-and-E-as-ultraextensions}), and give 
a~topological characterization of ultrafilter models
(Theorem~\ref{e-and-E-topology}). 
Defining a~natural concept of homomorphisms between 
ultrafilter models, we establish 
the First Extension Theorem for ultrafilter models 
(Theorem~\ref{FET-generalized}) 
and a~stronger variant of it 
(Theorem~\ref{intermediate-ET-generalized}).

In Section~4, we define an even wider concept of
ultrafilter models together with their semantics  
based on limits of ultrafilters, and show that 
this new concept absorbs the ordinary concept of 
models with the usual semantics  
(Theorem~\ref{ordinary-models-as-new-generalized})
as well as our previous concept of ultrafilter 
models with their semantics 
(Theorem~\ref{generalized-models:-old-vs-new}). 
We provide two more specific operations, $i$ and~$I$, 
which turn ultrafilter models in the narrow sense 
into ones in the wide sense, show how they relate 
to the operations $e$ and~$E$ via their limits in 
appropriate topologies (Theorems \ref{e-as-lim-of-i}
and~\ref{E-as-lim-of-I}), and establish necessary and 
sufficient conditions under which ultrafilter models 
in the wide sense are the images of ones in the narrow 
sense under $i$ and~$I$, and also are two canonical 
ultrafilter extensions of some ordinary models 
(Theorems \ref{top-char-of-i-of-old} 
and~\ref{top-char-of-I-of-old}). 
Finally, we define homomorphisms between ultrafilter  
models in the wide sense, and establish for them
an ``abstract extension theorem''  
(Theorem~\ref{abstract-ET-generalized}) and two Second
Extension Theorems (Theorems \ref{SET-generalized} 
and~\ref{SET-generalized-I-version}). In Section~5, 
we conclude the article by posing some problems 
and tasks.

A~part of the results mentioned in Sections~1--3 
was announced in~\cite{Poliakov Saveliev}; here 
we provide complete proofs of all our results.%
\footnote{In~\cite{Poliakov Saveliev}, it was 
erroneously stated that the set of right continuous 
maps forms a~compact Hausdorff space w.r.t.~the 
pointwise convergence topology; actually, the 
intended topology was a~{\it restricted\/} pointwise 
convergence topology, as explained in details below.}


\section[Extending the ultrafilter extension]
{Extending the ultrafilter extension
procedure}

A~purpose of this section is to provide a~uniform
approach to both types of ultrafilter extensions: 
the smaller $\wt{\;\;}\,$-extensions and the larger 
${}^*\,$-extensions. For this, we shall develop 
some ideas and machinery which will lead us in the 
next section to certain structures, called there 
ultrafilter models, generalizing ultrafilter 
extensions of each of the two types.

We shall give an alternative description of the 
${}^*\,$-extension  of relations in terms of the 
basic (cl)open sets and the continuous extension of 
maps. The crucial idea is to consider continuous 
extension of the procedure of ultrafilter extension 
itself, i.e.~a~{\it self-application\/} of the 
procedure. Let us clarify what is the idea precisely. 
For simplicity, consider firstly unary maps, for which 
the ultrafilter extensions are just the continuous 
extensions. To make the notation easier, let us denote 
the operation of continuous extension of maps by~$\ext$; 
i.e. $\ext(f)$~is another notation for~$\wt f$: 
$$
\ext(f)=\wt f.
$$ 
So if we consider (unary) maps of $X$ into~$Y$, then
$\ext$~is a~map of $Y^X$ into $C(\scc X,\scc Y)$, the 
set of all continuous functions of $\scc X$ into~$\scc Y$. 
If $C(\scc X,\scc Y)$ would be endowed with some compact 
Hausdorff topology, then we could extend the map~$\ext$ 
to a~(unique) continuous map~$\wt\ext$ of $\scc(Y^X)$ 
into $C(\scc X,\scc Y)$:
$$
\xymatrix{
&\scc\bigl(Y^X\bigr)\,
\ar@{-->}^{\wt{\ext}\;\;}[drr]&&
\\
&\;Y^X\,
\ar[u]
\ar[rr]^{\,\ext}&&
\;\;C(\scc X,\scc Y)&
}
$$
We are going to show that such a~topology on 
$C(\scc X,\scc Y)$ exists, and in fact, is 
a~weaker version of the pointwise convergence 
topology (while the standard full version of 
the topology is not compact, as explained 
in Remark~\ref{rmk: full pointwise conv}). 
Furthermore, as we shall see, the same 
approach will work in the case of $n$-ary maps 
(and relations, which can be reduced to maps).


\vskip+1em
\noindent 
\textbf{\textit{Restricted pointwise 
convergence topology.}}
Let $X$ and $Y$ be topological spaces and 
$A\subseteq X$. Define a~topology on the set
$Y^X$ of all maps of $X$ into~$Y$ by letting 
the family of sets $O_{a,B}=\{f\in Y^X:f(a)\in B\}$
for all $a\in A$ and all $B\subseteq Y$ which are 
open in~$Y$, as an open subbase. We shall call it 
the $A$-{\it pointwise convergence\/} topology. 
Clearly, if $A=X$ then it is the usual pointwise 
convergence topology, which, as well-known (see 
e.g.~\cite{Engelking}), coincides with the standard 
(Tychonoff) product topology. 

\begin{definition}\label{def: restr pointwise topol}
Consider $Y^{X_1\times\ldots\times X_n}$ 
as the set of $n$-ary maps, and choose subsets
$A_1\subseteq X_1,\ldots,A_n\subseteq X_n$. The 
topology with an open subbase consisting of sets
$$ 
O_{a_1,\ldots,a_n,B}=
\bigl\{
f\in Y^{X_1\times\ldots\times X_n}:
f(a_1,\ldots,a_n)\in B\bigr\}
$$
for all $a_1\in A_1,\ldots,a_n\in A_n$ and all 
$B\subseteq Y$ which are open in~$Y$, will be called 
the $(A_1,\ldots,A_n)$-{\it pointwise convergence\/} 
topology. 
\end{definition}

Although it is the same that the set of unary maps 
of $X_1\times\ldots\times X_n$ into~$Y$ endowed 
with the $A_1\times\ldots\times A_n$-pointwise 
convergence topology, we shall use this terminology 
to emphasize when we shall say about $n$-ary maps.

Let $1\leq i\leq j\leq n$. For any 
$f:X_1\times\ldots\times X_n\to Y$, 
$\boldsymbol a\in X_1\times\ldots\times X_{i-1}$, 
and $\boldsymbol u\in X_{j+1}\times\ldots\times X_{n}$, 
the map
$$
f_{\boldsymbol a}^{\boldsymbol u}:
X_{i}\times\ldots\times X_{j}\to Y
$$ 
is defined by letting 
$$
f_{\boldsymbol a}^{\boldsymbol u}
(x_i,\ldots,x_j)=
f(\boldsymbol a,x_i,\ldots,x_j,\boldsymbol u)
$$
for all $x_i\in X_i,\ldots,x_j\in X_j$. 
We omit the sub- and superscripts whenever the sequences 
$\boldsymbol a$ and $\boldsymbol u$ respectively are empty.

Let $\cur$ be the {\it currying\/} (or {\it evaluation\/}) 
map taking any $f:X_1\times\ldots\times X_n\to Y$ 
with $n\geq 2$ to the map 
$
\cur(f):X_n\to 
Y^{X_1\times\ldots\times X_{n-1}}
$ 
such that 
$$
\cur(f)(x)=f^x.
$$ 
(A~more precise term would be the 
{\it right currying\/} but we prefer the shorter one.) 
Clearly, the map~$\cur$ is bijective.

Let for any positive $n<\omega$,
topological spaces $X_1,\ldots,X_{n},Y$, and sets 
$A_1\subseteq X_1,\ldots,A_{n-1}\subseteq X_{n-1}$, 
$$
RC_{A_1,\ldots,A_{n-1}}(X_1,\ldots,X_n,Y)
$$ 
denote the set of $n$-ary maps 
$f:X_1\times\ldots\times X_n\to Y$ that are 
right continuous w.r.t.~$A_1,\ldots,A_{n-1}$, 
which we consider with the  
$(A_1,\ldots,A_n)$-pointwise convergence 
topology.

\begin{lemma}\label{currying-of-righ-cont}
If $f\in RC_{A_1,\ldots,A_{n-1}}(X_1,\ldots,X_n,Y)$ 
then 
$$
\cur(f)\in 
C\bigl(X_{n},
RC_{A_1,\ldots,A_{n-2}}(X_1,\ldots,X_{n-1},Y)
\bigr).
$$
\end{lemma}

\begin{proof}
By definition of currying, $\cur(f)$ maps $X_{n}$ 
into $RC_{A_1,\ldots,A_{n-2}}(X_1,\ldots,X_{n-1},Y)$.
Let us verify that it is continuous. Pick any
$
\boldsymbol{a}\in 
X_1\times\ldots\times X_{n-1},
$
open set~$B$ in~$Y$, and consider the subbasic open set
$
O_{\boldsymbol{a},B}=
\{h\in 
RC_{A_1,\ldots,A_{n-2}}(X_1,\ldots,X_{n-1},Y): 
h(\boldsymbol{a})\in B
\}
$ 
in the space 
$RC_{A_1,\ldots,A_{n-1}}(X_1,\ldots,X_n,Y)$.
We have:
$$
\cur(f)(u)\in O_{\boldsymbol{a},B}
\;\;\text{iff}\;\;
\cur(f)(u)(\boldsymbol{a})\in B
\;\;\text{iff}\;\;
f_{\boldsymbol{a}}(u)\in B.
$$
The set
$
\cur(f)^{-1}(O_{\boldsymbol{a},B})=
(f_{\boldsymbol{a}})^{-1}(B)
$ 
is open since the map~$f_{\boldsymbol{a}}$ is continuous. 
\end{proof}

\begin{lemma}\label{uniqueness-of-right-cont} 
For any positive $n<\omega$, topological spaces
$X_1,\ldots,X_{n}$, their dense subsets
$D_1\subseteq X_1,\ldots,D_n\subseteq X_n$, 
and Hausdorff space~$Y$,
\begin{enumerate}
\item[(i)] 
if maps 
$
f,g\in 
RC_{D_1,\ldots,D_{n-1}}(X_1,\ldots,X_n,Y)
$
coincide on $D_1\times\ldots\times D_n$, 
then they coincide everywhere,
\item[(ii)] 
the space $RC_{D_1,\ldots,D_{n-1}}(X_1,\ldots,X_n,Y)$ 
endowed with the $(D_1,\ldots,D_{n-1})$-pointwise 
convergence topology, is Hausdorff. 
\end{enumerate}
\end{lemma}

\begin{proof}
For brevity, let us denote the space 
$RC_{D_1,\ldots,D_{k-1}}(X_1\ldots,X_{k},Y)$ 
by~$\RC_k$. We argue by induction on~$n$. 
For induction basis, see \cite{Engelking}, 
Theorem~2.1.9. Assume we have already proven 
the claim for $n=k$. Let us prove this for $n=k+1$. 

Let maps $f,g\in\RC_{k+1}$ coincide on 
$D_1\times\ldots\times D_{k+1}$. 
Then for each $a\in D_{k+1}$ the maps $f^a$ and $g^a$ 
coincide on $D_1\times\ldots\times D_{k}$ 
and are right continuous w.r.t.~the~$D_i$.
By induction hypothesis, $f^a=g^a$. Hence, the maps 
$\cur(f),\cur(g):X_{k+1}\to\RC_{k}$ coincide on~$D_{k+1}$. 
By Lemma~\ref{currying-of-righ-cont}, the maps 
are continuous, while by induction hypothesis the
space~$\RC_{k}$ is Hausdorff. Hence, $\cur(f)=\cur(g)$ 
again by \cite{Engelking}, Theorem~2.1.9. 
Therefore, $f=g$ since $\cur$~is bijective.

Furthermore, this shows that the space~$\RC_{k+1}$ is 
Hausdorff. Indeed, let $f,g\in\RC_{k+1}$ and $f\neq g$. 
Then, by the just proven fact,
$f(\boldsymbol{a})\neq g(\boldsymbol{a})$ for some
$\boldsymbol{a}\in D_1\times\ldots\times D_{k+1}$. 
Since $Y$~is Hausdorff, pick any disjoint open 
neighborhoods $A,B\subseteq Y$ of $f(\boldsymbol{a})$ 
and $g(\boldsymbol{a})$. Then the sets 
$F=\{h\in\RC_{k+1}:h(\boldsymbol{a})\in A\}$ and 
$G=\{h\in\RC_{k+1}:h(\boldsymbol{a})\in B\}$ are 
disjoint open neighborhoods of $f\in F$ and $g\in G$. 
\end{proof}


\begin{lemma}\label{partial-pointwise-convergence}
Let $X_1,\ldots,X_n$ be discrete spaces and 
$Y$~a~compact Hausdorff space. The set 
$$
RC_{X_1,\ldots,X_{n-1}}(\scc X_1,\ldots,\scc X_n,Y)
$$ 
of $n$-ary maps of $\scc X_1\times\ldots\times\scc X_n$ 
into~$Y$ which are right continuous 
w.r.t.~$X_1,\ldots,X_{n-1}$, endowed with the 
$(X_1,\ldots,X_n)$-pointwise convergence topology,
is homeomorphic to the space 
$Y^{X_1\times\ldots\times X_n}$ endowed with 
the usual pointwise convergence topology. 
Therefore, the space is compact Hausdorff; 
moreover, it is zero-dimensional iff so is~$Y$. 
\end{lemma}

\begin{proof}
Let us verify that the map~$\ext$, which takes 
each $n$-ary~$f$ in $Y^{X_1\times\ldots\times X_n}$ 
to its extension $\ext(f)=\wt{f}$ in 
$RC_{X_1,\ldots,X_{n-1}}(\scc X_1,\ldots,\scc X_n,Y)$, 
is a~homeomorphism. The fact that $\ext$~is injective 
is trivial, and that $\ext$~is surjective follows 
from Lemma~\ref{uniqueness-of-right-cont} (since 
each $X_i$ is dense in~$\scc X_i$ and $Y$~is Hausdorff): 
whenever 
$g\in RC_{X_1,\ldots,X_{n-1}}(\scc X_1,\ldots,\scc X_n,Y)$ 
and $f=g\uhr(X_1\times\ldots\times X_n)$, then $\wt{f}=g$. 
Finally, the fact that it preserves in both directions 
open sets belonging to the subbases of the spaces, is 
immediate by the definition of the 
$(X_1,\ldots,X_n)$-pointwise convergence topology. 
Therefore, the space is homeomorphic to the usual product
space of~$Y$, hence, by the Tychonoff theorem, is compact 
Hausdorff, and moreover, the zero-dimensionality 
iff so is~$Y$ (see e.g.~\cite{Engelking}).
\end{proof}

\begin{lemma}\label{extensions-are-dense} 
Let $X_1,\ldots,X_n$ be discrete spaces,
$Y$~a~compact Hausdorff space, and $S\subseteq Y$ 
dense in~$Y$. Then the set 
$
\bigl\{
\wt{f}:f\in S^{X_1\times\ldots\times X_n}
\bigr\}
$ 
is dense in the space 
$RC_{X_1,\ldots,X_{n-1}}(\scc X_1,\ldots,\scc X_n,Y)$
endowed with the $(X_1,\ldots,X_n)$-pointwise 
convergence topology. 
\end{lemma}

\begin{proof}
Let $k<\omega$, pick for all $i<k$ arbitrary 
$\boldsymbol{a}_i\in X_1\times\ldots\times X_n$ 
and $B_i\subseteq Y$ open in~$Y$, and 
show that, whenever the basic open set 
$\bigcap_{i<k}O_{\boldsymbol{a}_i,B_i}$
is nonempty, then it contains a~point~$\wt{f}$ 
for some $f\in S^{X_1\times\ldots\times X_n}$.
Note that if some of the~$\boldsymbol{a}_i$ 
coincide, say, $\boldsymbol{a}_i=\boldsymbol{a}_j$ 
for all $i,j\in A$ and some $A\subseteq k$, then  
$\bigcap_{i\in A}B_i$ is nonempty whenever so is 
$\bigcap_{i<k}O_{\boldsymbol{a}_i,B_i}$. So we can 
assume w.l.g.~that all the~$\boldsymbol{a}_i$ are 
distinct. Then any $f\in S^{X_1\times\ldots\times X_n}$
satisfying $f(\boldsymbol{a}_i)\in B_i\cap S$ 
for all $i<k$, is as required. 
\end{proof}

\begin{question}\label{q: extensions-are-dense}
Does this remain true, moreover, for 
the full pointwise convergence topology? 
The answer is affirmative for unary maps, i.e.~the 
set $\{\wt{f}:f\in S^X\}$ is dense in $C(\scc X,Y)$. 
What happens for binary maps? 
(Problem~\ref{prb: 1}.)
\end{question}


\begin{remark}\label{rmk: full pointwise conv} 
One may ask whether the usual (unrestricted) 
pointwise convergence topology on the set 
$RC_{X_1,\ldots,X_{n-1}}(\scc X_1,\ldots,\scc X_n,Y)$ 
is compact, or equivalently, whether the set forms 
a~closed subspace of the compact Hausdorff space 
$Y^{\scc X_1\times\ldots\times\scc X_n}$ with the 
Tychonoff product topology. If~this would be the 
case, we could use this more common topology for 
our purpose. Let us show that the answer is in the
negative, even for unary maps.

1. 
The set $C(\scc X,\scc Y)$ endowed with the pointwise 
convergence topology is not compact.

It suffices to verify that for an arbitrary 
map $h:\scc X\to\scc Y$ there exists 
an ultrafilter~$\mathfrak f$ over $C(\scc X,\scc Y)$ 
converging to~$h$ (to recall related facts the reader 
can look at the beginning of Section~4). Since $h$~can 
be discontinuous, this will show that $C(\scc X,\scc Y)$ 
is not closed in~$\scc Y^{\scc X}$. Consider the family 
\begin{align*}
\mathcal{F}
&=
\bigl\{
O_{\mathfrak{u},\,\wt{S}}:
\mathfrak{u}\in\scc X 
\;\text{and}\;
h(\mathfrak{u})\in\wt S\,
\bigr\}
\\
&=
\bigl\{
\{f\in C(\scc X,\scc Y):S\in f(\mathfrak u)\}:
\mathfrak{u}\in\scc X 
\;\text{and}\;
S\in h(\mathfrak{u})
\bigr\}.
\end{align*}
Let us check that $\mathcal{F}$~is centered. 
It suffices to show that for any positive $n\in\omega$, 
ultrafilters $\mathfrak{u}_0,\ldots,\mathfrak{u}_{n-1}$ 
over~$X$, and non-empty subsets $S_0,\ldots,S_{n-1}$ 
of~$X$, there exists a~map $f\in C(\scc X,\scc Y)$ 
satisfying
$$
S_0\in f(\mathfrak{u}_0),  
\ldots, 
S_{n-1}\in f(\mathfrak{u}_{n-1}).
$$
To see, pick arbitrary pairwise disjoint sets 
$A_0,\ldots,A_{n-1}$ such that
$
A_0\in\mathfrak{u}_0,\ldots,
A_{n-1}\in\mathfrak{u}_{n-1},
$ 
elements $s_0\in S_0,\ldots,s_{n-1}\in S_{n-1}$, 
and consider a~map $g: X\to Y$ such that 
$g(x)=s_i$ whenever $x\in A_i$, $i<n$, and 
$g(x)=y$, where $y$~is a~fixed element of~$Y$,
otherwise. (Actually, $g$~on the set
$X\setminus(A_0\cup\ldots\cup A_{n-1})$ 
could be defined arbitrarily.) Let $f=\wt{g}$, 
so $f\in C(\scc X,\scc Y)$. For each $i<n$ we have
$
A_i\subseteq g^{-1}(S_i)
\;\;\text{and}\;\;
A_i\in\mathfrak{u}_i,
$
therefore, $g^{-1}(S_i)\in\mathfrak{u}_i$, and so,
$
S_i\in\wt{g}(\mathfrak{u}_i)=f(\mathfrak{u}_i).
$
Thus the map~$f$ witnesses that
the family~$\mathcal{F}$ is centered.

Now extend $\mathcal{F}$ to an ultrafilter
$\mathfrak{f}\in\scc\,C(\scc X,\scc Y)$. It is 
clear that $\mathfrak{f}$~converges to the map~$h$, 
as required.

2. 
Since we know that $C(\scc X,\scc Y)$ with the 
$X$-pointwise convergence topology is compact while 
with the (full) pointwise convergence topology is not, 
we may ask what is the map $f\in C(\scc X,\scc Y)$ 
such that the ultrafilter~$\mathfrak f$ defined above 
converges to~$f$ in the weaker (restricted) topology. 
It is not difficult to show that $f=\wt{h\uhr X}$. 
Note that $(\scc Y)^{\scc X}$ with the $X$-pointwise 
convergence topology is compact (since it is compact 
even with the stronger pointwise convergence topology). 
The map~$r$ of this compact space onto its compact 
subspace $C(\scc X,\scc Y)$, defined by letting 
for all $h\in(\scc Y)^{\scc X}$
$$
r(h)=\wt{h\uhr X},
$$
is a~natural retraction. 
However, $(\scc Y)^{\scc X}$ with the $X$-pointwise 
convergence topology is not Hausdorff nor even 
a~$T_0$-space since, whenever $h\in(\scc Y)^{\scc X}$ 
is discontinuous, then the points $h$ and $r(h)$ 
are distinct but have the same neighborhoods
(it suffices to consider subbasic neighborhoods, and 
for any $a\in X$ and open $B\subseteq\scc Y$ we have 
$h\in O_{a,B}$ iff $r(h)\in O_{a,B}$).

3. 
These observations hold in a~general setting, for 
$n$-ary maps into any compact Hausdorff space~$Y$: 
the full pointwise convergence topology on 
$RC_{X_1,\ldots,X_{n-1}}(\scc X_1,\ldots,\scc X_n,Y)$
is not compact, while 
the $(X_1,\ldots,X_n)$-pointwise convergence
topology on $Y^{\scc X_1\times\ldots\times\scc X_n}$ is 
compact but not~$T_0$, and the map~$r$ defined by letting 
for all $h\in Y^{\scc X_1\times\ldots\times\scc X_n}$ 
$$
r(h)=
\ext(h\uhr(X_1\times\ldots\times X_{n}))
$$ 
is a~natural retraction of 
$Y^{\scc X_1\times\ldots\times\scc X_n}$ onto 
$RC_{X_1,\ldots,X_{n-1}}(\scc X_1,\ldots,\scc X_n,Y)$. 
\end{remark}


\vskip+1em
\noindent 
\textbf{\textit{Self-application of 
the extension operation.}}
Now we are ready to define the continuous 
extension $\wt{\ext}$ of the map~$\ext$ 
in a~general form. 
Let $X_1,\ldots,X_n$ be discrete spaces and 
$Y$~a~compact Hausdorff space. Recall that 
for any $n$-ary map~$f$ of 
$X_1\times\ldots\times X_n$ into~$Y$, 
$\ext(f)$~is~$\wt f$, the extension of~$f$
to ultrafilters which is right continuous 
w.r.t.~principal ultrafilters:
\begin{align*} 
\ext:
Y^{X_1\times\ldots\times X_n}\to
RC_{X_1,\ldots,X_{n-1}}
(\scc X_1,\ldots,\scc X_n,Y). 
\end{align*}
By Lemma~\ref{partial-pointwise-convergence}, 
the set 
$
RC_{X_1,\ldots,X_{n-1}}
(\scc X_1,\ldots,\scc X_n,Y)
$
endowed with the $(X_1,\ldots,X_n)$-pointwise 
convergence topology is a~compact Hausdorff space. 
Therefore, $\ext$~extends to a~unique continuous 
map~$\wt{\ext}$ on ultrafilters over the set
$Y^{X_1\times\ldots\times X_n}$: 
$$
\xymatrix{
\;\quad\scc\bigl(Y^{X_1\times\ldots\times 
X_n}\bigr)\qquad\ar@{-->}^{\wt{\ext}}[dr]&&
\\
Y^{X_1\times\ldots\times X_n}\,\ar[u]\ar[r]^{\;\;\;\ext}&&%
\!\!\!\!\!\!\!\!\!\!\!\!\!\!\!\!%
RC_{X_1,\ldots,X_{n-1}}
(\scc X_1,\ldots,\scc X_n,Y)
}
$$

\begin{remark}\label{rmk: def of ext ext} 
Alternatively, we can first define $\wt{\ext}$ 
on ultrafilters over the set of unary maps and then 
extend it to $\wt{\ext}$ on ultrafilters over the set 
of $n$-ary maps by induction on~$n$ by using currying.

For this, we first note that 
the one-to-one correspondence between the sets 
$Y^{X_1\times\ldots\times X_n\times X_{n+1}}$ 
and $(Y^{X_1\times\ldots\times X_n})^{X_{n+1}}$ given
by~$\cur$ induces the one-to-one correspondence between 
the sets of ultrafilters over them, which takes 
each ultrafilter
$
\mathfrak f\in
\scc(Y^{X_1\times\ldots\times X_n\times X_{n+1}})
$ 
to the ultrafilter 
$
\mathfrak f\,'=
\{\cur\image A:A\in\mathfrak f\}\in
\scc((Y^{X_1\times\ldots\times X_{n}})^{X_{n+1}}).
$
Or else, $\mathfrak f\,'$~can be defined 
via the continuous extension of currying:
$$
\xymatrix{
\scc\bigl(
Y^{X_1\times\ldots\times X_{n+1}}
\bigr)
\;\;
\ar@{-->}^{\wt\cur\;\;\;\;\;\;\;\;}[r]
&\;\;
\scc\bigl(
\bigl(Y^{X_1\times\ldots\times X_n}\bigr)^{X_{n+1}}
\bigr)
\\
\ar[u] 
Y^{X_1\times\ldots\times X_{n+1}}
\;\;
\ar[r]^{\cur\;\;\;\;\;\;\;\;} 
&\;\;
\bigl(Y^{X_1\times\ldots\times X_n}\bigr)^{X_{n+1}}
\ar[u]. 
}
$$ 
Since $\cur$~is a~bijection, it is easily follows 
that so is~$\wt{\cur}$, and for all 
$
\mathfrak f\in
\scc(Y^{X_1\times\ldots\times X_n\times X_{n+1}})
$ 
we have
$
\wt{\cur}(\mathfrak f)=\mathfrak f\,'.
$

Now, for $n=1$, we extend $\ext:Y^X\to C(\scc X,Y)$ 
to $\wt{\ext}:\scc(Y^X)\to C(\scc X,Y)$. And assuming 
that $\wt{\ext}$~has been already defined for~$n$, 
we can define $\wt{\ext}(\mathfrak f)$ by letting
$$
\wt{\ext}(\mathfrak f)
(\mathfrak u_1,\ldots,\mathfrak u_n,\mathfrak u_{n+1})
=
\wt{\ext}\bigl(\wt{\ext}(\mathfrak f\,')
(\mathfrak u_{n+1})\bigr)
(\mathfrak u_1,\ldots,\mathfrak u_{n})
$$ 
since $\wt{\ext}$~has been already defined 
on $\mathfrak f\,'$ and 
$\wt{\ext}(\mathfrak f\,')(\mathfrak u_{n+1})$ 
by induction hypothesis.
\end{remark}


\begin{question}\label{q: wider ext ext}
One can offer another, alternative way to extend 
the ultrafilter extension procedure by considering 
it as the map not into the space of right continuous 
maps but into set of all maps with the usual product 
topology. Thus for any discrete $X_1,\ldots,X_n$ and 
compact Hausdorff~$Y$, let $\ext$~be a~map of the 
discrete space $Y^{X_1\times\ldots\times X_n}$ into 
$Y^{\scc X_1\times\ldots\times\scc X_n}$ endowed with 
the usual product topology (or equivalently, the usual,
unrestricted pointwise convergence topology). As the 
range is a~compact Hausdorff space, the map~$\ext$ 
continuously extends to~$\wt\ext$ (in the new sense):
$$
\xymatrix{
\;\quad\scc\bigl(Y^{X_1\times\ldots\times 
X_n}\bigr)\qquad\ar@{-->}^{\wt{\ext}}[dr]&&
\\
Y^{X_1\times\ldots\times X_n}\,\ar[u]\ar[r]^{\;\;\;\ext}&&%
\!\!\!\!\!\!\!\!\!\!\!\!\!\!\!\!%
Y^{\,\scc X_1\times\ldots\times\scc X_n}
}
$$
Unlike the previous construction, now some ultrafilters 
$\mathfrak f\in\scc(Y^{X_1\times\ldots\times X_n})$ 
are mapped into maps 
$
\wt\ext(\mathfrak f)\in 
Y^{\,\scc X_1\times\ldots\times\scc X_n}
$
that no longer are right continuous w.r.t.~principal 
ultrafilters (as explained in the remarks above). 
However, these maps $\wt\ext(\mathfrak f)$ are still 
close to those: any neighborhood of 
$\wt\ext(\mathfrak f)$ contains some right continuous 
map; this is because 
$$
\wt{\ext}\image\,
\scc\bigl(Y^{X_1\times\ldots\times X_n}\bigr)
=
\cl_{Y^{\,\scc X_1\times\ldots\times\scc X_n}}
\bigl(
\ext\image\,Y^{X_1\times\ldots\times X_n}
\bigr).
$$
Is this version of~$\wt\ext$ surjective?
This is the case iff the image of~$\ext$ 
is dense in the space; see 
Question~\ref{q: extensions-are-dense}.

Can this version of $\wt\ext$ lead to some 
interesting possibilities?  
(Problem~\ref{prb: 2}.)
\end{question}


\begin{lemma}\label{ext-ext-surj-non-inj}
For any positive $n<\omega$, discrete spaces 
$X_1,\ldots,X_n$, and compact Hausdorff space~$Y$, 
the continuous map 
$$
\wt{\ext}:\scc(Y^{X_1\times\ldots\times X_n})
\to 
RC_{X_1,\ldots,X_{n-1}}(\scc X_1,\ldots,\scc X_n,Y)
$$ 
is surjective and, whenever at least one of the~$X_i$
is infinite, non-injective.
\end{lemma}

\begin{proof}
To simplify the notation, let $\RC$~denote the space
$RC_{X_1,\ldots,X_{n-1}}(\scc X_1,\ldots,\scc X_n,Y)$
endowed with the $(X_1,\ldots,X_n)$-pointwise 
convergence topology. Pick any $f\in\RC$, let 
$$
Z=
\bigl\{
(\boldsymbol{a},B):
\boldsymbol{a}\in X_1\times\ldots\times X_n\,,\; 
B\text{ is open in~}Y,
\text{ and }
f(\boldsymbol{a})\in B 
\bigr\},
$$
and consider the following family~$\mathcal F$ 
of subsets of $Y^{X_1\times\ldots\times X_n}$:
$$
\mathcal F=
\bigl\{
\bigl\{
g\in Y^{X_1\times\ldots\times X_n}: 
g(\boldsymbol{a})\in B
\bigr\}:
(\boldsymbol{a},B)\in Z
\bigr\}.
$$
The family~$\mathcal F$ is centered; this can be 
stated by arguments similar to those in the first 
remark after Lemma~\ref{extensions-are-dense}. 
We are going to prove the following key property 
of the family~$\mathcal F$: 
$$
\wt\ext(\mathfrak f)=f 
\;\;\text{for all }
\mathfrak f\in
\scc\bigl(Y^{X_1\times\ldots\times X_n}\bigr)
\text{ such that }\mathcal F\subseteq\mathfrak f.
$$
The lemma will be deduced from this property: 
since the argument works for all $f\in\RC$, 
the property shows that $\wt\ext$~is surjective; 
and that $\wt\ext$~is non-injective will be shown 
once two distinct such ultrafilters 
$\mathfrak f\,',\mathfrak f\,''$ will be constructed.

Let us verify the following equality:
$$
\bigcap_{A\in\mathcal F}
\cl_{\RC}\bigl\{\wt g:g\in A\bigr\}=
\{f\}.
$$
First note that by Lemma~\ref{extensions-are-dense}, 
for every 
$
A=
\{g\in Y^{X_1\times\ldots\times X_n}:
g(\boldsymbol{a})\in B\}
$ 
in~$\mathcal F$ we have
$$
\cl_{\RC}\{\wt g:g\in A\}=
\cl_{\RC}\{h\in\RC:h(\boldsymbol{a})\in B\}.
$$
Therefore,
$$
f\in\bigcap_{A\in\mathcal F}
\cl_{\RC}\{\wt g:g\in A\}.
$$
Next, toward a~contradiction, assume that there 
exists $f\,'\in\RC$ such that $f\,'\neq f$ and 
$
f\,'\in\bigcap_{A\in\mathcal F}
\cl_{\RC}\{\wt g:g\in A\}.
$
By Lemma~\ref{uniqueness-of-right-cont}, 
there exists 
$\boldsymbol{b}\in X_1\times\ldots\times X_n$ 
such that 
$f(\boldsymbol{b})\neq f\,'(\boldsymbol{b})$. 
As $Y$~is Hausdorff, pick disjoint open 
neighborhoods $U$ and~$U\,'$ of the points
$f(\boldsymbol{b})$ and $f\,'(\boldsymbol{b})$, 
respectively. We have:
\begin{align*}
\bigcap_{A\in\mathcal F}\cl_{\RC}\{\wt g:g\in A\}
&=
\bigcap_{(\boldsymbol{a},B)\in Z}
\cl_{\RC}
\{h\in\RC:h(\boldsymbol{a})\in B\}
\\
&\subseteq
\cl_{\RC}
\bigl\{
h\in\RC:h(\boldsymbol{b})\in U
\bigr\}
\\
&\subseteq 
\cl_{\RC}
\bigl\{
h\in\RC:h(\boldsymbol{b})\in Y\setminus U\,'
\bigr\}
=
\bigl\{
h\in\RC:h(\boldsymbol{b})\in Y\setminus U\,'
\bigr\},
\end{align*} 
(where the last equality holds since the set 
$\{h\in\RC:h(\boldsymbol{b})\in Y\setminus U\,'\}$
is the complement in~$\RC$ of the subbasic open set 
$
\{h\in\RC:h(\boldsymbol{b})\in U\,'\}=
O_{\boldsymbol{b},U\,'}
$). 
Therefore,
$$
f\,'\notin
\bigcap_{A\in\mathcal F}
\cl_{\RC}\{\wt g:g\in A\},
$$ 
a~contradiction. Thus we have verified that 
the equality is true.

\par 
Now the required key property of the 
family~$\mathcal F$, i.e.~that we have 
$\wt{\ext}(\mathfrak f)=f$ whenever 
$\mathfrak f\supseteq\mathcal F,$
is clearly follows from this equality. 
As observed above, this property immediately implies 
that $\wt{\ext}$~is surjective; and to show that 
$\wt{\ext}$~is also non-injective, it remains to 
construct two distinct ultrafilters 
$\mathfrak f\,',\mathfrak f\,''\supseteq\mathcal F$.

Pick a~family 
$
\{
B_{\boldsymbol{a}}:
\boldsymbol{a}
\in X_1\times\ldots\times X_n
\}
$ 
of subsets of~$Y$ such that 
$B_{\boldsymbol{a}}\neq Y$ and 
$f(\boldsymbol{a})\in B_{\boldsymbol{a}}$ for all 
$\boldsymbol{a}\in X_1\times\ldots\times X_n$. 
The families 
\begin{align*}
\mathcal F\,'
&=
\mathcal F\cup 
\bigl\{\bigl\{
g\in Y^{X_1\times\ldots\times X_n}:
(\forall\boldsymbol{a}
\in X_1\times\ldots\times X_n)\;
g(\boldsymbol{a})\in B_{\boldsymbol{a}}
\bigr\}\bigr\},
\\
\mathcal F\,''
&=
\mathcal F\cup
\bigl\{\bigl\{
g\in Y^{X_1\times\ldots\times X_n}:
(\exists\boldsymbol{a}
\in X_1\times\ldots\times X_n)\;
g(\boldsymbol{a})\notin B_{\boldsymbol{a}}
\bigr\}\bigr\}
\end{align*}
are both centered (the fact that $\mathcal F\,''$~is 
centered uses that one of the~$X_i$ is infinite). 
We extend them to two (automatically distinct) 
ultrafilters $\mathfrak f\,'$ and~$\mathfrak f\,''$, 
respectively. By the key property of~$\mathcal F$, 
we obtain
$$
\wt{\ext}(\mathfrak f\,')=
\wt{\ext}(\mathfrak f\,'')=
f,
$$
thus showing that $\wt\ext$~is not injective. 
Note also that, since $f\in\RC$ was choosen 
arbitrary, we have established a~bit more: 
the preimage of {\it each\/} point in~$\RC$ under 
the map~$\wt\ext$ consists of more than one point.

The lemma is proved. 
\end{proof}


\begin{lemma}\label{ext-ext}
Let $X_1,\ldots,X_n$ be discrete spaces, 
$Y$~a~compact Hausdorff space, $S\subseteq Y$, 
and\, $R\subseteq S^{X_1\times\ldots\times X_n}$. 
Then\, $\wt\ext$ maps the closure of~$R$ 
in the space $\scc(S^{X_1\times\ldots\times X_n})$ 
onto the closure of~$R$ in the space 
$
RC_{X_1,\ldots,X_{n-1}}
(\scc X_1,\ldots,\scc X_n,Y)
$ 
endowed with the $(X_1,\ldots,X_n)$-pointwise 
convergence topology:
$$
\wt{\ext}\image
\cl_{\scc(S^{X_1\times\ldots\times X_n})}R
=
\cl_{RC_{X_1,\ldots,X_{n-1}}
(\scc X_1,\ldots,\scc X_n,Y)}\,
\ext\image R.
$$
\end{lemma}

\begin{proof} 
Again, to simplify notation, we temporarily let: 
\begin{align*}
Z&=
\scc\bigl(S^{X_1\times\ldots\times X_n}\bigr),
\\
\RC&=
RC_{X_1,\ldots,X_{n-1}}
(\scc X_1,\ldots,\scc X_n,Y).
\end{align*} 
We consider $Z$ with the standard topology on the space 
of ultrafilters, so it actually does not depend on the 
topology on $S$ as a~subspace of~$Y$. The fact that 
$Y$~is compact Hausdorff is used only to get the same 
properties of the topology on~$\RC$, which are 
essential to extend $\ext$ to~$\wt{\ext}$.

To prove the inclusion
$$
\wt{\ext}\image\cl_{Z}\,R
\subseteq
\cl_{\RC}\,\ext\image R,
$$
recall that by the general definition of continuous 
extensions of unary maps, for any $\mathfrak f\in Z$ 
we have 
$
\{\wt{\ext}(\mathfrak f)\}=
\bigcap_{A\in\mathfrak f}
\cl_{\RC}\,\ext\image A.
$
Therefore, 
\begin{align*}
\wt{\ext}\image
\cl_{Z}R
=
\bigl\{
\wt{\ext}(\mathfrak f):
\mathfrak f\in
\cl_{Z}R
\bigr\}
=
\bigcup_{\mathfrak f\in
\cl_{Z}R}\;
\bigcap_{A\in\mathfrak f}\;
\cl_{\RC}\,\ext\image A.
\end{align*}
But for any $\mathfrak f\in\cl_{Z}R$ we have 
$R\in\mathfrak f$ and hence 
$
\bigcap_{A\in\mathfrak f}
\cl_{\RC}\,\ext\image A
\subseteq
\cl_{\RC}\,\ext\image R,
$
whence it follows 
$$
\bigcup_{\mathfrak f\in
\cl_{Z}R}\;
\bigcap_{A\in\mathfrak f}\;
\cl_{\RC}\,\ext\image A
\subseteq 
\cl_{\RC}\,\ext\image R,
$$ 
which gives the required inclusion.

To prove the converse inclusion
$$
\cl_{\RC}\,\ext\image R
\subseteq
\wt{\ext}\image\cl_{Z}R, 
$$
note that 
$
\cl_{\RC}\,\wt{\ext}\image R
\subseteq
\wt{\ext}\image\cl_{Z}R 
$
since the map $\wt{\ext}:Z\to RC$ is closed as
a~continuous map of a~compact space into a~Hausdorff 
space (see e.g.~\cite{Engelking}, Corollary~3.1.11),
and that $\wt{\ext}\image R=\ext\image R$ since
$R$~consists of principal ultrafilters over the set 
$S^{X_1\times\ldots\times X_n}$ (under our customary 
identification of elements with principal ultrafilters 
given by them).

The lemma is proved. 
\end{proof}


Now we are ready to give the promised alternative 
description of the ${}^*\,$-extension of relations. 
For simplicity, we formulate it only for the case
when $X_1=\ldots=X_n=X$; nevertheless, this formulation 
does not lose generality since for given~$X_i$ we can 
take their union as such an~$X$.

\begin{theorem}\label{modal-via-ext}
Let\, $R\subseteq X\times\ldots\times X$ be any 
$n$-ary relation on a~set~$X$. Then its extension 
$R^*\subseteq\scc X\times\ldots\times\scc X$ is 
(identified with) the image under~$\wt\ext$ of 
the basic set~$\wt R$ in the space~$\scc(X^n)$ 
where $R$~is considered as a~unary relation on~$X^n$:  
\begin{align*}
R^*=    
\bigl\{\wt{\ext}(\mathfrak r):R\in\mathfrak r\bigr\}=
\wt{\ext}\image\wt{R}.
\end{align*}
\end{theorem}

\begin{proof}
By Theorem~\ref{modaltopology}, 
$R^*=\cl_{\scc X\times\ldots\times\scc X}R$. As usual, 
the product space $\scc X\times\ldots\times\scc X$
($n$~times) is identified with $(\scc X)^n$, 
so up to this identification we can let 
$$
R^*=\cl_{(\scc X)^n}R.
$$
We are going to use Lemma~\ref{ext-ext} by choosing 
appropriate discrete $X_1,\ldots,X_m$, a~compact 
Hausdorff~$Y$, and $S\subseteq Y$. Let $m=1$, 
let the space~$X_1$ be~$n$ with the discrete topology, 
so $\scc X_1=X_1=n$, let 
the space~$Y$ be~$\scc X$ with the standard topology 
on the space of ultrafilters, and let $S$ be~$X$, so 
we have: 
$$
\scc\bigl(S^{X_1}\bigr)=\scc(X^n) 
\;\;\text{and}\;\; 
C(X_1,Y)=(\scc X)^n
$$ 
(clearly, the $X_1$-pointwise 
convergence topology on the latter set is the same 
that the full pointwise convergence topology). 
Now Lemma~\ref{ext-ext} gives us 
$$
\wt{\ext}\image\cl_{\scc(X^n)}R
=
\cl_{(\scc X)^n}\,\ext\image R.
$$ 
But $\cl_{\scc(X^n)}R=\wt{R}$ where $R$~is considered 
as a~unary relation on~$X^n$ (recall that if $Z$~is 
discrete and $A\subseteq Z$, then the basic open 
set~$\wt{A}$ equals the closure $\cl_{Z}A$), 
and furthermore, $\ext\image R=R$ (since $\wt f=f$ 
for all $f\in Y^n$ as $\scc\,n=n$). Putting all this 
together, we obtain:
$$
R^*=
\cl_{(\scc X)^n}R=
\cl_{(\scc X)^n}\,\ext\image R=
\wt{\ext}\image\cl_{\scc(X^n)}R=
\wt{\ext}\image\wt{R}, 
$$ 
as required. 
\end{proof}

Although this description of the $^{*}$-extension 
of relations is not simpler than one given by 
Theorem~\ref{modaltopology}, it provides some 
connection of this larger extension with 
the smaller $\wt{\;\;}$-extension of relations 
(by using also continuous extensions of maps). 
Other interrelations between the $\wt{\;\;}$- and  
$^{*}$-extensions of relations are established via 
Vietoris-type topologies in~\cite{Saveliev(2 concepts)}.


\section{Ultrafilter interpretations}

In this section, we define our main concepts: 
ultrafilter interpretations (of functional and 
relational symbols) and ultrafilter models 
(involving ultrafilter evaluations and ultrafilter
interpretations) together with their semantics. 
Then we provide two specific operations turning 
ultrafilter models into ordinary ones, establish 
necessary and sufficient conditions under which 
the latter are two canonical ultrafilter extensions 
of some ordinary models, and give a~topological
characterization of ultrafilter models. Finally, 
we define homomorphisms of ultrafilter models and 
prove for them a~version of the First Extension 
Theorem and an its refinement.

\vskip+1em
\noindent
\textbf{\textit{Ultrafilter models.}}
Using ultrafilters over maps in our previous 
considerations leads us to the following concept.

\begin{definition}\label{def: ultra model (narrow)}
Given a~signature~$\tau$, we define 
an {\it ultrafilter interpretation\/} as 
a~map~$\imath$ that takes each $n$-ary functional 
symbol~$F$ in~$\tau$ to an ultrafilter over the set 
of $n$-ary operations on~$X$, and each $n$-ary 
predicate symbol~$R$ in~$\tau$ to an ultrafilter 
over the set of $n$-ary relations on~$X$; 
let also $v$~be an {\it ultrafilter valuation\/} 
of variables, i.e.~a~valuation which takes each 
variable~$x$ to an ultrafilter over a~given set~$X$:
\begin{align*}
v(x)\in\scc X,
\quad
\imath(F)\in\scc\bigl(X^{X\times\ldots\times X}\bigr),
\quad
\imath(R)\in\scc\,\mathcal{P}(X\times\ldots\times X).
\end{align*}
We refer to the structure  
$$(\scc X,\imath(F),\ldots,\imath(R),\ldots)$$ 
as an {\it ultrafilter model\/} of~$\tau$.%
\footnote{
Ultrafilter models were introduced 
in~\cite{Poliakov Saveliev} under the name 
of generalized models. 
In Section~4, we shall introduce 
a~wider concept of ultrafilter models 
(Definition~\ref{def: ultra model (wide)}); 
the ultrafilter models defined here will be 
referred to as those ``in the narrow sense''.
}
\end{definition}

Now we are going to define an appropriate 
satisfiability relation between ultrafilter models 
and first-order formulas, which we shall denote 
by the symbol~$\VDash$\,.

First, given an interpretation~$\imath$ of non-logical 
symbols, we expand any valuation~$v$ of variables to 
the map~$v_\imath$ defined on all terms as follows. 
Let 
$
\app:
X_1\times\ldots\times X_n\times 
Y^{X_1\times\ldots\times X_n}
\to Y
$ 
be the {\it application\/} operation:
$$
\app(a_1,\ldots,a_n,f)=f(a_1,\ldots,a_n).
$$ 
Extend it to the map 
$
\wt\app:
\scc X_1\times\ldots\times\scc X_n\times
\scc(Y^{X_1\times\ldots\times X_n})\to\scc Y
$ 
right continuous w.r.t.~the principal ultrafilters, 
in the usual way:
$$
\xymatrix{\scc X_1\times\ldots\times\scc X_n\times
\scc(Y^{X_1\times\ldots\times X_n})
\!\!\!\!\!\!\!\!\!\!\!\!\!\!\!\!\! 
&\ar@{-->}^{\wt\app\;\;}[r]
&\;\scc Y
\\
\ar[u] X_1\times\ldots\times X_n\times 
Y^{X_1\times\ldots\times X_n}
\!\!\!\!\!\!\!\!\!\!\!\!\!\!\!\!\! 
&\ar[r]^{\app\;\;} &\;Y\ar[u]. }
$$ 
Let $v_\imath$~coincide with~$v$ on variables, 
and if $v_\imath$~has been already defined on 
terms $t_1,\ldots,t_n$, we let
$$
v_\imath(F(t_1,\ldots,t_n))=
\wt\app(v_\imath(t_1),\ldots,v_\imath(t_n),\imath(F)).
$$

\begin{remark}\label{rmk: ext-app}
We can consider, more generally, for any compact 
Hausdorff space~$Y$ the extension
$
\wt\app:
\scc X_1\times\ldots\times\scc X_n\times
\scc(Y^{X_1\times\ldots\times X_n})\to Y
$ 
right continuous w.r.t.~the principal ultrafilters:
$$
\xymatrix{
\scc X_1\times\ldots\times\scc X_n\times
\scc(Y^{X_1\times\ldots\times X_n})
\qquad\qquad\ar@{-->}^{\qquad\;\wt\app}[dr]&&
\!\!\!\!\!\!\!\!\!\!\!\!\!\!\! 
\\
\;\;\;{X_1\times\ldots\times X_n\times 
Y^{X_1\times\ldots\times X_n}}\,
\ar[u]\ar[r]^{\qquad\qquad\app}& 
Y
}
$$
though this is redundant for our immediate purposes.
\end{remark}

Further, given an ultrafilter model 
$
\mathfrak A=
(\scc X,\imath(F),\ldots,\imath(R),\ldots),
$ 
define the satisfiability in~$\mathfrak A$ as follows. 
Let 
$
\inn\subseteq 
X_1\times\ldots\times X_n\times
\mathcal{P}(X_1\times\ldots\times X_n)
$ 
be the {\it membership\/} predicate:
$$
\inn(a_1,\ldots,a_n,R)
\;\;\text{iff}\;\;
(a_1,\ldots,a_n)\in R.
$$
Extend it to the relation 
$
\wt{\inn}\subseteq 
\scc X_1\times\ldots\times\scc X_n\times
\scc\,\mathcal{P}(X_1\times\ldots\times X_n)
$ 
right clopen w.r.t.~principal ultrafilters.

\begin{definition}\label{def: sat narrow}
The {\it satisfiability} of a~formula~$\varphi$ 
in~$\mathfrak A$ is defined by induction on 
the construction of~$\varphi$:
If $t_1=t_2$ is an identity, we let
$$ 
\mathfrak A\VDash
t_1=t_2\;[v]
\;\;\text{iff}\;\;
v_\imath(t_1)=v_\imath(t_2).
$$
If $R(t_1,\ldots,t_n)$ is an atomic formula 
in which $R$~is not the equality predicate, we let
$$
\mathfrak A\VDash
R(t_1,\ldots,t_n)\;[v]
\;\;\text{iff}\;\;
\wt\inn(v_\imath(t_1),\ldots,v_\imath(t_n),\imath(P)). 
$$
(Equivalently, we could define the satisfiability 
of atomic formulas by identifying predicates with 
their characteristic functions and using the 
satisfiability of equalities of the resulting terms.)
Finally, if $\varphi(t_1,\ldots,t_n)$ is obtained 
by negation, conjunction, or quantification from 
formulas for which $\VDash$~has been already defined,
we define $\mathfrak A\VDash\varphi\,[v]$ 
in the standard way. 
\end{definition}

When needed, we shall use variants of notation 
commonly used for ordinary models and satisfiability, 
for our generalized variants. E.g.~for an ultrafilter  
model~$\mathfrak{A}$ with the universe~$\scc X$, 
a~formula $\varphi(x_1,\ldots,x_n)$, and elements
$\mathfrak u_1,\ldots,\mathfrak u_n$ of~$\scc X$, 
the notation 
$
\mathfrak{A}\VDash\varphi\,
[\mathfrak u_1,\ldots,\mathfrak u_n]
$ means that 
$\varphi$~is satisfied in~$\mathfrak A$ under 
a~valuation taking the variables $x_1,\ldots,x_n$ to 
the ultrafilters $\mathfrak u_1,\ldots,\mathfrak u_n$, 
respectively.


Ultrafilter models actually generalize not all 
ordinary models but those that are ultrafilter 
extensions of some models. It is worth also pointing 
out that whenever an ultrafilter interpretation is 
{\it principal}, i.e.~all non-logical symbols are
interpreted by principal ultrafilters, we naturally 
identify it with the obvious ordinary interpretation 
with the same universe~$\scc X$; however, not every
ordinary interpretation with the universe~$\scc X$ 
is of this form. Precise relationships between 
ultrafilter models, ordinary models, and ultrafilter
extensions will be described in Theorems 
\ref{e-and-E-as-ultraextensions} 
and~\ref{e-and-E-topology}. 
Let us also note in advance that ultrafilter models 
in the wide sense, which we shall define in Section~4,  
will cover (up to some natural identification) not 
ultrafilter extensions only but all ordinary models.


An ultrafilter valuation~$v$ is {\it principal\/} 
iff it takes any variable to a~principal ultrafilter.

\begin{lemma}\label{equiv}
Let\, 
$
\mathfrak{A}=
(\scc X,\imath(F),\ldots,\imath(R),\ldots)
$ 
and 
$
\mathfrak{B}=
(\scc X,\jmath(F),\ldots,\jmath(R),\ldots)
$ 
be two ultrafilter models of the same signature 
and having the same universe~$\scc X$.
If for all functional symbols~$F$, predicate symbols~$R$, 
variables $x_1,\ldots,x_n$, and principal valuations~$v$, 
\begin{align*}
\wt{\app}(v(x_1),\ldots,v(x_n),\imath(F))
&=
\wt{\app}(v(x_1),\ldots,v(x_n),\jmath(F)),
\\
\wt{\inn}(v(x_1),\ldots,v(x_n),\imath(R))
\;\,&\text{iff}\;\;\,
\wt{\inn}(v(x_1),\ldots,v(x_n),\jmath(R)),
\end{align*}
then for all formulas~$\varphi$, terms $t_1,\ldots,t_n$, 
and valuations~$v$,
$$
\mathfrak{A}\VDash\varphi(t_1,\ldots,t_n)\,[v]
\;\;\text{iff}\;\;
\mathfrak{B}\VDash\varphi(t_1,\ldots,t_n)\,[v].
$$
\end{lemma}

\begin{proof} 
By induction on construction of formulas using 
the right continuity of~$\wt{\app}$ and the right 
clopenness of~$\wt{\inn}$ w.r.t.~$X$. 
\end{proof}

\begin{cor}\label{principalpred}
Given an ultrafilter model\, 
$
\mathfrak{A}=
(\scc X,\imath(F),\ldots,\imath(R),\ldots),
$ 
define an ultrafilter model 
$
\mathfrak{B}=
(\scc X,\jmath(F),\ldots,\jmath(R),\ldots)
$ 
of the same signature as follows: 
let $\mathfrak B$ have the same universe~$\scc X$, 
let $\jmath$~coincide with~$\imath$ on functional 
symbols, and for each predicate symbol~$R$ let 
$\jmath(R)$~be the principal ultrafilter 
given by the relation 
$$
\bigl\{
(a_1,\ldots,a_n)\in X^n:
\wt{\inn}(a_1,\ldots, a_n,\imath(R))
\bigr\}.
$$
Then for all valuations~$v$, formulas~$\varphi$, 
and terms $t_1,\ldots,t_n$, 
$$
\mathfrak{A}\VDash\varphi(t_1,\ldots,t_n)\,[v]
\;\;\text{iff}\;\;
\mathfrak{B}\VDash\varphi(t_1,\ldots,t_n)\,[v]. 
$$
\end{cor}

\begin{proof}
Lemma~\ref{equiv}. 
\end{proof}


\begin{definition}\label{def: pseudo-principal uf}
If $X_1,\ldots,X_n,Y$ are discrete spaces, 
let us say that an ultrafilter~$\mathfrak f$ 
over the set $Y^{X_1\times\ldots\times X_n}$ of 
$n$-ary maps is {\it pseudo-principal\/} iff 
$\wt\app$~takes any $n$-tuple consisting of 
principal ultrafilters together with~$\mathfrak f$ 
to a~principal ultrafilter: 
\begin{align*}
a_1\in X_1,\ldots,a_n\in X_n
\;\;\text{implies}\;\; 
\wt\app(a_1,\ldots,a_n,\mathfrak f)\in Y.
\end{align*}
\end{definition}

Clearly, if the space~$Y$ is finite, then all  
ultrafilters in $\scc(Y^{X_1\times\ldots\times X_n})$
are pseudo-principal. 
(More generally, if we would defined $\wt{\app}$ 
with the range in any compact Hausdorff~$Y$, as 
proposed in the remark above, then all ultrafilters 
in $\scc(Y^{X_1\times\ldots\times X_n})$ were 
pseudo-principal.)

\begin{lemma}\label{pseudoprinc} 
Let $X_1,\ldots,X_n,Y$ be discrete spaces.
In $\scc(Y^{X_1\times\ldots\times X_n})$, every 
principal ultrafilter is pseudo-principal, and 
if $Y$ and at least one of the~$X_i$ are infinite, 
then there exist pseudo-principal ultrafilters that 
are not principal as well as ultrafilters that are 
not pseudo-principal. 
\end{lemma}

\begin{proof}
Pick any $f\in Y^{X_1\times\ldots\times X_n}$. 
Let $\mathcal F$~be the following family of subsets 
of the space $Y^{X_1\times\ldots\times X_n}$:
$$
\mathcal F=
\bigl\{
\bigl\{g\in Y^{X_1\times\ldots\times X_n}:
g(\boldsymbol{a})=f(\boldsymbol{a})\bigr\}:
\boldsymbol{a}\in X_1\times\ldots\times X_n
\bigr\}
\cup
\bigl\{
\bigl\{g\in Y^{X_1\times\ldots\times X_n}: 
g\neq f\bigr\}
\bigr\}.
$$
The family~$\mathcal F$ is centered (as at 
least one of the~$X_i$ is infinite), so pick 
any ultrafilter~$\mathfrak f$ over the set
$Y^{X_1\times\ldots\times X_n}$ such that 
$\mathcal F\subseteq\mathfrak f$.  
Since $\bigcap\mathcal F$~is empty, 
the ultrafilter~$\mathfrak f$ is non-principal.  
On the other hand, for every 
$
\boldsymbol{a}=(a_1,\ldots, a_n)\in 
X_1\times\ldots\times X_n
$ 
we have:
\begin{align*}
S\in\wt{\app}(a_1,\ldots,a_n,\mathfrak f)
&\;\;\text{iff}\;\;
(\forall^{\,a_1}x_1)
\ldots
(\forall^{\,a_n}x_n)
(\forall^{\,\mathfrak f}g)\;
\app(x_1,\ldots,x_n,g)\in S
\\
&\;\;\text{iff}\;\;
(\forall^{\,a_1}x_1)
\ldots
(\forall^{\,a_n}x_n)
(\forall^{\,\mathfrak f}g)\;
g(x_1,\ldots,x_n)\in S
\\
&\;\;\text{iff}\;\;
(\forall^{\,\mathfrak f}g)\;
g(a_1,\ldots,a_n)\in S
\\
&\;\;\text{iff}\;\;
(\exists F\in\mathfrak f)
(\forall g\in F)\;
g(a_1,\ldots,a_n)\in S.
\end{align*}
(The first equivalence follows from the definition 
of extensions of maps via ultrafilter quantifiers,
the second holds by the definition of~$\app$, 
the third since $a_1,\ldots,a_n$ are principal, 
and the fourth decodes the definition of 
the $\forall^{\,\mathfrak f}$~quantifier.)
Letting $S=\{f(a_1,\ldots,a_n)\}$, we have 
$
\wt{\app}(a_1,\ldots,a_n,\mathfrak f)=
f(a_1,\ldots,a_n)\in Y, 
$ 
thus witnessing that $\mathfrak f$~is 
pseudo-principal.

\par 
To construct a~non-pseudo-principal ultrafilter, 
pick any 
$
\boldsymbol{a}=(a_1,\ldots,a_n)
\in X_1\times\ldots\times X_n
$ 
and $\mathfrak u\in\scc Y\setminus Y$ (as 
$Y$~is infinite), and expand the centered family
\begin{align*}
\mathcal G
&=
\bigl\{
\bigl\{f\in Y^{X_1\times\ldots\times X_n}: 
f(\boldsymbol{a})\in S\bigr\}:S\in\mathfrak u
\bigr\}
\\
&=
\bigl\{
O_{\boldsymbol{a},\,\wt{S}}:S\in\mathfrak u
\bigr\}
\end{align*}
to an ultrafilter $\mathfrak g\supseteq\mathcal G$ 
over $Y^{X_1\times\ldots\times X_n}$. Calculations 
similar to those in the above give us 
$$
\wt{\app}(a_1,\ldots,a_n,\mathfrak g)=\mathfrak u,
$$
thus witnessing that $\mathfrak g$~is not 
pseudo-principal.
\end{proof}


\begin{definition}\label{def: pseudo-principal int}
An ultrafilter interpretation~$\imath$ is 
{\it pseudo-principal on functional symbols\/} 
iff $\imath(F)$~is a~pseudo-principal ultrafilter 
for each functional symbol~$F$ 
(and then, for each term~$t$).
\end{definition}

\begin{cor}\label{principalfunc}
Given an ultrafilter model
$
\mathfrak{A}=
(\scc X,\imath(F),\ldots,\imath(R),\ldots)
$ 
with $\imath$~pseudo-principal on functional symbols, 
define an ultrafilter model 
$
\mathfrak{B}=
(\scc X,\jmath(F),\ldots,\jmath(R),\ldots)
$ 
of the same signature as follows: 
let $\mathfrak B$ have the same universe~$\scc X$, 
let $\jmath$~coincide with~$\imath$ on predicate symbols, 
and for each functional symbol~$F$ let $\jmath(F)$~be 
the principal ultrafilter given by the operation 
$f:X^n\to X$ defined by letting 
$$
f(a_1,\ldots,a_n)=\wt{\app}(a_1,\ldots,a_n,\imath(F)).
$$
Then for all valuations~$v$, formulas~$\varphi$, 
and terms $t_1,\ldots,t_n$, 
$$
\mathfrak{A}\VDash\varphi(t_1,\ldots,t_n)\,[v]
\;\;\text{iff}\;\;
\mathfrak{B}\VDash\varphi(t_1,\ldots,t_n)\,[v]. 
$$
\end{cor}

\begin{proof}
Lemma~\ref{equiv}.
\end{proof}

It follows that for every ultrafilter 
model~$\mathfrak{A}$ whose interpretation 
is pseudo-principal on functional symbols, 
by replacing its relations 
as in Corollary~\ref{principalpred} and 
its operations as in Corollary~\ref{principalfunc}, 
one obtains an ordinary model~$\mathfrak{B}$ 
with the same universe such that 
for all formulas~$\varphi$ and elements 
$\mathfrak u_1,\ldots,\mathfrak u_n$ of the universe, 
$
\mathfrak{A}\VDash\varphi\,
[\mathfrak u_1,\ldots,\mathfrak u_n]
$
iff 
$
\mathfrak{B}\vDash\varphi\,
[\mathfrak u_1,\ldots,\mathfrak u_n]. 
$

We do not formulate this fact as a~separate theorem 
since we shall be able to establish stronger facts 
soon. In Theorem~\ref{e-preserves-formulas}, we shall 
establish that for any ultrafilter model~$\mathfrak A$, 
not necessarily with a~pseudo-principal interpretation, 
one can construct a~certain ordinary model 
$e(\mathfrak A)$ satisfying the same formulas; 
and then, in Theorem~\ref{e-and-E-as-ultraextensions}, 
that whenever $\mathfrak A$~has a~pseudo-principal 
interpretation, $e(\mathfrak A)$~is nothing but 
the $\wt{\;\;}$-extension of some model. In fact, 
in the latter case, $e(\mathfrak A)$~coincides 
with $\mathfrak B$ from the previous paragraph.

%
%



\vskip+1em
\noindent
\textbf{\textit{Operations $e$ and~$E$.}}
Let us now define two operations, $e$ and~$E$, 
which turn ultrafilter models into certain ordinary 
models that (as we shall see soon) generalize 
the ${}^*\,$- and $\,\wt{\;}\;$-extensions. 
Both operations take ultrafilters over $n$-ary maps 
to $n$-ary maps over ultrafilters, and ultrafilters 
over $n$-ary relations to $n$-ary relations over 
ultrafilters. Both operations are surjective and 
non-injective (Lemma~\ref{e,E-surj-non-inj}).

The map~$e$ on ultrafilters over maps will be the
map~$\wt{\ext}$ defined and discussed in Section~2. 
Now we extend $\wt{\ext}$ to ultrafilters over
relations by identifying $n$-ary relations 
with their $n$-ary characteristic functions 
into the discrete space~$2=\{0,1\}$: 
$$
\xymatrix{
\;\quad\scc
\mathcal P(X_1\times\ldots\times X_n)
\qquad\ar@{-->}^{\;\;\;\;\wt{\ext}}[dr]&&
\\
\mathcal P(X_1\times\ldots\times X_n)
\,\ar[u]\ar[r]^{\;\;\;\;\;\;\;\;\ext}&&%
\!\!\!\!\!\!\!\!\!\!\!\!\!\!\!\!%
\{
Q\subseteq\scc X_1\times\ldots\times\scc X_n:
Q\text{ is right clopen w.r.t.~}
X_1,\ldots,X_{n-1}
\}
}
$$
(Recall that by Theorem~\ref{topological-char}(i),
a~subset of $\scc X_1\times\ldots\times\scc X_n$ 
is right clopen w.r.t.~$X_1,\ldots,X_{n-1}$ iff it 
is of form~$\wt{R}$ for some $n$-ary subset~$R$ of 
$X_1\times\ldots\times X_n$.) Let the map~$e$ on
ultrafilters over relations also coincide with 
the map~$\wt{\ext}$ on them. So in result we have:
\begin{align*}
e\image
\scc\bigl(Y^{X_1\times\ldots\times X_n}\bigr)
&\subseteq
\scc Y^{\,\scc X_1\times\ldots\times\scc X_n},
\\
e\image
\scc\,\mathcal{P}(X_1\times\ldots\times X_n)
&\subseteq
\mathcal{P}(\scc X_1\times\ldots\times\scc X_n).
\end{align*}

We observe that $e$ and $\wt\app$ (or $\wt\inn$) 
are expressed via each other:

\begin{lemma}\label{appviaext} 
Let $X_1,\ldots,X_n,Y$ be discrete spaces. 
For all\, 
$
\mathfrak f\in
\scc(Y^{X_1\times\ldots\times X_n}), 
$
$
\mathfrak r\in
\scc\,\mathcal{P}(X_1\times\ldots\times X_n),
$ 
and\, 
$
\mathfrak u_1\in
\scc X_1,\ldots,\mathfrak u_n\in\scc X_n,
$
\begin{align*}
e(\mathfrak f)(\mathfrak u_1,\ldots,\mathfrak u_n)
&=
\wt\app(\mathfrak u_1,\ldots,\mathfrak u_n,\mathfrak f),
\\
e(\mathfrak r)(\mathfrak u_1,\ldots,\mathfrak u_n)
&\;\;\text{iff}\;\;\,
\wt\inn
(\mathfrak u_1,\ldots,\mathfrak u_n,\mathfrak r).
\end{align*}
\end{lemma}

In other words,
\begin{align*}
e(\mathfrak f)
&=
\bigl\{
(\mathfrak u_1,\ldots,\mathfrak u_n,\mathfrak v)\in
\scc X_1\times\ldots\times\scc X_n\times\scc Y:
\wt{\app}(\mathfrak u_1,\ldots,\mathfrak u_n,\mathfrak f)
=\mathfrak v
\bigr\},
\\
e(\mathfrak r)
&=
\bigl\{
(\mathfrak u_1,\ldots,\mathfrak u_n)\in
\scc X_1\times\ldots\times\scc X_n:
\wt{\inn}(\mathfrak u_1,\ldots,\mathfrak u_n,\mathfrak r)
\bigr\}.
\end{align*}


\begin{proof} 
To simplify the notation, let $\RC$ be the space 
$
RC_{X_1,\ldots,X_{n-1}}
(\scc X_1,\ldots,\scc X_n,\scc Y)
$ 
of $n$-ary maps on $\scc X_1\times\ldots\times\scc X_n$ 
into~$\scc Y$ that are right continuous 
w.r.t.~$X_1,\ldots,X_{n-1}$, endowed with 
the $(X_1,\ldots,X_n)$-pointwise convergence topology. 
By Lemma~\ref{partial-pointwise-convergence}, 
$\RC$~is compact Hausdorff. Recall that for any 
$\mathfrak f\in\scc(Y^{X_1\times\ldots\times X_n})$ 
we have 
$$
e(\mathfrak f)=
\wt{\ext}(\mathfrak f)=
g\in\RC
\;\;\text{such that}\;\;
\{g\}=
\bigcap_{A\in\mathfrak f}\cl_{\RC}
\bigl\{\wt{f}:f\in A\bigr\}, 
$$
and that $\wt{\app}^{\,\mathfrak f}$ is the $n$-ary map 
on $\scc X_1\times\ldots\times\scc X_n$ into~$\scc Y$ 
defining by letting 
$
\wt{\app}^{\,\mathfrak f}
(\mathfrak u_1,\ldots,\mathfrak u_n)=
\wt{\app}
(\mathfrak u_1,\ldots,\mathfrak u_n,\mathfrak f)
$
for all 
$
\mathfrak u_1\in\scc X_1,
\ldots,
\mathfrak u_n\in\scc X_n.
$

Note that both maps $\wt{\app}^{\,\mathfrak f}$ 
and $e(\mathfrak f)$ are in~$\RC$ (the first follows 
from the fact that $\wt{\app}$ is right continuous 
w.r.t.~$X_1,\ldots,X_{n-1}$, the second holds 
since $\wt{\ext}$ is a~map into~$\RC$). 
Therefore, by Lemma~\ref{uniqueness-of-right-cont}, 
in order to show that they coincide, it suffices to 
verify that they coincide on principal ultrafilters.

For this, pick any $a_1\in X_1,\ldots,a_n\in X_n$ 
and $S\subseteq Y$. We have: 
\begin{align*} 
S\in
\wt{\app}^{\,\mathfrak f}(a_1,\ldots,a_n)
&\;\;\text{iff}\;\;
(\forall^{\,a_1}x_1)
\ldots
(\forall^{\,a_n}x_n)
(\forall^{\,\mathfrak f}f)\;
\app(x_1,\ldots,x_n,f)\in S
\\
&\;\;\text{iff}\;\;
(\forall^{\,\mathfrak f}f)\;
\app(a_1,\ldots,a_n,f)\in S
\\
&\;\;\text{iff}\;\;
(\forall^{\,\mathfrak f}f)\;
f(a_1,\ldots,a_n)\in S 
\\
&\;\;\text{iff}\;\;
\bigl\{
f\in Y^{X_1\times\ldots\times X_n}:
f(a_1,\ldots,a_n)\in S 
\bigr\}
\in\mathfrak f.
\end{align*}
(The first equivalence follows from the definition 
of extensions of maps via ultrafilter quantifiers,
the second holds since $a_1,\ldots,a_n$ are principal, 
the third by the definition of~$\app$, 
and the fourth by the definition of 
the $\forall^{\,\mathfrak f}$~quantifier.)
Therefore, 
$$
e(\mathfrak f)\in
\cl_{\RC}\,
\bigl\{
\wt{f}\in\RC:
f\in Y^{X_1\times\ldots\times X_n}
\text{ and }
f(a_1,\ldots,a_n)\in S
\bigr\}.
$$ 
As stated in Lemma~\ref{partial-pointwise-convergence}, 
the space~$\RC$ is zero-dimensional; in particular,
the open set 
$
O_{a_1,\ldots,a_n,\wt{S}}=
\bigl\{
g\in\RC:g(a_1,\ldots,a_n)\in\wt{S}\,
\bigr\}
$ 
is closed (since its complement
$
\RC\,\setminus O_{a_1,\ldots,a_n,\,\wt{S}}=
O_{a_1,\ldots,a_n,\,\scc Y\setminus\wt{S}}
$ 
is open too). It follows that 
$$
\cl_{\RC}\,
\bigl\{
\wt{f}\in\RC:
f\in Y^{X_1\times\ldots\times X_n}
\text{ and }
f(a_1,\ldots,a_n)\in S
\bigr\}
\subseteq 
O_{a_1,\ldots,a_n,\,\wt{S}}\,.
$$
Therefore, we obtain
$$
e(\mathfrak f)\in
O_{a_1,\ldots,a_n,\,\wt{S}}\,,
$$
or, in other words,
$
e(\mathfrak f)(a_1,\ldots,a_n)\in\wt{S}.
$
The latter is clearly equivalent to 
$S\in e(\mathfrak f)(a_1,\ldots,a_n).$ 
Thus we get the inclusion
$
\wt{\app}(a_1,\ldots,a_n,\mathfrak f)
\subseteq 
e(\mathfrak f)(a_1,\ldots,a_n).
$
But since both 
$\wt{\app}(a_1,\ldots,a_n,\mathfrak f)$ 
and $e(\mathfrak f)(a_1,\ldots,a_n)$ are 
{\it ultra\/}filters, 
the inclusion actually gives the equality 
$
\wt{\app}(a_1,\ldots,a_n,\mathfrak f)
=e(\mathfrak f)(a_1,\ldots,a_n).
$

This proves the lemma for ultrafilters over sets 
of maps. The remaining claim about ultrafilters 
over sets of relations follows by replacing the 
relations with their characteristic functions. 
\end{proof}

\begin{question}\label{q: appviaext for compact}
For which compact Hausdorff space~$Y$, instead of 
$\scc Y$ with a~discrete~$Y$, does Lemma~\ref{appviaext} 
remain true (providing that $\wt\app$ is defined 
as in Remark~\ref{rmk: ext-app})? Does this hold 
at least for all zero-dimensional, or all 
extremally disconnected compact Hausdorff~$Y$? 
(Problem~\ref{prb: 3}.)
\end{question}


\begin{cor}\label{pseudo-principal-as-preimage} 
Let $X_1,\ldots,X_n,Y$ be discrete spaces. 
The set of pseudo-principal ultrafilters 
is the preimage of the set\, 
$\{\wt{f}:f\in Y^{X_1\times\ldots\times X_n}\}$ 
under the map~$e$:
$$
\bigl\{
\,\mathfrak{f}\in
\scc\bigl(Y^{X_1\times\ldots\times X_n}\bigr):
\mathfrak{f}\text{ is pseudo-principal}\,
\bigr\}
=
e^{-1}\,
\bigl\{
\wt{f}:f\in Y^{X_1\times\ldots\times X_n}
\bigr\}. 
$$ 
\end{cor}

Recalling that $e=\wt\ext$, that 
on the set $Y^{X_1\times\ldots\times X_n}$
(identified with principal ultrafilters) 
$\wt\ext$~is just~$\ext$, and that 
$
e\image Y^{X_1\times\ldots\times X_n}=
\{\wt{f}:f\in Y^{X_1\times\ldots\times X_n}\},
$ 
we can rewrite the set of pseudo-principal
ultrafilters also by
$$
e^{-1}\,e\image\, 
Y^{X_1\times\ldots\times X_n}
=
\wt{\ext}{\,}^{-1}\,\ext\image\, 
Y^{X_1\times \ldots\times X_n}.
$$

\begin{proof}
Show first that if $\mathfrak f$~is pseudo-principal,
then $e(\mathfrak f)=\wt{f}$ for some 
$f\in Y^{X_1\times\ldots\times X_n}$. 
By the definition of~$e$ ($=\wt\ext$), always 
$e(\mathfrak f)$~is a~map belonging to the set 
$
RC_{X_1,\ldots,X_{n-1}}
(\scc X_1,\ldots,\scc X_n,Y).
$ 
Since by Lemma~\ref{appviaext} we have 
$
e(\mathfrak f)
(\mathfrak u_1,\ldots,\mathfrak u_n)=
\wt\app
(\mathfrak u_1,\ldots,\mathfrak u_n,\mathfrak f),
$
we see that the map~$e$ takes principal ultrafilters 
to principal ultrafilters whenever $\mathfrak f$~is 
pseudo-principal. But then it follows from 
Lemma~\ref{uniqueness-of-right-cont} that 
$e(\mathfrak f)$ coincides with $\wt{f}$ 
if the map~$f$ is the restriction of 
$e(\mathfrak f)$ to principal ultrafilters:
$$
e(\mathfrak f)=\wt{f}
\;\;\text{for}\;\;
f=e(\mathfrak f)\uhr(X_1\times\ldots\times X_n).
$$ 
It remains to show the converse implication, 
i.e.~that for every~$\wt{f}$ there exists 
a~pseudo-principal ultrafilter~$\mathfrak f$ 
with $e(\mathfrak f)=\wt{f}$. For this, 
it clearly suffices to let $\mathfrak f$ equal 
to the principal ultrafilter given by~$f$. 
\end{proof}

\begin{question}\label{q: topol of pseudo-princ}
What are topological properties of the set 
of pseudo-principal ultrafilters in the space
$\scc(Y^{X_1\times\ldots\times X_n})$? topological 
properties of its preimage under~$e$, the set 
$\{\wt{f}:f\in Y^{X_1\times\ldots\times X_n}\}$,
in the space
$
RC_{X_1,\ldots,X_{n-1}}
(\scc X_1,\ldots,\scc X_n,\scc Y)
$ 
with the $(X_1,\ldots,X_n)$-pointwise convergence 
topology (besides the fact that it is dense there, 
as stated in Lemma~\ref{extensions-are-dense}), 
or with the (usual) pointwise convergence topology? 
in the space
$(\scc Y)^{\,\scc X_1\times\ldots\times\scc X_n}$ 
with the pointwise convergence topology?

Let us point out that objects naturally defined 
in terms of ultrafilter extensions often have 
rather hardly definable topological properties, 
cf.~\cite{Hindman Strauss topol-properties,
Hindman Strauss non-Borel}.
(Problem~\ref{prb: 4}.)
\end{question}


\begin{cor}\label{corappviaext}
For all ultrafilter models 
$
\mathfrak A=
(\scc X,\imath(F),\ldots,\imath(R),\ldots)
$
and valuations~$v$, 
\begin{align*}
v_\imath(F(t_1,\ldots,t_n))
&=
e(\imath(F))(v_\imath(t_1),\ldots,v_\imath(t_n)),
\\
\mathfrak A\VDash
R(t_1,\ldots,t_n)\;[v]
\;\;&\text{iff}\;\;\,
e(\imath(R))(v_\imath(t_1),\ldots,v_\imath(t_n)).
\end{align*}
\end{cor}

\begin{proof} 
Lemma~\ref{appviaext} with 
$X_1=\ldots=X_n=Y=X$.
\end{proof}



\begin{definition}\label{def: e-of-model}
For an ultrafilter model 
$
\mathfrak B=
(\scc X,\mathfrak f,\ldots,\mathfrak r,\ldots),
$
let 
$$
e(\mathfrak B)=
(\scc X,e(\mathfrak f),\ldots,e(\mathfrak r),\ldots).
$$
\end{definition}

Note that $e(\mathfrak B)$~is an ordinary model.

The following theorem is the first of the three main 
results of this section, it states that in point of 
view of the satisfaction of formulas, any ultrafilter
model~$\mathfrak A$ is not distinguished 
from the ordinary model $e(\mathfrak A)$.

\begin{theorem}\label{e-preserves-formulas}
If $\mathfrak A$ is an ultrafilter model, then 
for all formulas~$\varphi$ and 
elements $\mathfrak u_1,\ldots,\mathfrak u_n$ 
of the universe of~$\mathfrak A$, 
$$
\mathfrak{A}\VDash\varphi\,
[\mathfrak u_1,\ldots,\mathfrak u_n]
\;\;\text{iff}\;\;
e(\mathfrak{A})\vDash\varphi\,
[\mathfrak u_1,\ldots,\mathfrak u_n].
$$
\end{theorem}

\begin{proof}
Induction on~$\varphi$ starting 
from Corollary~\ref{corappviaext}. 
\end{proof}


Now we define the map~$E$, which has the same 
domain that the map~$e$ does and also satisfying 
\begin{align*}
E\image
\scc\bigl(Y^{X_1\times\ldots\times X_n}\bigr)
&\subseteq
\scc Y^{\,\scc X_1\times\ldots\times\scc X_n},
\\
E\image
\scc\,\mathcal{P}(X_1\times\ldots\times X_n)
&\subseteq
\mathcal{P}(\scc X_1\times\ldots\times\scc X_n),
\end{align*}
as follows: $E$ and~$e$ coincide on 
$\scc(Y^{X_1\times\ldots\times X_n})$, and whenever
$
\mathfrak r\in
\scc\,\mathcal{P}(X_1\times\ldots\times X_n)
$
then we define 
\begin{align*}
E(\mathfrak r)
=
\wt\ext\image
\wt\ext(\mathfrak r)
=
\bigl\{
\wt\ext(\mathfrak q):
\mathfrak q\in\wt\ext(\mathfrak r)
\bigr\}
\end{align*} 
where $\mathfrak r$~is considered as an ultrafilter 
over unary relations on $X_1\times\ldots\times X_n$
while $\mathfrak q$~is considered as an ultrafilter 
over unary maps on~$n$ (and $\wt\ext$~has 
the corresponding meaning). Let us now explain 
the construction in details.

First, we consider
$\mathcal P(X_1\times\ldots\times X_n)$ as the set of 
{\it unary\/} relations on $X_1\times\ldots\times X_n$.
Then the map~$\ext$ takes any subset~$R$ of 
$X_1\times\ldots\times X_n$ to the clopen 
subset~$\wt{R}$ of $\scc(X_1\times\ldots\times X_n)$. 
Therefore, the extended map~$\wt\ext$ 
takes any ultrafilter~$\mathfrak r$ over 
$\mathcal{P}(X_1\times\ldots\times X_n)$ to 
some clopen subset $Q=\wt\ext(\mathfrak r)$ 
of $\scc(X_1\times\ldots\times X_n)$:
$$
\xymatrix{
\;\quad\scc\mathcal P(X_1\times\ldots\times X_n)
\qquad\ar@{-->}^{\;\;\;\;\wt{\ext}}[dr]&&
\\
\mathcal P(X_1\times\ldots\times X_n)
\,\ar[u]\ar[r]^{\;\;\;\;\;\;\;\;\ext}&&%
\!\!\!\!\!\!\!\!\!\!\!\!\!\!\!\!%
\{
Q\subseteq\scc(X_1\times\ldots\times X_n):
Q\text{ is clopen}\,
\}
}
$$
Next, we identify the product 
$X_1\times\ldots\times X_n$ 
with the set of unary maps~$f$ from the set~$n$ 
into $\bigcup_iX_i$ satisfying $f(i)\in X_{i+1}$ 
(for all $i<n$). Then the map~$\ext$ takes any 
such~$f$ to the unary continuous map~$\wt{f}$ 
from~$n$ into $\bigcup_i\scc X_i$ satisfying 
$f(i)\in\scc X_{i+1}$, and we identify the set 
of such maps~$\wt{f}$ backwards with the product 
$\scc X_1\times\ldots\times\scc X_n$. 
Therefore, the extended map~$\wt\ext$ 
takes any ultrafilter~$\mathfrak q$ over
$X_1\times\ldots\times X_n$ to some $n$-tuple 
$
(\mathfrak u_1,\ldots,\mathfrak u_n)=
\wt\ext(\mathfrak q)
$
in $\scc X_1\times\ldots\times\scc X_n$:
$$
\xymatrix{
\;\quad\scc(X_1\times\ldots\times X_n)
\qquad\ar@{-->}^{\;\;\;\;\wt{\ext}}[dr]&&
\\
X_1\times\ldots\times X_n
\,\ar[u]\ar[r]^{\;\;\;\;\;\;\;\;\ext}&&%
\!\!\!\!\!\!\!\!\!\!\!\!\!\!\!\!%
\scc X_1\times\ldots\times\scc X_n 
}
$$
(An analogous construction was previously used in 
Theorem~\ref{modal-via-ext}.) In result, the set 
$
Q=\wt\ext(\mathfrak r)\subseteq
\scc(X_1\times\ldots\times X_n)
$ 
is mapped onto the set
$
\wt\ext\image Q=E(\mathfrak r)\subseteq
\scc X_1\times\ldots\times\scc X_n. 
$ 
Since $Q$~is clopen and the map~$\wt\ext$ is closed 
(as a~continuous map between compact Hausdorff spaces), 
the resulting $E(\mathfrak r)$ is a~closed subset 
of the space $\scc X_1\times\ldots\times\scc X_n.$


\begin{lemma}\label{e-E-relations}
Let\, 
$
\mathfrak r\in
\scc\,\mathcal{P}(X_1\times\ldots\times X_n).
$ 
Then 
\begin{align*}
e(\mathfrak r)=\wt R
\;\;\text{and}\;\;
E(\mathfrak r)=R^*
\end{align*}
for 
$
R=
e(\mathfrak r)\cap(X_1\times\ldots\times X_n)=
E(\mathfrak r)\cap(X_1\times\ldots\times X_n)=
\bigcap_{S\in\mathfrak{r}}\bigcup S.
$
Consequently, 
$$
e(\mathfrak r)\subseteq E(\mathfrak r).
$$
\end{lemma}

We can write up this~$R$ more explicitly:
\begin{align*}
R
&= 
\bigl\{
(a_1,\ldots,a_n)\in X_1\times\ldots\times X_n: 
(\forall S\in\mathfrak{r})\,
(\exists Q\in S)\;
Q(a_1,\ldots,a_n)
\bigr\}.
\end{align*}

\begin{proof} 
For $e(\mathfrak r)=\wt R$, apply Lemma~\ref{appviaext}.
For $E(\mathfrak r)=R^*$, note that 
$\mathfrak q\in\wt\ext(\mathfrak r)$ iff 
$\bigcap_{S\in\mathfrak{r}}\bigcup S\in\mathfrak q$. 
\end{proof}



\begin{definition}\label{def: E-of-model}
For an ultrafilter model 
$
\mathfrak B=
(\scc X,\mathfrak f,\ldots,\mathfrak r,\ldots),
$ 
let
\begin{align*}
E(\mathfrak B)=
(\scc X,E(\mathfrak f),\ldots,E(\mathfrak r),\ldots).
\end{align*}
Then $E(\mathfrak B)$, like $e(\mathfrak B)$, 
is an ordinary model. 
\end{definition}


The following easy observation is similar to 
Theorem~\ref{homo-of-model-to-modal}, and moreover, 
it turns out to be that theorem whenever the 
interpretation of~$\mathfrak B$ is pseudo-principal 
on functional symbols, as we shall see after
Theorem~\ref{e-and-E-as-ultraextensions}.

\begin{theorem}\label{homo-of-e-to-E}
For any ultrafilter model~$\mathfrak B$ 
the identity map on its universe is a~homomorphism 
of $e(\mathfrak B)$ onto $E(\mathfrak B)$:
$$
\xymatrix{
&\mathfrak B\ar_{e}[dl]\ar^{E}[dr]&
\\
e(\mathfrak B)\ar^{\id}[rr]&&E(\mathfrak B)
}
$$
\end{theorem}

Therefore, all positive formulas satisfied 
in $e(\mathfrak B)$ are also satisfied 
in~$E(\mathfrak B)$.

\begin{proof}
Immediate from Lemma~\ref{e-E-relations} since 
$\wt{R}\subseteq R^*$ for all relations~$R$. 
\end{proof}

Now we are going to establish two 
remaining main results of this section, 
Theorems \ref{e-and-E-as-ultraextensions} 
and~\ref{e-and-E-topology}. The first of them 
characterizes ultrafilter models such that 
their $e$- and $E$-images are ultraextensions of 
ordinary models, while the second one characterizes 
ordinary models that are the $e$- and $E$-images of
ultrafilter models. Before this we prove two more 
auxiliary lemmas, which actually follow from 
the previously stated facts.

\begin{lemma}\label{pseudoprinc-to-princ}
Let $\mathfrak A$ be an ultrafilter model with 
a~pseudo-principal interpretation of functional 
symbols, and $\mathfrak B$ the ultrafilter model 
with a~principal interpretation of functional 
symbols constructed from $\mathfrak A$ as in
Corollary~\ref{principalfunc}. Then 
$e(\mathfrak A)=e(\mathfrak B)$.
\end{lemma}

\begin{proof}
Let $\imath$ and $\jmath$ be the interpretations 
in $\mathfrak A$ and~$\mathfrak B$, respectively. 
If $F$~is a~functional symbol, then 
the operations $e(\imath(F))$ and $e(\jmath(F))$ 
are right continuous w.r.t.~principal ultrafilters. 
Therefore, by Lemma~\ref{uniqueness-of-right-cont}, 
in order to show that they coincide, it suffices to 
verify that they coincide on principal ultrafilters.

If the symbol~$F$ is $n$-ary, let $\boldsymbol{a}$ 
be any $n$-tuple of principal ultrafilters. We have: 
\begin{align*}
\wt{\app}(\boldsymbol{a},\imath(F))=
\jmath(F)(\boldsymbol{a})=
\app(\boldsymbol{a},\jmath(F))=
\wt{\app}(\boldsymbol{a},\jmath(F))
\end{align*}
(the first equality holds by the definition 
of~$\jmath$ from Corollary~\ref{principalfunc}, 
the second as $\jmath(F)$~is principal, and 
the third as $\wt{\app}$ extends~$\app$). 
By Lemma~\ref{appviaext}, 
$$
e(\imath(F))(\boldsymbol{a})=
\wt{\app}(\boldsymbol{a},\imath(F))
\;\;\text{and}\;\;
e(\jmath(F))(\boldsymbol{a})=
\wt{\app}(\boldsymbol{a},\jmath(F))
$$
(that holds for $n$-tuples of non-principal 
ultrafilters as well). This completes the proof.
\end{proof}

\begin{lemma}\label{e,E-surj-non-inj}
Both operations $e$ and $E$ are surjective and 
non-injective. More precisely, 
\begin{enumerate}
\item[(i)]
$e$ (and $E$) on 
$\scc\bigl(Y^{X_1\times\ldots\times X_n}\bigr)$ 
is a~surjection onto 
$RC_{X_1,\ldots,X_{n-1}}(\scc X_1,\ldots,\scc X_n,Y)$, 
\item[(ii)]
$e$ on $\scc\,\mathcal P(X_1\times\ldots\times X_n)$ 
is a~surjection onto 
$
\bigl\{
\wt{R}\in
\mathcal P(\scc X_1\times\ldots\times\scc X_n):
R\subseteq X_1\times\ldots\times X_n
\bigr\} 
$
$
=
\bigl\{
Q\in\mathcal P(\scc X_1\times\ldots\times\scc X_n):
Q$~is right clopen w.r.t.~$X_1,\ldots,X_n
\bigr\},  
$ 
\item[(iii)]
$E$ on $\scc\,\mathcal P(X_1\times\ldots\times X_n)$ 
is a~surjection onto 
$
\bigl\{
R^*\in
\mathcal P(\scc X_1\times\ldots\times\scc X_n):
R\subseteq X_1\times\ldots\times X_n
\bigr\} 
$
$
=
\bigl\{
Q\in\mathcal P(\scc X_1\times\ldots\times\scc X_n):
Q$~is regular closed\,$
\bigr\},
$
\end{enumerate}
and each of the three maps is not an injection 
whenever at least one of the~$X_i$ is infinite. 
\end{lemma}

\begin{proof} 
Item~(i) is Lemma~\ref{ext-ext-surj-non-inj};  
items (ii) and~(iii) are immediate from 
Lemma~\ref{e-E-relations}; the non-injectivity 
is easy from the cardinality argument since 
both maps $R\mapsto\wt{R}$ and $R\mapsto R^*$ 
are bijections. (Alternatively, (ii)~can be 
obtained from~(i) by replacing relations with 
their characteristic functions.) 
The equalities in (ii) and (iii) were stated 
in Theorem~\ref{topological-char}(i),(ii).
\end{proof}


By Lemma~\ref{e-E-relations}, relations of the model 
$e(\mathfrak B)$ are the $\wt{\;\;\,}$-extensions 
of some relations on~$X$, while relations of the 
model $E(\mathfrak B)$ are the ${}^*\,$-extensions 
of the same relations. Whether the whole models 
$e(\mathfrak B)$ and $E(\mathfrak B)$ are the 
ultrafilter extensions of some models depends only 
on the ultrafilter interpretation of functional 
symbols in~$\mathfrak B$:

\begin{theorem}\label{e-and-E-as-ultraextensions}
Let\, $\mathfrak B$ be an ultrafilter model with 
the universe~$\scc X$. The following are equivalent:
\begin{enumerate}
\item[(i)]
$e(\mathfrak B)=\wt{\,\mathfrak A\,}$ for 
an ordinary model~$\mathfrak A$ with the universe~$X$,
\item[(ii)]
$E(\mathfrak B)=\mathfrak A^*$ for 
an ordinary model~$\mathfrak A$ with the universe~$X$,
\item[(iii)]
the interpretation in~$\mathfrak B$ is 
pseudo-principal on functional symbols.
\end{enumerate} 
Moreover, the model~$\mathfrak A$ in (i) and~(ii) 
is the same.  
$$
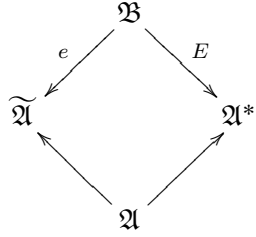
\xymatrix{
&\mathfrak B\ar_{e}[dl]\ar^{E}[dr]&
\\
\wt{\,\mathfrak A\,}
&&\mathfrak A^*
\\
&\mathfrak A\ar[ul]\ar[ur]&
}
$$
\end{theorem}

\begin{proof} 
The implications from each of (i) and (ii) to~(iii) are 
obvious: if the interpretation~$\jmath$ in $\mathfrak B$ 
is not pseudo-principal, then there are a~functional 
symbol~$F$ and a~sequence~$\boldsymbol{a}$ of principal
ultrafilters over~$\scc X$ such that the operation
$G=e(\jmath(F))$ on~$\scc X$ takes $\boldsymbol{a}$ 
to a~non-principal ultrafilter $G(\boldsymbol{a})$ 
over~$X$. Therefore, $G$~is not of form~$\wt{f}$ 
for any operation~$f$ on~$X$. Since $G$~is the 
interpretation of~$F$ in both models $e(\mathfrak B)$ 
and~$E(\mathfrak B)$, it follows that these models are 
not of form $\wt{\,\mathfrak A\,}$ and $\mathfrak A^*$ 
for any ordinary model~$\mathfrak A$.

Let us show now that, conversely, (iii)~implies each of 
(i) and~(ii). By Lemma~\ref{pseudoprinc-to-princ}, it 
suffices to handle the case when the pseudo-principal 
interpretation~$\jmath$ in $\mathfrak B$ is principal. 
So suppose this is the case and define an ordinary 
interpretation~$\imath$ of the same language by letting,
for all functional symbols~$F$ and predicate symbols~$R$, 
\begin{align*}
\imath(F)=G
&\;\text{ if the principal ultrafilter~$\jmath(F)$ 
over $X^{X\times\ldots\times X}$ is given by~$G$},
\\
\imath(R)=Q
&\;\text{ if }Q=e(\jmath(R))\cap(X\times\ldots\times X).
\end{align*}
We have: 
$$
e(\jmath(F))=E(\jmath(F))=\wt{\imath(F)}
$$ 
since $\jmath(F)$~is principal and $e$ (and~$E$) 
on principal ultrafilters is~$\ext$, and 
$$
e(\jmath(R))=\wt{\imath(R)}
\;\;\text{and}\;\;
E(\jmath(R))=(\imath(R))^*
$$ 
by Lemma~\ref{e-E-relations}. 
Thus if $\mathfrak A$~is the ordinary model 
given by~$\imath$, we obtain
$e(\mathfrak B)=\wt{\,\mathfrak A\,}$ and 
$E(\mathfrak B)=\mathfrak A^*$, as required. 
\end{proof}


Finally, we point out that the fact whether 
an ordinary model with the universe~$\scc X$ 
is of form $e(\mathfrak B)$, and whether it is 
of form $E(\mathfrak B)$, for some ultrafilter 
model~$\mathfrak B$ (clearly, with the same 
universe~$\scc X$) depends only on its 
topological properties:

\begin{theorem}\label{e-and-E-topology}
Let\, $\mathfrak A$ be an ordinary model 
with the universe~$\scc X$. Then:
\begin{enumerate}
\item[(i)]
$\mathfrak A=e(\mathfrak B)$ for 
an ultrafilter model~$\mathfrak B$ 
iff in~$\mathfrak A$ all operations are 
right continuous w.r.t.~$X$ and 
all relations are right clopen w.r.t.~$X$,
\item[(ii)]
$\mathfrak A=E(\mathfrak B)$ for 
an ultrafilter model~$\mathfrak B$ 
iff in~$\mathfrak A$ all operations are 
right continuous w.r.t.~$X$ and 
all relations are regular closed. 
\end{enumerate}
\end{theorem}

\begin{proof}
Lemma~\ref{e,E-surj-non-inj}. 
\end{proof}

Since by Theorem~\ref{e-and-E-as-ultraextensions}, 
$e$ and $E$ applied to ultrafilter models with 
pseudo-principal interpretations give the 
$\wt{\;\;}$- and ${}^*$-extensions of ordinary models, 
Theorem~\ref{e-and-E-topology} can be considered 
as a~generalization of Theorems \ref{modeltopology} 
and~\ref{modaltopology}.


\vskip+1em
\noindent
\textbf{\textit{First Extension Theorems.}}
Here we discuss a~possible generalization of the 
First Extension Theorems (Theorems \ref{modalFET} 
and~\ref{modelFET}) to ultrafilter models. To start, 
let us restate both them in a~single way as follows.

\begin{theorem}\label{FET-restated}
Let $\mathfrak A$ and $\mathfrak B$ be two 
(ordinary) models of the same signature, and 
let $h:X\to Y$ be a~map between their universes. 
The following are equivalent: 
\begin{enumerate}
\item[(i)]
$h$~is a~homomorphism of $\mathfrak A$ 
into $\mathfrak B$,
\item[(ii)]
$\wt{h}$~is a~homomorphism of $\wt{\,\mathfrak A\,}$ 
into $\wt{\,\mathfrak B\,}$,
\item[(iii)]
$\wt{h}$~is a~homomorphism of $\mathfrak A^*$ 
into $\mathfrak B^*$:
\end{enumerate} 
$$
\xymatrix{ 
&&\mathfrak A^\ast
\ar@{-->}^{\wt h}[rr] 
&&\mathfrak B^\ast
\\
\wt{\,\mathfrak A\,}
\ar@{-->}^{\wt h}[rr]
&&\wt{\,\mathfrak B\,}&&
\\
&\mathfrak A
\ar[rr]^h\ar[lu]\ar[ruu]
&&\mathfrak B\ar[lu]\ar[ruu]& 
}
$$
\end{theorem}

\begin{proof}
Theorems \ref{modelFET} and~\ref{modalFET}.
\end{proof}

This leads to a~conclusion for our ultrafilter models:

\begin{lemma}\label{homomorphisms-vs-e-and-E}
Let $\mathfrak U$ and $\mathfrak V$ be two 
ultrafilter models of the same signature, and let 
$h:\scc X\to\scc Y$ be a~map between their universes. 
The following are equivalent: 
\begin{enumerate}
\item[(i)]
$h$~is a~homomorphism of $e(\mathfrak U)$ 
into $e(\mathfrak V)$,
\item[(ii)]
$h$~is a~homomorphism of $E(\mathfrak U)$ 
into $E(\mathfrak V)$.
\end{enumerate} 
\end{lemma}

\begin{proof}
If $\mathfrak f$~is an ultrafilter over operations, 
we have $e(\mathfrak f)=E(\mathfrak f)$ by definition 
of $e$ and~$E$, hence the claim for homomorphisms 
w.r.t.~operations holds trivially. 
If $\mathfrak r$~is an ultrafilter over relations, 
we have $e(\mathfrak r)=\wt R$ and $E(\mathfrak r)=R^*$ 
by Lemma~\ref{e-E-relations}, hence the claim for 
homomorphisms w.r.t.~relations holds
by Theorem~\ref{FET-restated}.
\end{proof}

This observation leads to the following definition: 

\begin{definition}\label{def: homo (narrow)}
If $\mathfrak U$ and $\mathfrak V$ are two 
ultrafilter models of the same signature, we say that 
a~map $h:\scc X\to\scc Y$ between their universes is 
a~{\it homomorphism\/} ({\it of ultrafilter models\/}) 
iff it is a~homomorphism of $e(\mathfrak U)$ into 
$e(\mathfrak V)$ (or a~homomorphism of $E(\mathfrak U)$ 
into $E(\mathfrak V)$, which is equivalent by 
Lemma~\ref{homomorphisms-vs-e-and-E}). 
\end{definition}

The concepts of {\it epimorphisms\/}, {\it quotients\/}, 
{\it isomorphic embeddings\/}, {\it submodels\/}, 
{\it elementary embeddings\/}, 
{\it elementary submodels\/}, etc., 
for ultrafilter models are defined likewise.


The following can be considered as the First 
Extension Theorem for ultrafilter models:

\begin{theorem}\label{FET-generalized}
Let $\mathfrak U$ and $\mathfrak V$ be two ultrafilter  
models of the same signature with the universes 
$\scc X$ and~$\scc Y$, both having pseudo-principal 
interpretations on functional symbols, let
$\mathfrak A$ and~$\mathfrak B$ denote 
the models such that 
$\wt{\,\mathfrak A\,}=e(\mathfrak U)$ and 
$\wt{\,\mathfrak B\,}=e(\mathfrak V)$,  
and so $\mathfrak A^*=E(\mathfrak U)$ and 
$\mathfrak B^*=E(\mathfrak V)$, and let $h:X\to Y$. 
The following are equivalent: 
\begin{enumerate}
\item[(i)] 
$h$~is a~homomorphism of $\mathfrak A$ 
into~$\mathfrak B$,
\item[(ii)] 
$\wt{h}$~is a~homomorphism of $\mathfrak U$ 
into~$\mathfrak V$, 
\item[(iii)]
$\wt{h}$~is a~homomorphism of $\wt{\,\mathfrak A\,}$ 
into $\wt{\,\mathfrak B\,}$,
\item[(iv)]
$\wt{h}$~is a~homomorphism of $\mathfrak A^*$ 
into $\mathfrak B^*$:
\end{enumerate}
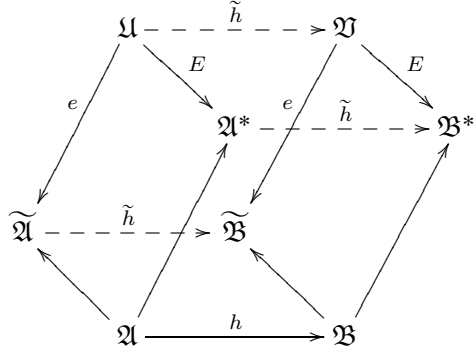
$$
\xymatrix{ 
&\mathfrak U
\ar[ldd]_(.4)e
\ar[rd]^E\ar@{-->}^{\wt h}[rr]
&&\mathfrak V
\ar[ldd]_(.4)e\ar[rd]^E&
\\
&&\mathfrak A^\ast
\ar@{-->}^{\wt h}[rr] 
&&\mathfrak B^\ast
\\
\wt{\,\mathfrak A\,}
\ar@{-->}^{\wt h}[rr]
&&\wt{\,\mathfrak B\,}&& 
\\
&\mathfrak A
\ar[rr]^h\ar[lu]\ar[ruu]
&&\mathfrak B\ar[lu]\ar[ruu]& 
}
$$
\end{theorem}

\begin{proof} 
The equivalence of items (i) and~(ii) follows 
from Theorem~\ref{e-and-E-as-ultraextensions}, 
Lemma~\ref{homomorphisms-vs-e-and-E}, and the 
definition of homomorphisms of ultrafilter models. 
The equivalence of items (i), (iii), and~(iv) 
repeats Theorem~\ref{FET-restated}. 
\end{proof}


For an ultrafilter model~$\mathfrak U$ with 
the universe~$\scc X$, the set~$X$ of principal 
ultrafilters forms an ultrafilter submodel (and 
also ordinary submodels of $e(\mathfrak U)$ 
and~$E(\mathfrak U)$) iff the interpretation 
in~$\mathfrak U$ is pseudo-principal on functional 
symbols; this can be added as item~(iv) to
Theorem~\ref{e-and-E-as-ultraextensions}. 
We shall call the submodel consisting of principal
ultrafilters the {\it principal submodel\/}. Thus 
Theorem~\ref{FET-generalized} can be reformulated 
by replacing ``both having pseudo-principal 
interpretations'' with ``both having principal
submodels''.


In fact, we can omit here the assumption about 
the pseudo-principality in the ultrafilter
model~$\mathfrak V$ by applying the Second 
Extension Theorems (Theorems \ref{modelSET} 
and~\ref{modalSET}):

\begin{theorem}\label{intermediate-ET-generalized}
Let $\mathfrak U$ and $\mathfrak V$ be two 
ultrafilter models of the same signature with 
the universes $\scc X$ and~$\scc Y$, let the 
interpretation of $\mathfrak U$ be pseudo-principal 
on functional symbols with $\mathfrak A$~the 
principal submodel (having the universe~$X$), 
so $\wt{\,\mathfrak A\,}=e(\mathfrak U)$ 
and $\mathfrak A^*=E(\mathfrak U)$, 
and let $h:X\to Y$. 
The following are equivalent: 
\begin{enumerate}
\item[(i)] 
$h$~is a~homomorphism of $\mathfrak A$ 
into $e(\mathfrak V)$, 
\item[(ii)] 
$h$~is a~homomorphism of $\mathfrak A$
into $E(\mathfrak V)$,  
\item[(iii)] 
$\wt{h}$~is a~homomorphism of $\mathfrak U$ 
into~$\mathfrak V$,
\item[(iv)] 
$\wt{h}$~is a~homomorphism of $\wt{\;\mathfrak A\;}$ 
into $e(\mathfrak V)$, 
\item[(v)] 
$\wt{h}$~is a~homomorphism of $\mathfrak A^*$
into $E(\mathfrak V)$:  
\end{enumerate}
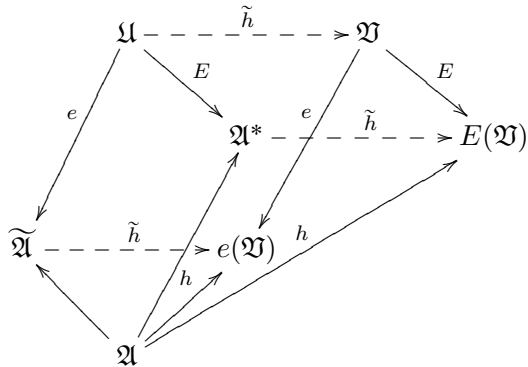
$$
\xymatrix{ 
&\mathfrak U
\ar[ldd]_(.4)e
\ar[rd]^E
\ar@{-->}^{\wt h}[rr]
&&\mathfrak V
\ar[ldd]_(.4)e
\ar[rd]^E&
\\
&&\mathfrak A^\ast
\ar@{-->}^{\wt h}[rr] 
&&E(\mathfrak V)
\\
\wt{\,\mathfrak A\,}
\ar@{-->}^{\wt h}[rr]
&&e(\mathfrak V)&& 
\\
&\mathfrak A
\ar[ru]^(.6){h}
\ar[lu]
\ar[ruu]
\ar[rrruu]^h&&& 
}
$$
\end{theorem}

\begin{proof} 
The equivalence of items (i) and~(iii) follows 
from Theorem~\ref{modelSET} and 
Theorem~\ref{e-and-E-topology}(i), while 
the equivalence of items (ii) and~(iii) follows
from Theorem~\ref{modalSET} and 
Theorem~\ref{e-and-E-topology}(ii). Finally, 
(ii) is equivalent to (iv) by Theorem~\ref{modelSET} 
and to~(v) by Theorem~\ref{modalSET}.  
\end{proof}

Observe that, by Theorem~\ref{e-and-E-topology}, 
whenever the interpretation of $\mathfrak V$ is 
{\it not\/} pseudo-principal on functional symbols, 
then the models $e(\mathfrak V)$ and $E(\mathfrak V)$ 
are {\it not\/} of form $\wt{\,\mathfrak{A}\,}$ and 
$\mathfrak A^*$ for any ordinary model~$\mathfrak A$; 
nevertheless, these models still satisfy the conditions 
of Theorems \ref{modelSET} and~\ref{modalSET} (playing 
the role of the model~$\mathfrak C$ there). Therefore, 
Theorem~\ref{intermediate-ET-generalized} is indeed 
more general than Theorem~\ref{FET-generalized}; 
it has a~character intermediate between the First and 
the Second Extension Theorems. To have a~reasonable
generalization of Second Extension Theorems in the 
full form, we need to have a~more general concept 
of ultrafilter models; this is the subject of the 
next, last section of our article.

\begin{remark}\label{rmk: variants of FET}
Theorems \ref{FET-restated}--%
\ref{FET-generalized} remain true 
for epimorphisms and isomorphic embeddings, 
and Theorem~\ref{intermediate-ET-generalized} 
for epimorphisms. 
Also they can be stated for so-called 
homotopies and isotopies; these concepts 
(generalizing homomorphisms and isomorphisms) 
for ordinary models, together with both extension 
theorems, were introduced in
\cite{Saveliev} (and~\cite{Saveliev(inftyproc)}).
For ultrafilter models they can be defined 
in the same way as this was done for homomorphisms and 
embeddings. Finally, versions for multi-sorted models 
(having rather many universes $X_1,\ldots,X_n$ than 
one universe~$X$) can be also easily stated. 
\end{remark}



\section{Wider ultrafilter interpretations} 

Here we discuss a~possible generalization of the 
Second Extension Theorems (Theorems \ref{modelSET} 
and~\ref{modalSET}) to ultrafilter models. 
For this, we should have a~wider concept 
of ultrafilter models which, on the one hand, 
would replace compact Hausdorff right topological
models in these theorems, and on the other hand, 
would turn into our previous concept of 
ultrafilter models whenever the universe is 
of form~$\scc Y$ to include our versions of 
the First Extension Theorem for ultrafilter 
models (Theorems \ref{FET-generalized} 
and~\ref{intermediate-ET-generalized}).
Also we should have a~concept of satisfiability 
in these models which would turn into 
the satisfiability in our previous ultrafilter 
model; recall that the latter can be redefined 
in terms of the map~$e$ 
(Theorem~\ref{e-preserves-formulas}). 
Actually, our new concepts of ultrafilter models 
and satisfiability will be wide enough to cover 
all ordinary models, not only ultrafilter 
extensions.

\vskip+1em
\noindent
\textbf{\textit{Ultrafilter models in the wide sense.}}
The new definition of ultrafilter models requires 
only a~minor modification of the former one. By 
an {\it ultrafilter\/} 
{\it interpretation\/} we still mean a~map which
takes functional and relational symbols to 
ultrafilters over operations and relations on 
a~set~$X$. But valuations of variables now will 
be in the set~$X$ itself, not in~$\scc X$.

\begin{definition}\label{def: ultra model (wide)}
Given a~signature~$\tau$, an 
{\it ultrafilter model} {\it of~$\tau$}
{\it in the wide sense} 
is a~structure~$\mathfrak U$ of the form
$$
\mathfrak U=
(X,\mathfrak f,\ldots,\mathfrak r,\ldots),
$$ 
where $X$~is the universe of~$\mathfrak U$ 
(so individual variables are valuated by elements 
of~$X$),%
\hide{
if $F$~is an $n$-ary functional symbol in~$\tau$
then it is interpreted by some 
$\mathfrak f\in\scc(X^{X^n})$, and 
if $R$~is an $n$-ary predicate symbol in~$\tau$ 
then it is interpreted by some 
$\mathfrak r\in\scc\,\mathcal P(X^n)$. 
} 
each $n$-ary functional symbol in~$\tau$ is 
interpreted by some $\mathfrak f\in\scc(X^{X^n})$, 
and each $n$-ary predicate symbol in~$\tau$ is 
interpreted by some 
$\mathfrak r\in\scc\,\mathcal P(X^n)$. 
\end{definition}

Ultrafilter models in the sense of  
Definition~\ref{def: ultra model (narrow)} 
will be referred to as {\it ultrafilter models\/} 
{\it in the narrow sense\/}. Henceforth we shall 
never omit the words ``in the narrow sense''.

To revise the concept of satisfiability making 
it adequate for ultrafilter models in the wide 
sense, we use the notion of convergence of 
ultrafilters. Recall that a~filter~$\mathfrak d$ 
over a~topological space~$X$ {\it converges\/} 
to a~point $x\in X$ iff any neighborhood of~$x$ 
belongs to~$\mathfrak d$. If a~filter~$\mathfrak d$ 
converges to a~unique point~$x$ then $x$~is called 
the {\it limit\/} of~$\mathfrak d$, in which case 
we shall write $\lim\mathfrak d=x$. 
As well-known, any filter over~$X$ converges to 
at most one point iff $X$~is Hausdorff, and any 
{\it ultra\/}filter over~$X$ converges to at least 
one point iff $X$~is compact (see e.g.~\cite{Engelking};
for the concept of $\mathfrak u$-limit with 
ultrafilters~$\mathfrak u$ see~\cite{Hindman Strauss}). 
Note also that, even if $X$~is compact Hausdorff, 
some filter over~$X$ that is not an ultrafilter 
may converge to {\it no\/} single point (consider
e.g.~any discrete~$X$ with $1<|X|<\omega$ and 
the trivial filter consisting of~$X$ as its 
single element). 

\hide{
%
%
}


For an ultrafilter model in the wide sense
$
\mathfrak U=
(X,\mathfrak f,\ldots,\mathfrak r,\ldots),
$ 
fix some topologies on the sets $X^{X^n}$ and 
$\mathcal P(X^n)$ for all $n<\omega$ such that 
the signature has $n$-ary operations, 
respectively, relations.

\begin{definition}\label{def: limit ultra model}
The ultrafilter model $\mathfrak U$ 
{\it converges\/} (w.r.t~the specified family 
of topologies) to an ordinary model 
$\mathfrak A=(X,F,\ldots,R,\ldots)$ of 
the same signature iff the interpretation 
of each symbol in~$\mathfrak{U}$ converges to 
one in~$\mathfrak{A}$. Moreover, $\mathfrak A$~is
the {\it limit\/} of~$\mathfrak U$ iff the
interpretation of~$\mathfrak A$ is the pointwise
limit of one of~$\mathfrak U$, in which case 
we write $\mathfrak A=\lim\mathfrak U$. 
\end{definition}

Thus whenever the limits of all the ultrafilters 
$\mathfrak f,\ldots,\mathfrak r,\ldots$ exist then
the ultrafilter model~$\mathfrak U$ converges to 
its limit:
$$
\lim\mathfrak U=
(X,\lim\mathfrak f,\ldots,\lim\mathfrak r,\ldots),
$$ 
which is an ordinary model of the same signature.


\begin{definition}\label{def: limit satisfiability}
The ultrafilter model $\mathfrak U$ 
(endowed with the specified family of topologies)
has a~{\it well-defined satisfiability\/} iff 
there exists the limit of~$\mathfrak U$, in which 
case we define it as the ordinary satisfiability 
in the limit:
$$
\mathfrak U\VDash_{\lim}\varphi\,[v]
\;\;\text{iff}\;\;
\lim\mathfrak U\vDash\varphi\,[v]
$$ 
for all formulas~$\varphi$ and valuations~$v$. 
\end{definition}

We use the symbol~$\VDash_{\lim}$ for the renewed 
concept of satisfiability temporarily; after 
Theorem~\ref{generalized-models:-old-vs-new}, 
which states that on ultrafilter models in 
the narrow sense this concept coincides with 
the former one, we shall continue to use the 
former symbol~$\VDash$\,.


Let us firstly show that all ordinary models 
and the satisfiability in them can be regarded 
as ultrafilter models in the wide sense 
and the satisfiability defined via limits.

\begin{theorem}%
\label{ordinary-models-as-new-generalized}
Any ordinary model~$\mathfrak A$ with the usual 
satisfaction relation~$\vDash$ is (up to a~natural 
identification) an ultrafilter model~$\mathfrak U$ 
with the satisfaction relation~$\VDash_{\lim}$, so 
we have
$$
\mathfrak U\VDash_{\lim}\varphi\,[v]
\;\;\text{iff}\;\;
\mathfrak A\vDash\varphi\,[v]
$$ 
for all formulas~$\varphi$ and valuations~$v$. 
Moreover, the same is true for ordinary models 
endowed with arbitrary topologies. 
\end{theorem}

\begin{proof}
Define $\mathfrak U$ as follows: let the universe
of~$\mathfrak U$ coincide with one of~$\mathfrak A$,
which we denote by~$X$, and let the interpretation 
in~$\mathfrak U$ be the principal interpretation
giving with one in~$\mathfrak A$, i.e.~if 
an $n$-ary functional symbol is interpreted 
in $\mathfrak A$ by $F\in X^{X^n}$ then it is 
interpreted in $\mathfrak U$ by the principal
ultrafilter over $X^{X^n}$ given by~$F$, and 
likewise for predicate symbols. 
We can suppose that all topologies on $X^{X^n}$ and 
$\mathcal P(X^n)$ are discrete. Since any principal 
ultrafilter given by a~point~$x$ has the limit~$x$, 
we conclude that $\VDash_{\lim}$ in~$\mathfrak U$ 
coincides with~$\vDash$ in~$\mathfrak A$. Moreover, 
the same fact is true for every topologies on 
$X^{X^n}$ and $\mathcal P(X^n)$, which proves 
the last claim of the theorem. 
\end{proof}


Now we are going to show that the wider concepts 
of ultrafilter models and the satisfiability 
in them cover the former, narrow concepts. 
(Let us also note that the new concept is not  
exhausted by the two cases of ordinary models 
and ultrafilter models in the narrow sense.)

\begin{theorem}\label{generalized-models:-old-vs-new}
Any ultrafilter model $\mathfrak A$ in 
the narrow sense with the satisfaction relation 
is, up to a~natural identification, an ultrafilter  
model~$\mathfrak U$ in the wide sense with 
the satisfaction relation defined via limits 
in certain appropriate topologies, so we have 
$$
\mathfrak U\VDash_{\lim}\varphi\,[v]
\;\;\text{iff}\;\;
\mathfrak A\VDash\varphi\,[v]
$$ 
for all formulas~$\varphi$ and valuations~$v$. 
\end{theorem}


\begin{proof} 
We start by describing how to represent 
an ultrafilter model~$\mathfrak A$ in 
the narrow sense by a~certain ultrafilter 
model~$\mathfrak U$ in the wide sense. Let 
$\scc X$~be the universe of~$\mathfrak A$, and suppose 
that the universe of~$\mathfrak U$ coincides with~it. 
Now we must identify ultrafilters over the sets 
$X^{X^n}$ and $\mathcal P(X^n)$ with certain 
ultrafilters over the sets $(\scc X)^{(\scc X)^n}$ and 
$\mathcal P((\scc X)^n)$, respectively. Let us provide 
a~more general procedure, which will be referred to 
as the {\it identification map\/} and denoted by~$i$.

\vskip+1em
\noindent
\textbf{\textit{Identification map~$i$.}}
For any positive $n<\omega$, discrete spaces 
$X_1,\ldots,X_n$, compact Hausdorff space~$Y$, 
and $S\subseteq Y$,
we construct the map~$i$ taking ultrafilters over 
$S^{X_1\times\ldots\times X_n}$ 
to ultrafilters over 
$Y^{\scc X_1\times\ldots\times\scc X_n}$: 
\begin{align*}
i\image
\scc\bigl(S^{X_1\times\ldots\times X_n}\bigr)
\subseteq
\scc\bigl(Y^{\scc X_1\times\ldots\times\scc X_n}\bigr).
\end{align*}
The construction is going in two steps.

First, recall that the map~$\ext$ provides 
a~canonical one-to-one correspondence between 
the set $S^{X_1\times\ldots\times X_n}$ and its image 
$$
\ext\image S^{X_1\times\ldots\times X_n}=
\bigl\{\wt f:f\in S^{X_1\times\ldots\times X_n}\bigr\}
\subseteq
RC_{X_1,\ldots,X_{n-1}}
(\scc X_1,\ldots,\scc X_n,Y)
\subseteq 
Y^{\scc X_1\times\ldots\times\scc X_n}.
$$ 
This induces the bijection~${}^+$ of
$\scc(S^{X_1\times\ldots\times X_n})$ onto 
$\scc\,(\ext\image S^{X_1\times\ldots\times X_n})$ 
taking each ultrafilter~$\mathfrak f$ over 
$S^{X_1\times\ldots\times X_n}$ 
to an ultrafilter~$\mathfrak f^+$ over
$\ext\image S^{X_1\times\ldots\times X_n}$ 
by letting
$$
\mathfrak f^+=
\{\ext\image A:A\in\mathfrak f\,\}.
$$

Second, for any $S\subseteq T$ we define
the {\it lifting\/} map of $\scc S$ into~$\scc T$,
by letting for all $\mathfrak u\in\scc S$, 
\begin{align*}
\mathfrak u^{T}=
\{
B\subseteq T:
A\in\mathfrak u\text{ and }B\supseteq A
\}.
\end{align*} 
Define also the {\it projection\/} map 
of $\{\mathfrak v\in\scc T:S\in\mathfrak v\}$ 
into~$\scc S$, by letting for all such~$\mathfrak v$,
\begin{align*}
\mathfrak v_{S}=
\{A\cap S:A\in\mathfrak v\}. 
\end{align*} 
Clearly, the domain of the projection is
the range of the lifting, and moreover, 
$(\mathfrak v_S)^T=\mathfrak v$ 
and $(\mathfrak u^T)_S=\mathfrak u$, 
thus the lifting and the projection maps 
are two mutually inverse bijections. 
(Often one identifies these ultrafilters,
thus considering $\scc S$ as the closed subset 
of~$\scc T$ consisting of those ultrafilters 
over~$T$ that are concentrated on~$S$, 
see e.g.~\cite{Hindman Strauss}, Section~3.3).

Now we define $i$ as the composition of 
${}^+$ and lifting, thus for all 
$\mathfrak f\in\scc(S^{X_1\times\ldots\times X_n})$ 
we let
$$
i(\mathfrak f)=
(\mathfrak f^{+})%
^{Y^{\scc X_1\times\ldots\times\scc X_n}}.
$$
In result, for any ultrafilter~$\mathfrak f$ 
over $S^{X_1\times\ldots\times X_n}$, 
its image~$i(\mathfrak f)$ is an ultrafilter 
over $Y^{\scc X_1\times\ldots\times\scc X_n}$ 
which is concentrated on 
$
\ext\image S^{X_1\times\ldots\times X_n}=
\bigl\{\wt f:f\in S^{X_1\times\ldots\times X_n}\bigr\}.
$


Let us now expand the domain of the map~$i$ 
to ultrafilters over relations. We want to get 
$i$~taking ultrafilters over 
$\mathcal P(X_1\times\ldots\times X_n)$ 
to ultrafilters over
$\mathcal P(\scc X_1\times\ldots\times\scc X_n):$
\begin{align*}
i\image
\scc\,\mathcal P(X_1\times\ldots\times X_n)
\subseteq
\scc\,\mathcal P(\scc X_1\times\ldots\times\scc X_n).
\end{align*}
For this, we may identify $n$-ary relations with 
their characteristic functions, i.e.~$n$-ary maps
into $2=\{0,1\}$, where $2$~is endowed with the discrete 
topology, and use the definition of~$i$ for ultrafilters 
over maps (with $Y=S=2$). Equivalently, we might 
imitate the above construction: for each
$
\mathfrak r\in
\scc\,\mathcal P(X_1\times\ldots\times X_n),
$
we might turn it firstly to 
$
\mathfrak r^+\in
\scc(\ext\image\mathcal P(X_1\times\ldots\times X_n))
$
where $\ext(R)=\wt{R}$, so by 
Theorem~\ref{topological-char},
\begin{align*}
\ext\image\mathcal P(X_1\times\ldots\times X_n)
&=
\bigl\{
\wt R:R\subseteq X_1\times\ldots\times X_n
\bigr\}
\\
&=
\bigl\{Q\subseteq\scc X_1\times\ldots\times\scc X_n:
Q\text{ is right clopen w.r.t.~}%
X_1,\ldots,X_n
\bigr\}, 
\end{align*}
by letting 
$$
\mathfrak r^+=
\{\ext\image A:A\in\mathfrak r\,\},  
$$ 
and secondly, by lifting 
the obtaining ultrafilter to an ultrafilter over 
$\mathcal P(\scc X_1\times\ldots\times\scc X_n)$, 
thus letting 
$$
i(\mathfrak r)=(\mathfrak r^+)%
^{\mathcal P(\scc X_1\times\ldots\times\scc X_n)}.
$$
In result, for any ultrafilter~$\mathfrak r$ 
over $\mathcal P(X_1\times\ldots\times X_n)$, 
its image~$i(\mathfrak r)$ is an ultrafilter over
$\mathcal P(\scc X_1\times\ldots\times\scc X_n)$ 
which is concentrated on
$\{Q\subseteq\scc X_1\times\ldots\times\scc X_n:Q$~is 
right clopen w.r.t.~$X_1,\ldots,X_n\}$.


\begin{remark}\label{rmk: on +}
In fact, the map~${}^+$ is $\wt{\ext}$ for $\ext$ 
considered as a~bijection between two discrete spaces, 
and thus a~homeomorphism between the spaces of 
ultrafilters over them.

For maps these discrete spaces are
$S^{X_1\times\ldots\times X_n}$ and 
$\ext\image S^{X_1\times\ldots\times X_n}$, 
so $\wt{\ext}$~is a~homeomorphism between 
$\scc(S^{X_1\times\ldots\times X_n})$ and 
$\scc(\ext\image S^{X_1\times\ldots\times X_n})$, 
and $\mathfrak f^{+}=\wt{\ext}(\mathfrak f)$: 
$$
\xymatrix{
\scc\bigl(S^{X_1\times\ldots\times X_n}\bigr)
\!\!\!\!\!\!\!\!\!\!\!\!\!\!\!\!
&\ar@{-->}^{{}^{+}\qquad\qquad}[r]
&\;
\scc\bigl(\ext\image S^{X_1\times\ldots\times X_n}\bigr)
\\
\ar[u] 
S^{X_1\times\ldots\times X_n}
\!\!\!\!\!\!\!\!\!\!\!\!\!\!\!\!
&\ar[r]^{\:\ext\qquad\qquad} 
&\;
\ext\image S^{X_1\times\ldots\times X_n}
\ar[u]
}
$$ 
and analogously, for relations the discrete spaces 
are $\mathcal P(X_1\times\ldots\times X_n)$ and 
$\ext\image\mathcal P(X_1\times\ldots\times X_n)$, 
so $\wt{\ext}$~is a~homeomorphism between 
$\scc\,\mathcal P(X_1\times\ldots\times X_n)$ and 
$\scc(\ext\image\mathcal P(X_1\times\ldots\times X_n))$, 
and $\mathfrak r^{+}=\wt{\ext}(\mathfrak r)$: 
$$
\xymatrix{
\scc\,\mathcal P(X_1\times\ldots\times X_n)
\!\!\!\!\!\!\!\!\!\!\!\!\!\!\!\!
&\ar@{-->}^{{}^{+}\qquad\qquad\quad}[r]
&\;
\scc(\ext\image\mathcal P(X_1\times\ldots\times X_n))
\\
\ar[u] 
\mathcal P(X_1\times\ldots\times X_n)
\!\!\!\!\!\!\!\!\!\!\!\!\!\!\!\!
&\ar[r]^{\:\ext\qquad\qquad\quad} 
&\;
\ext\image\mathcal P(X_1\times\ldots\times X_n)
\ar[u]
}
$$ 
(cf.~Remark~\ref{rmk: def of ext ext} 
explaining a~similar situation with currying).
Nevertheless, we use the symbol~${}^+$ to avoid 
confusing with $\wt{\ext}$ for the map~$\ext$ into 
a~compact Hausdorff space~$Y$, which will be also 
used in our arguments below. 
\end{remark}


\begin{lemma}\label{i-is-1-1}
The map~$i$ is a~bijection between:
\begin{enumerate}
\item[(i)] 
the set of all ultrafilters over 
$S^{X_1\times\ldots\times X_n}$ 
and the set of the ultrafilters over 
$Y^{\scc X_1\times\ldots\times\scc X_n}$ 
that are concentrated on 
$\ext\image S^{X_1\times\ldots\times X_n}$, 
\item[(ii)] 
the set of all ultrafilters over 
$\mathcal P(X_1\times\ldots\times X_n)$ 
and the set of the ultrafilters over 
$\mathcal P(\scc X_1\times\ldots\times\scc X_n)$ 
that are concentrated on 
$\ext\image\mathcal P(X_1\times\ldots\times X_n)$.
\end{enumerate}
\end{lemma}

\begin{proof} 
For brevity, we let: 
$$
A=
\scc\bigl(S^{X_1\times\ldots\times X_n}\bigr),
\;\;
B=
\scc\bigl(\ext\image S^{X_1\times\ldots\times X_n}\bigr),
\;\;
C=
\bigl\{
\mathfrak g\in
\scc\bigl(Y^{\scc X_1\times\ldots\times\scc X_n}\bigr):
\ext\image S^{X_1\times\ldots\times X_n}\in\mathfrak g
\bigr\}.
$$
As we have already pointed out, the map~${}^+$ is 
a~bijection of $A$ onto $B$, and the lifting map 
is a~bijection of $B$ onto~$C$. Therefore, $i$,~as 
the composition of the two maps, is a~bijection of 
$A$ onto~$C$, which proves item~(i). Item~(ii) is 
either proved similarly or obtained from~(i) by 
replacing relations with their characteristic functions. 
(We may also note that these maps are homeomorphic 
embeddings.)
\end{proof}


Let us now turn back to 
Theorem~\ref{generalized-models:-old-vs-new} 
and the discussed there situation with 
ultrafilter models in the former, narrow sense. 
In this case, all the discrete spaces $X_1,\ldots,X_n$ 
are equal to~$X$, while the compact Hausdorff space~$Y$ 
is~$\scc X$ and its subset~$S$ is $X$ or $2$ in the cases 
of operations and relations of the model, respectively. 
(Recall that we identify elements of~$X$ with the 
principal ultrafilters given by them, so any $n$-ary 
operation on~$X$ is identified with a~map of $X^n$ 
into~$\scc X$.) We expand $i$ to ultrafilter models 
in the narrow sense by defining it pointwise:

\begin{definition}\label{def: i-of-model}
Given an ultrafilter model 
$
\mathfrak A=
(\scc X,\mathfrak f,\ldots,\mathfrak r,\ldots)
$
in the narrow sense, we let 
$$
i(\mathfrak A)=
(\scc X,i(\mathfrak f),\ldots,i(\mathfrak r),\ldots).
$$
\end{definition}

Continuing the proof of 
Theorem~\ref{generalized-models:-old-vs-new}, 
we define $\mathfrak U$, the ultrafilter model 
in the wide sense corresponding to $\mathfrak A$, 
the given ultrafilter model in the narrow sense, 
by letting 
$$\mathfrak U=i(\mathfrak A).$$


It remains to verify that the new satisfiability, 
defined via limits, coincides with the old one. 
By Theorem~\ref{e-preserves-formulas}, which states 
that for all formulas~$\varphi$ and valuations~$v$, 
we have
$\mathfrak A\VDash\varphi\,[v]$ 
iff
$e(\mathfrak A)\vDash\varphi\,[v],$ 
the latter can be redefined via the map~$e$. 
Therefore, it suffices to check that for all 
formulas~$\varphi$ and valuations~$v$, we have
$$
\lim\mathfrak U\vDash\varphi\,[v]
\;\;\text{iff}\;\;
e(\mathfrak A)\vDash\varphi\,[v].
$$ 
But actually, a~stronger fact is true:
$$
\lim\mathfrak U=e(\mathfrak A),
$$ 
thus leading to the following result.

\begin{theorem}\label{e-as-lim-of-i}
If $\mathfrak A$~is an ultrafiler model 
in the narrow sense, then 
$\lim i(\mathfrak A)=e(\mathfrak A)$: 
$$
\xymatrix{
&i(\mathfrak A)\ar^{\lim}[dd]
\\
\mathfrak A\ar[dr]^{e}\ar[ur]^{i}&
\\
&e(\mathfrak A)
}
$$
\end{theorem}

\begin{proof}
We must verify the equalities
$\lim i(\mathfrak f)=e(\mathfrak f)$ and 
$\lim i(\mathfrak r)=e(\mathfrak r)$ for all 
$\mathfrak f\in\scc(X^{X\times\ldots\times X})$ and 
$\mathfrak r\in\scc\,\mathcal P(X\times\ldots\times X)$.
This will be stated in the next, more general lemma.


Recall once more that the map~$e$ on ultrafilters 
over $n$-ary maps is $\wt{\ext}$, the continuous 
extension of the map~$\ext$, where the latter, in turn, 
takes $n$-ary maps $f:X_1\times\ldots\times X_n\to Y$ 
of discrete spaces $X_1,\ldots,X_n$ into a~compact 
Hausdorff space~$Y$, to their extensions $\wt f=\ext(f)$ 
that are right continuous w.r.t.~$X_1,\ldots,X_{n-1}$, 
and that these extensions form a~compact Hausdorff space 
w.r.t.~the $(X_1,\ldots,X_n)$-pointwise convergence 
topology:
$$
\xymatrix{
\;\quad
\scc\bigl(Y^{X_1\times\ldots\times X_n}\bigr)
\qquad\ar@{-->}^{\wt{\ext}}[dr]&&
\\
Y^{X_1\times\ldots\times X_n}\,
\ar[u]\ar[r]^{\;\;\;\ext}&&%
\!\!\!\!\!\!\!\!\!\!\!\!\!\!\!\!%
RC_{X_1,\ldots,X_{n-1}}
(\scc X_1,\ldots,\scc X_n,Y)
}
$$
The next lemma states that the map~$e$ is 
the composition of the identification map~$i$ 
and taking the limit.

\begin{lemma}\label{e-as-lim} 
Let $X_1,\ldots,X_n$ be discrete spaces,
$Y$~a~compact Hausdorff space, and let the spaces 
$RC_{X_1,\ldots,X_{n-1}}(\scc X_1,\ldots,\scc X_n,Y)$ 
and 
$
\{Q\subseteq\scc X_1\times\ldots\times\scc X_n:
Q$~is right clopen w.r.t.~$X_1,\ldots,X_{n-1}\}
$
be endowed with the $(X_1,\ldots,X_n)$-pointwise 
convergence topologies. Then 
we have
$$
\lim i(\mathfrak f)=e(\mathfrak f)
\;\;\text{and}\;\;
\lim i(\mathfrak r)=e(\mathfrak r)
$$ 
for every ultrafilters 
$
\mathfrak f\in
\scc(Y^{X_1\times\ldots\times X_n})
$ 
and 
$
\mathfrak r\in
\scc\,\mathcal P(X_1\times\ldots\times X_n).
$ 
\end{lemma}

\begin{proof}
For brevity, let $\RC$ denote the space
$RC_{X_1,\ldots,X_{n-1}}(\scc X_1,\ldots,\scc X_n,Y)$. 
By Lemma~\ref{partial-pointwise-convergence},
the space~$\RC$ is compact Hausdorff. Hence, 
$i(\mathfrak f)$, which is an ultrafilter over 
the set $Y^{\scc X_1\times\ldots\times\scc X_n}$ 
concentrated on its subset~$\RC$ and thus 
can be identified with its projection to~$\RC$, 
converges to a~unique point of~$\RC$, i.e.~has 
a~limit in~$\RC$. We need to show that the limit 
is exactly the map $e(\mathfrak f)$.

Denote $e(\mathfrak f)$ by~$F$. Then, since 
$e=\wt{\ext}$ and $\ext=\wt{\;\;\;}$, we get: 
\begin{align*}
\{F\}
&=
\bigcap_{A\in\mathfrak f}
\cl_{\RC}\,\ext\image A
=
\bigcap_{A\in\mathfrak f}
\cl_{\RC}\bigl\{\wt{f}:f\in A\bigr\}.
\end{align*}
Therefore, for every $A\in\mathfrak f$ and any 
neighborhood~$O$ of the point~$F$ in the space,  
there exists $f\in A$ such that $\wt{f}\in O$; 
here we use that the set 
$
\ext\image Y^{X_1\times\ldots\times X_n}=
\{\wt f:f\in Y^{X_1\times\ldots\times X_n}\}
$ 
is dense in~$\RC$ by Lemma~\ref{extensions-are-dense}.

Let us verify that, moreover, 
for any neighborhood~$O$ of $F\in\RC$ the set 
$\{f\in Y^{X_1\times\ldots\times X_n}:\wt f\in O\}$ 
is in~$\mathfrak f$. Assume the converse: 
there exists a~neighborhood~$O$ of~$F$ such that 
the set 
$
\{f\in Y^{X_1\times\ldots\times X_n}:
\wt f\in O\}
$ 
is not in~$\mathfrak f$. 
Then, as $\mathfrak f$~is an ultrafilter, 
the complement 
$$
A=
\bigl\{
f\in Y^{X_1\times\ldots\times X_n}:
\wt f\notin O
\bigr\}
$$ 
is in~$\mathfrak f$. 
However, this contradicts to the above stated fact.

The case of ultrafilters over relations reduces 
to the case of ultrafilters over maps with $Y=2$. 
The lemma is proved. 
\end{proof}

This proves Theorem~\ref{e-as-lim-of-i}. 
\end{proof}

Now the proof of 
Theorem~\ref{generalized-models:-old-vs-new} 
is complete.
\end{proof}

Theorem~\ref{generalized-models:-old-vs-new} permits 
us to eliminate our temporary symbol~$\VDash_{\lim}$ 
and use the former symbol~$\VDash$ also to denote the 
satisfaction in ultrafilter models in the wide sense. 
Moreover, 
by Theorem~\ref{ordinary-models-as-new-generalized}
we might use the only ordinary symbol~$\vDash$ to denote 
the satisfaction in both ordinary and ultrafilter models; 
we however prefer to retain the symbol~$\VDash$ 
for a~convenience of reading.


Finally, we refine the first part of 
Theorem~\ref{generalized-models:-old-vs-new} 
(concerning rather models than the satisfaction 
relation) by characterizing the ultrafilter 
models in the wide sense that correspond to 
those in the narrow sense:

\begin{theorem}\label{top-char-of-i-of-old}
Let $\mathfrak U$ be an ultrafilter model 
in the wide sense. Then:
\begin{enumerate}
\item[(i)] 
$\mathfrak U=i(\mathfrak A)$ 
for some ultrafilter model~$\mathfrak A$ in 
the narrow sense iff 
the universe of~$\mathfrak U$ is $\scc X$ for some~$X$ 
and the interpretation takes all functional symbols to 
ultrafilters concentrated on $\ext\image X^{X^n}$, and 
all relational symbols to ultrafilters concentrated on 
$\{Q\subseteq(\scc X)^n:Q$~is right clopen w.r.t.~$X\}$;
\item[(ii)]
$\lim\mathfrak U=\wt{\,\mathfrak A\,}$ 
for some ordinary model~$\mathfrak A$ iff 
the universe of~$\mathfrak U$ is $\scc X$ for some~$X$ 
and the interpretation takes all functional symbols to 
ultrafilters in 
$\{i(\mathfrak f):\mathfrak f\in\scc(X^{X^n})$~is 
pseudo-principal\,$\}$, and 
all relational symbols to ultrafilters concentrated on 
$\{Q\subseteq(\scc X)^n:Q$~is right clopen w.r.t.~$X\}$.
\end{enumerate}
\end{theorem}

\begin{proof} 
Item~(i) is immediate from Lemma~\ref{i-is-1-1}; 
we recall only that the images of ultrafilters over
$\mathcal P(X^n)$ under~$i$ are exactly ultrafilters 
over $\mathcal P((\scc X)^n)$ that are concentrated 
on $\ext\image\mathcal P(X^n)$: 
\begin{align*}
i\image\scc\,\mathcal P(X^n)
&=
\bigl\{
\mathfrak s\in\scc\,\mathcal P((\scc X)^n):
\ext\image\mathcal P(X^n)\in\mathfrak s
\bigr\}
\\
&=
\bigl\{
\mathfrak s\in\scc\,\mathcal P((\scc X)^n):
\{\wt{R}:R\subseteq X^n\}
\in\mathfrak s
\bigr\}
\\
&=
\bigl\{
\mathfrak s\in\scc\,\mathcal P((\scc X)^n):
\{Q\subseteq(\scc X)^n:Q\text{ is 
right clopen w.r.t.~}X\}\in\mathfrak s
\bigr\}.
\end{align*}
Item~(ii) follows from item~(i) and 
Theorem~\ref{e-and-E-as-ultraextensions}. 
\end{proof}


\vskip+1em 
\noindent
\textbf{\textit{Map~$I$.}}
Here we consider a~variant of the map~$i$, which we 
denote by~$I$. This map relates to the operation~$E$ 
in the same way as the map~$i$ to the operation~$e$ 
does (which explains our choosing of the symbol~$I$).

The map~$I$ has the same domain and range that 
the map~$i$ does:
\begin{align*}
I\image
\scc\bigl(S^{X_1\times\ldots\times X_n}\bigr)
&\subseteq
\scc\bigl(Y^{\scc X_1\times\ldots\times\scc X_n}\bigr),
\\
I\image
\scc\,\mathcal P(X_1\times\ldots\times X_n)
&\subseteq
\scc\,\mathcal P(\scc X_1\times\ldots\times\scc X_n).
\end{align*}
and is defined as follows: 
on ultrafilters over $S^{X_1\times\ldots\times X_n}$
it coincides with~$i$, and on ultrafilters 
over $\mathcal P(X_1\times\ldots\times X_n)$
it is defined likewise~$i$ except 
for taking ${}^\times$ instead of~${}^+$, where 
$\mathfrak r^\times$ uses rather $\cl$ than~$\ext$, 
i.e.~turning $R\subseteq X_1\times\ldots\times X_n$ 
not to $\wt R$ but to~$R^*$ (recall that, by 
Theorem~\ref{topological-char}, $\cl$~is a~bijection
between all subsets of $X_1\times\ldots\times X_n$ 
and regular closed subsets of 
$\scc X_1\times\ldots\times\scc X_n$).

Thus for each
$
\mathfrak r\in
\scc\,\mathcal P(X_1\times\ldots\times X_n),
$
we might turn it firstly to 
$
\mathfrak r^{\times}\in
\scc(\cl\image\mathcal P(X_1\times\ldots\times X_n))
$
where $\cl(R)=R^*$, so by 
Theorem~\ref{topological-char},
\begin{align*}
\mathfrak r^\times
&\in
\scc\,
\bigl\{
R^*:R\subseteq X_1\times\ldots\times X_n
\bigr\}
\\
&=
\scc\,
\bigl\{Q\subseteq\scc X_1\times\ldots\times\scc X_n:
Q\text{ is regular closed\,} 
\bigr\}, 
\end{align*}
by letting 
$$
\mathfrak r^{\times}=
\{\cl\image A:A\in\mathfrak r\,\},  
$$ 
and secondly, by lifting 
the obtained ultrafilter to an ultrafilter over 
$\mathcal P(\scc X_1\times\ldots\times\scc X_n)$, 
thus letting 
$$
I(\mathfrak r)=(\mathfrak r^\times)%
^{\mathcal P(\scc X_1\times\ldots\times\scc X_n)}.
$$
In result, for any ultrafilter~$\mathfrak r$ 
over $\mathcal P(X_1\times\ldots\times X_n)$, 
its image~$I(\mathfrak r)$ is an ultrafilter over
$\mathcal P(\scc X_1\times\ldots\times\scc X_n)$ 
which is concentrated on
$\{Q\subseteq\scc X_1\times\ldots\times\scc X_n:Q$~is 
regular closed\,$\}$.  
(To make an analogy between ${}^+$ and ${}^\times$ more
complete, we can also let that ${}^\times$~is defined 
on ultrafilters over $S^{X_1\times\ldots\times X_n}$
and coincides there with~${}^+$\,; but in fact we do 
not need this.)

\begin{remark}\label{rmk: on x}
Again, the map~${}^{\times}$ is $\wt{\,\cl\,}$ 
for $\cl$ considered as a~bijection between 
two discrete spaces
$\mathcal P(X_1\times\ldots\times X_n)$ and 
$\cl\image\mathcal P(X_1\times\ldots\times X_n)$, 
so $\wt{\,\cl\,}$~is a~homeomorphism between 
$\scc\,\mathcal P(X_1\times\ldots\times X_n)$ and 
$\scc(\cl\image\mathcal P(X_1\times\ldots\times X_n))$, 
and $\mathfrak r^{\times}=\wt{\,\cl\,}(\mathfrak r)$: 
$$
\xymatrix{
\scc\,\mathcal P(X_1\times\ldots\times X_n)
\!\!\!\!\!\!\!\!\!\!\!\!\!\!\!\!
&\ar@{-->}^{{}^{\times}\qquad\qquad\quad}[r]
&\;
\scc(\cl\image\mathcal P(X_1\times\ldots\times X_n))
\\
\ar[u] 
\mathcal P(X_1\times\ldots\times X_n)
\!\!\!\!\!\!\!\!\!\!\!\!\!\!\!\!
&\ar[r]^{\:\cl\qquad\qquad\quad} 
&\;
\cl\image\mathcal P(X_1\times\ldots\times X_n)
\ar[u].
}
$$ 
Nevertheless, we use the symbol~${}^\times$ 
to keep the analogy with~${}^+$\,. 
\end{remark}


Two next lemmas and the subsequent theorem are 
counterparts of Lemmas \ref{i-is-1-1} and~\ref{e-as-lim}
and Theorem~\ref{e-as-lim-of-i}, respectively.

\begin{lemma}\label{I-is-1-1}
The map~$I$ is a~bijection between:
\begin{enumerate}
\item[(i)] 
the set of all ultrafilters over 
$S^{X_1\times\ldots\times X_n}$ 
and the set of the ultrafilters over 
$Y^{\scc X_1\times\ldots\times\scc X_n}$ 
that are concentrated on 
$\ext\image S^{X_1\times\ldots\times X_n}$, 
\item[(ii)] 
the set of all ultrafilters over 
$\mathcal P(X_1\times\ldots\times X_n)$ 
and the set of the ultrafilters over 
$\mathcal P(\scc X_1\times\ldots\times\scc X_n)$ 
that are concentrated on 
$\cl_{\,\scc X_1\times\ldots\times\scc X_n}%
\!\!\image\,
\mathcal P(X_1\times\ldots\times X_n)$.
\end{enumerate}
\end{lemma}

\begin{proof} 
Item~(i) just repeats Lemma~\ref{i-is-1-1}(i) 
since $I$~coincides with~$i$ on ultrafilters over maps. 
For item~(ii), let 
\begin{gather*}
A=
\scc\,\mathcal P(X_1\times\ldots\times X_n),
\;\;
B=
\scc(\cl_{\,\scc X_1\times\ldots\times\scc X_n}%
\!\!\image\,\mathcal P(X_1\times\ldots\times X_n)),
\\
C=
\bigl\{
\mathfrak s\in
\scc\,\mathcal P(\scc X_1\times\ldots\times\scc X_n):
\cl_{\,\scc X_1\times\ldots\times\scc X_n}%
\!\!\image\,\mathcal P(X_1\times\ldots\times X_n)
\in\mathfrak s
\bigr\}.
\end{gather*} 
(Here the closure 
$\cl_{\scc X_1\times\ldots\times\scc X_n}$ refers  
to the product topology where the spaces~$\scc X_i$ 
are endowed with their standard topologies, and the set 
$
\cl_{\,\scc X_1\times\ldots\times\scc X_n}%
\!\!\image\,\mathcal P(X_1\times\ldots\times X_n)
$
in the definition of~$B$ is considered as a~discrete 
space.) As the map~${}^\times$ is a~bijection of $A$ 
onto $B$ and the lifting map is a~bijection of $B$ 
onto~$C$, the map~$I$, which the composition of the 
two maps, is a~bijection of $A$ onto~$C$, 
thus proving~(ii). 
\end{proof}


In what follows we consider 
the space $\scc X_1\times\ldots\times\scc X_n$ endowed 
with the usual product topology of the spaces~$\scc X_i$
and the set of regular closed sets in this space 
endowed with a~compact Hausdorff topology.
This topology is induced from the compact Hausdorff 
space $\mathcal P(X_1\times\ldots\times X_n)$, which 
we identify with the space $2^{X_1\times\ldots\times X_n}$ 
(where $2$~is discrete and the space carries the usual 
product topology) by the natural bijection taking
$R\subseteq X_1\times\ldots\times X_n$ to  
$R^*\subseteq\scc X_1\times\ldots\times\scc X_n$.

Recall also that 
by Lemma~\ref{partial-pointwise-convergence} 
(and its proof),
the $(X_1,\ldots,X_n)$-pointwise convergence topology 
on $RC_{X_1,\ldots,X_{n-1}}(\scc X_1,\ldots,\scc X_n,Y)$ 
can be induced from the product topology on
$Y^{X_1\times\ldots\times X_n}$ by the bijection~$\ext$, 
which takes each $f\in Y^{X_1\times\ldots\times X_n}$ 
to $\wt{f}\in Y^{\scc X_1\times\ldots\times\scc X_n}$. 
In particular, if $Y=2$ then $\ext$~on relations 
(identified with their characteristic functions), 
which takes each $R\subseteq X_1\times\ldots\times X_n$ 
to $\wt{R}\subseteq\scc X_1\times\ldots\times\scc X_n$, 
induces the above considered compact Hausdorff topology 
on $\mathcal P(\scc X_1\times\ldots\times\scc X_n)$. 
Therefore, we have three homeomorphic spaces: the 
space of subsets~$R$ of $X_1\times\ldots\times X_n$ 
and its images under the homeomorphisms $\ext$ 
and~$\cl$ taking $R,\wt{R},R^*$ into each others: 
$$
\begin{array}{c}
\xymatrix{ 
\RClop
\ar@{<->}[rr] 
&&\RegCl
\\
&\ar[ul]^{\ext}
\ar[ur]_{\cl}& 
} 
\\ 
[-2 mm] 
\mathcal P(X_1\times\ldots\times X_n)
\end{array}
$$
where 
\begin{align*}
\RClop
&=
\;\ext\image
\mathcal P(X_1\times\ldots\times X_n)
\\
&=
\bigl\{
\wt{R}:R\in\mathcal P(X_1\times\ldots\times X_n)
\bigr\}
\\
&=
\bigl\{
Q\subseteq\scc X_1\times\ldots\times\scc X_n:
Q~\text{ is right clopen w.r.t.~}X_1,\ldots,X_{n-1}
\bigr\},
\\
\RegCl
&=
\,\cl\image
\mathcal P(X_1\times\ldots\times X_n)
\\
&=
\bigl\{
R^*:R\in\mathcal P(X_1\times\ldots\times X_n)
\bigr\}
\\
&=
\bigl\{
Q\subseteq\scc X_1\times\ldots\times\scc X_n:
Q~\text{ is regular closed\,}
\bigr\}.
\end{align*}

\begin{question}\label{q: RegCl via Vietoris} 
Redefine the topology on the set $\RegCl$ as 
a~restricted version of the Vietoris topology (in 
an analogy with the restricted version of pointwise 
convergence topology turning out $\RC$ into a~compact 
Hausdorff space homeomorphic to the product space
$Y^{X_1\times\ldots\times X_n}$). Note that we cannot 
use the usual (unrestricted) Vietoris topology since 
in it, $\RegCl$ is not a~closed subset of the space
$\{Q\subseteq\scc X_1\times\ldots\times\scc X_n:
Q$~is closed$\}$.
(Problem~\ref{prb: 7}.)
\end{question}


\begin{lemma}\label{E-and-I from e-and-i}
Let $f$~be the homeomorphism between 
$
\RClop=
\{\wt{R}:R\in\mathcal P(X_1\times\ldots\times X_n)\}
$
and 
$
\RegCl=
\{R^*:R\in\mathcal P(X_1\times\ldots\times X_n)\}
$
taking $\wt{R}$ to~$R^*$. Then we have
$$
f\circ e=E
\;\;\text{and}\;\;
\wt{f}\circ i=I:
$$
$$
\begin{array}{c}
\xymatrix{ 
\scc\,\RClop 
\ar@{-->}^{\wt{f}}[rr] 
&&\scc\,\RegCl
\\
\RClop
\ar@{->}^{f}[rr]
\ar[u]  
&& 
\RegCl
\ar[u]
\\
&\ar[ul]^{e} 
\ar[ur]_{E} 
\ar[uul]_(.7){i} 
\ar[uur]^(.7){I}& 
} 
\\
[-2 mm]
\scc\,\mathcal P(X_1\times\ldots\times X_n) 
\\
\end{array}
$$
\end{lemma}

\begin{proof}
The equality $f\circ e=E$ follows from 
Theorem~\ref{e-E-relations} and, in turn, 
implies the equality $\wt{f}\circ i=I$.
\end{proof}


\begin{lemma}\label{E-as-lim} 
Let $X_1,\ldots,X_n$ be discrete spaces,
$Y$~a~compact Hausdorff space, and let the spaces 
$
\RC=
RC_{X_1,\ldots,X_{n-1}}(\scc X_1,\ldots,\scc X_n,Y)
$ 
and 
$
\RegCl= 
\{Q\subseteq\scc X_1\times\ldots\times\scc X_n:
Q$~is regular closed\,$\}
$
be endowed with the topology induced from 
$Y^{X_1\times\ldots\times X_n}$ and 
$\mathcal P(X_1\times\ldots\times X_n)$,
respectively. Then we have 
$$
\lim I(\mathfrak f)=E(\mathfrak f)
\;\;\text{and}\;\;
\lim I(\mathfrak r)=E(\mathfrak r)
$$ 
for every ultrafilters 
$
\mathfrak f\in
\scc(Y^{X_1\times\ldots\times X_n})
$ 
and 
$
\mathfrak r\in
\scc\,\mathcal P(X_1\times\ldots\times X_n).
$ 
\end{lemma}

\begin{proof} 
The first equality repeats the first equality in 
Lemma~\ref{e-as-lim} as the topology on~$\RC$ 
induced from $Y^{X_1\times\ldots\times X_n}$ coincides 
with the $(X_1,\ldots,X_n)$-pointwise convergence 
topology by Theorem~\ref{partial-pointwise-convergence}.

For the second equality, recall first a~general fact: 
if a~map $g:A\to B$ is continuous and
$\mathfrak u$~is an ultrafilter over~$A$,  
then $g(\lim\mathfrak u)=\lim\wt{g}(\mathfrak u)$ 
whenever both limits exist. It easily follows that 
if $g$~is a~homeomorphism, then
$\lim\mathfrak u=g^{-1}(\lim\wt{g}(\mathfrak u))$ and 
$\lim\wt{g}(\mathfrak u)=
g(\lim\wt{g}^{-1}(\mathfrak u))$.
In our situation, we have:
$$
\lim I(\mathfrak r)=
f(\lim\wt{f}^{-1}(I(\mathfrak r)))=
f(\lim i(\mathfrak r))=
f(e)(\mathfrak r)=
E(\mathfrak r)
$$
where the first equality holds by this general fact, 
the second follows from Lemma~\ref{E-and-I from e-and-i},
the third holds by Lemma~\ref{e-as-lim}, and the last 
again by Lemma~\ref{E-and-I from e-and-i}.  
\end{proof}


Likewise $i$, we expand $I$ to ultrafilter models 
in the narrow sense pointwise: 

\begin{definition}\label{def: I-on_model}
For all ultrafilter models 
$
\mathfrak A=
(\scc X,\mathfrak f,\ldots,\mathfrak r,\ldots)
$
in the narrow sense, let 
$$
I(\mathfrak A)=
(\scc X,I(\mathfrak f),\ldots,I(\mathfrak r),\ldots).
$$
\end{definition}

\begin{theorem}\label{E-as-lim-of-I}
If $\mathfrak A$~is an ultrafilter model 
in the narrow sense, then 
$\lim I(\mathfrak A)=E(\mathfrak A)$:
$$
\xymatrix{
&I(\mathfrak A)\ar^{\lim}[dd]
\\
\mathfrak A\ar[dr]^{E}\ar[ur]^{I}&
\\
&E(\mathfrak A)
}
$$
\end{theorem}

\begin{proof}
Immediate from Lemma~\ref{E-as-lim}. 
\end{proof}


We summarize the interplay between ultrafilter 
models $\mathfrak A$ in the narrow sense, their 
limits, and the operations $i,I$, $e,E$, and $f$ 
(where $f$ is defined on models of form 
$e(\mathfrak A)$ as expected) 
in the following diagram:
$$
\xymatrix{ 
i(\mathfrak A)
\ar@{-->}^{\wt{f}}[rr]
\ar_{\lim}[d] 
&&I(\mathfrak A) 
\ar^{\lim}[d]
\\
e(\mathfrak A)
\ar@{->}^{f}[rr]
&&E(\mathfrak A)
\\
&\mathfrak A
\ar[ul]^{e}
\ar[ur]_{E}
\ar[uul]_(.7){i}
\ar[uur]^(.7){I} 
}
$$


Despite the fact that Theorem~\ref{E-as-lim-of-I} 
is an $E$-analog of Theorem~\ref{e-as-lim-of-i} used 
to get Theorem~\ref{generalized-models:-old-vs-new}, 
the latter theorem has no such analog. This is due
to an asymmetry between the operations $e$ and~$E$ 
w.r.t.~the satisfiability in ultrafilter models in 
the narrow sense, as it has been defined: there is 
no $E$-analog of Theorem~\ref{e-preserves-formulas}, 
which was also used in proving
Theorem~\ref{generalized-models:-old-vs-new} 
(cf., however, Problem~\ref{prb: 5}). 
Nonetheless, we are still able to get a~counterpart 
of Theorem~\ref{top-char-of-i-of-old}, which does not 
involve satisfiability:

\begin{theorem}\label{top-char-of-I-of-old}
Let $\mathfrak U$ be an ultrafilter model 
in the wide sense. Then:
\begin{enumerate}
\item[(i)] 
$\mathfrak U=I(\mathfrak A)$ 
for some ultrafilter model~$\mathfrak A$ 
in the narrow sense iff 
the universe of~$\mathfrak U$ is $\scc X$ for some~$X$ 
and the interpretation takes 
all functional symbols to ultrafilters
concentrated on $\ext\image X^{X^n}$, and 
all relational symbols to ultrafilters concentrated 
on $\{Q\subseteq(\scc X)^n:Q$~is regular closed\,$\}$;
\item[(ii)]
$\lim\mathfrak U=\mathfrak A^*$ 
for some ordinary model~$\mathfrak A$ iff 
the universe of~$\mathfrak U$ is $\scc X$ for some~$X$ 
and the interpretation takes 
all functional symbols to ultrafilters in 
$\{I(\mathfrak f):\mathfrak f\in\scc(X^{X^n})$~is 
pseudo-principal\,$\}$, and 
all relational symbols to ultrafilters concentrated 
on $\{Q\subseteq(\scc X)^n:Q$~is regular closed\,$\}$.
\end{enumerate}
\end{theorem}

\begin{proof} 
Item~(i) is immediate from Lemma~\ref{I-is-1-1}; 
recall that the images of ultrafilters over
$\mathcal P(X^n)$ under~$I$ are exactly ultrafilters 
over $\mathcal P((\scc X)^n)$ that are concentrated on
$
\cl_{\,\scc X_1\times\ldots\times\scc X_n}\!\!\image
\,\mathcal P(X^n)
$: 
\begin{align*}
i\image\scc\,\mathcal P(X^n)
&=
\bigl\{
\mathfrak s\in\scc\,\mathcal P((\scc X)^n):
\cl_{\,\scc X_1\times\ldots\times\scc X_n}\!\!\image
\,\mathcal P(X^n)\in\mathfrak s
\bigr\}
\\
&=
\bigl\{
\mathfrak s\in\scc\,\mathcal P((\scc X)^n):
\{R^*:R\subseteq X^n\}
\in\mathfrak s
\bigr\}
\\
&=
\bigl\{
\mathfrak s\in\scc\,\mathcal P((\scc X)^n):
\{Q\subseteq(\scc X)^n:Q\text{ is 
regular closed\,}\}\in\mathfrak s
\bigr\}.
\end{align*}
Item~(ii) follows from item~(i) and 
Theorem~\ref{e-and-E-as-ultraextensions}. 
\end{proof}


\vskip+1em
\noindent
\textbf{\textit{Second Extension Theorems.}}
Now we define homomorphisms between ultrafilter 
models in the wide sense as homomorphisms of 
their limits:

\begin{definition}\label{def: homo (wide)}
Let $\mathfrak U$ and $\mathfrak V$ be two ultrafilter 
models in the wide sense, of the same signature,
with the universes $X$ and~$Y$, respectively.  
A~map $h:X\to Y$ is a~{\it homomorphism} 
({\it of ultrafilter  models in the wide sense}) 
iff it is a~homomorphism of  
$\lim\mathfrak U$ into $\lim\mathfrak V$. 
\end{definition}

Theorem~\ref{generalized-models:-old-vs-new} 
guaranties that for ultrafilter models in the narrow 
sense, Definition~\ref{def: homo (wide)} 
gives the same, up to the identification map~$i$, 
that Definition~\ref{def: homo (narrow)}.
Similar concepts ({\it epimorphisms}, {\it quotients}, 
{\it isomorphic embeddings}, {\it submodels},  
{\it elementary embeddings}, {\it elementary submodels}, 
etc.~of ultrafilter models in the wide sense) 
are defined likewise and also coincide with 
the corresponding concepts for ultrafilter models 
in the narrow sense.

The proofs of the Second Extension Theorems 
(Theorems \ref{modelSET} and~\ref{modalSET}) 
are based on Theorems \ref{modeltopology} 
and~\ref{modaltopology}, which describe the 
topological properties of the $\wt{\;\;}$- and 
${}^*$-extensions, respectively, and a~result 
called the ``abstract extension theorem'' 
in~\cite{Saveliev(inftyproc)}. This result is 
rather about restrictions of continuous maps than 
about continuous extensions of maps, but it also
states that such a~map is a~homomorphism of the 
whole models whenever it is a~homomorphism of 
certain submodels in them; we restate it in 
the next theorem:

\begin{theorem}\label{abstract-ET}
Let $\mathfrak A$ and $\mathfrak B$ be 
two (ordinary) models of the same signature 
whose universes $X$ and~$Y$, respectively, 
both carry topologies, the topology on~$Y$ 
is Hausdorff, and let $D\subseteq X$ 
be a~dense subset of~$X$ which forms
a~submodel~$\mathfrak D$ of~$\mathfrak A$. 
Let, moreover, $h:X\to Y$ be a~continuous map, 
and suppose that 
 \begin{enumerate}
\item[(a)] 
all operations in~$\mathfrak A$ are right continuous
w.r.t.~$D$, and in~$\mathfrak B$ right continuous 
w.r.t.~$h\image D$,
\item[(b)] 
one of two following items holds:
\begin{enumerate}
\item[($\alpha$)] 
all relations in~$\mathfrak A$ are right open
w.r.t.~$D$, and in~$\mathfrak B$ right closed
w.r.t.~$h\image D$,
\item[($\beta$)] 
all relations in~$\mathfrak A$ are regular closed
in the product topology on $X^n$, in~$\mathfrak B$ 
closed in the product topology on~$Y^n$ (where $n$~is 
the arity of a~given relation), and $h$~is a~closed map.
\end{enumerate} 
\end{enumerate}
Then the following are equivalent:
\begin{enumerate}
\item[(i)] 
$h\uhr D$ is a~homomorphism of $\mathfrak D$ 
into~$\mathfrak B$, 
\item[(ii)] 
$h$~is a~homomorphism of $\mathfrak A$ 
\end{enumerate} 
$$
\xymatrix{
\mathfrak A
\ar@{-->}[rr]^{h}
&&\mathfrak B&
\\
\mathfrak D
\ar[rr]^{h\upharpoonright D}
\ar[u]
&&\mathfrak E
\ar[u]&
}
$$
where $\mathfrak E$~denotes the submodel 
of~$\mathfrak B$ with the universe $h\image D$.
\end{theorem}

\begin{proof} 
That (ii) implies~(i) is trivial since 
$\mathfrak D$ is a~submodel of~$\mathfrak A$. 
For the converse implication in the case of~($\alpha$), 
see \cite{Saveliev} or~\cite{Saveliev(inftyproc)}, 
Theorem~4.1. The case of~($\beta$) is obtained from 
the case of~($\alpha$) as follows.

If $R$~is an $n$-ary relation on~$X$ belonging to the
model~$\mathfrak A$, consider $R$ as a~unary relation
on~$X^n$ and note that under~(i), the restriction 
$h\uhr D$ is also a~homomorphism between the model 
$(D^n,(\inter_{X^n}R)\cap D^n)$ and the model $(Y^n,S)$,
where $\inter_{X^n}R$ is the interior of~$R$ in 
the product topology on~$X^n$, so the set
$(\inter_{X^n}R)\cap D^n=\inter_{D^n}(R\cap D^n)$ 
is an open unary relation on~$D^n$, and 
$S$~is the relation on~$Y$ interpreting the same 
predicate symbol that $R$~doing and also considered 
as a~unary relation on~$Y^n$. 
By~($\alpha$) we conclude that $h$~is a~homomorphism 
between $(X^n,\inter_{X^n}R)$ and $(Y^n,S)$. But as 
in the case of~($\beta$) the map~$h$ is closed, we have: 
$h\image\cl_{X^n}\inter_{X^n}R\subseteq\cl_{Y^n}S$, 
thus $h$~is a~homomorphism between 
$(X^n,\cl_{X^n}\inter_{X^n}R)$ and $(Y^n,\cl_{Y^n}S)$. 
Finally, as under~($\beta$), $R$~is regular closed 
in~$X^n$ and $S$~is closed in~$Y^n$, we have 
$\cl_{X^n}\inter_{X^n}R=R$ and $\cl_{Y^n}S=S$, 
thus showing that $h$~is a~homomorphism between 
$(X^n,R)$ and $(Y^n,S)$, and hence, between 
$(X,R)$ and $(Y,S)$ where the relations $R$ and~$S$ 
are considered as $n$-ary. This gives~(ii), 
completing the proof of the theorem.
\end{proof}


\begin{remark}\label{rmk: refining abstract-ET}
The argument for proving ($\beta$) from~($\alpha$) 
allows to obtain stronger statements. Instead of 
the assumption of~($\beta$), it suffices to suppose 
that the interior of~$R$ is dense in the closure of~$R$, 
i.e.~the closure of~$R$ is regular closed:
$\cl_{X^n}\inter_{X^n}R=\cl_{X^n}R.$
This includes the cases of open as well as
of regular closed~$R$. 
Also instead of the assumption of~($\alpha$), 
it suffices to suppose that $R$~``has right dense 
interior w.r.t.~$D$'', i.e.~that for each 
$i,1\le i\le n$, and every $a_1,\ldots,a_{i-1}\in D$ 
and $x_{i+1},\ldots,x_n\in X$, the set 
$
\{
x\in X:
(a_1,\ldots,a_{i-1},x,x_{i+1},\ldots,x_n)
\in R
\}
$
has the interior dense in the closure of this set. 
As easy to see, in both cases the same proof works 
as well.

The same is applied to the following theorem. 
\end{remark}


The next result is an immediate analog of 
Theorem~\ref{abstract-ET} for ultrafilter models 
in the wide sense.

\begin{theorem}\label{abstract-ET-generalized}
Let $\mathfrak U$ and $\mathfrak V$ be two 
ultrafilter models in the wide sense, of 
the same signature, whose universes $X$ and~$Y$, 
respectively, both carry topologies, the topology 
on~$Y$ is Hausdorff, and let $D\subseteq X$ be 
a~dense subset of~$X$ which forms an ultrafilter 
submodel~$\mathfrak D$ of~$\mathfrak U$. 
Let, moreover, $h:X\to Y$ be a~continuous map, 
and suppose that for any $n<\omega$,
\begin{enumerate}
\item[(a)]
$n$-ary functional symbols are interpreted: 
in~$\mathfrak U$ by ultrafilters having limits 
in $RC_D(X^n,X)$, and 
in~$\mathfrak V$ by ultrafilters having limits 
in $RC_{h\image D}(Y^n,Y)$,
\item[(b)]
one of two following items holds:
\begin{enumerate}
\item[($\alpha$)]
$n$-ary predicate symbols are interpreted: 
in~$\mathfrak U$ by ultrafilters having limits 
in $\{R\subseteq X^n:R$~is right open w.r.t.~$D\}$, 
and in~$\mathfrak V$ by ultrafilters having limits 
in $\{S\subseteq Y^n:S$~is right closed 
w.r.t.~$h\image D\}$, 
\item[($\beta$)]
$n$-ary predicate symbols are interpreted: 
in~$\mathfrak U$ by ultrafilters having limits 
in $\{R\subseteq X^n:R$~is regular closed\,$\}$, 
and in~$\mathfrak V$ by ultrafilters having limits 
in $\{S\subseteq Y^n:S$~is closed\,$\}$, 
in the product topologies on $X^n$ and~$Y^n$, 
respectively, and $h$~is a~closed map. 
\end{enumerate}
\end{enumerate}
Then the following are equivalent:
\begin{enumerate}
\item[(i)] 
$h\uhr D$ is a~homomorphism of $\mathfrak D$ 
into~$\mathfrak V$, 
\item[(ii)] 
$h$ is a~homomorphism of $\mathfrak U$ 
into~$\mathfrak V$: 
\end{enumerate}
$$
\xymatrix{
\mathfrak U
\ar@{-->}[rr]^h
\ar[ddr]^(.3){\lim}
&&\mathfrak V
\ar[ddr]^(.3){\lim}&
\\
\mathfrak D
\ar[rr]^{h\upharpoonright D}
\ar[u]
\ar[ddr]^(.3){\lim}
&&\mathfrak E
\ar[u]
\ar[ddr]^(.3){\lim}&&
\\
&\lim\mathfrak U
\ar@{-->}[rr]^h
&&\lim\mathfrak V&
\\
&\lim\mathfrak D
\ar[rr]^{h\upharpoonright D}
\ar[u]
&&\lim\mathfrak E\ar[u]& 
}
$$
where $\mathfrak E$~denotes the ultrafilter submodel 
of~$\mathfrak V$ with the universe $h\image D$.
\end{theorem}

\begin{proof} 
By definition, homomorphisms of ultrafilter models
$\mathfrak D$, $\mathfrak U$, and $\mathfrak V$ 
are precisely homomorphisms of the ordinary models 
$\lim\mathfrak D$, $\lim\mathfrak U$, and 
$\lim\mathfrak V$. Now apply Theorem~\ref{abstract-ET} 
to the latter three models.
\end{proof}

Note that Theorem~\ref{abstract-ET-generalized} 
includes Theorem~\ref{abstract-ET} as a~particular 
case by identifying operations and relations 
with principal ultrafilters given by them as in 
Theorem~\ref{ordinary-models-as-new-generalized}.


Before formulating an extension theorem 
for ultrafilter models in the wide sense, 
let us state one more auxiliary result:

\begin{lemma}\label{pre-SET-generalized}
Let $\mathfrak U$ and $\mathfrak V$ be two ultrafilter 
models in the wide sense, of the same signature, 
whose universes $\scc X$ and~$Y$, respectively, 
both carry topologies where 
the topology on~$\scc X$ is standard, 
let $\mathfrak U$~coincide, up to the identification 
map~$i$, with an ultrafilter model~$\mathfrak B$ in 
the narrow sense (also having the universe~$\scc X$), 
and let $g:\scc X\to Y$. 
Then the following are equivalent: 
\begin{enumerate}
\item[(i)] 
$g$~is a~homomorphism of $\mathfrak U$ 
into~$\mathfrak V$,
\item[(ii)] 
$g$~is a~homomorphism of $e(\mathfrak B)$
into $\mathfrak V$.  
\end{enumerate}
If moreover, the interpretation in $\mathfrak V$ takes 
$n$-ary predicate symbols to ultrafilters having limits 
in $\{S\subseteq Y^n:S$~is closed\,$\}$ where 
$Y^n$~carries the product topology, 
and the map~$g$ is closed, then the following item:
\begin{enumerate}
\item[(iii)] 
$g$~is a~homomorphism of $E(\mathfrak B)$
into $\mathfrak V$, 
\end{enumerate} 
also is equivalent to each of items (i) and~(ii): 
$$
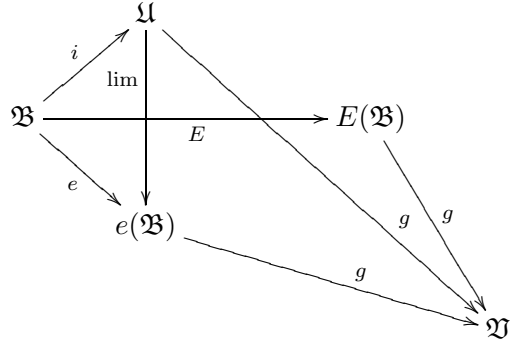
\xymatrix{
&\mathfrak U
\ar[dd]_(.3){\lim}
\ar@{->}[rrrddd]^(.7)g&&& 
\\
\mathfrak B
\ar[ur]^i
\ar[dr]_e
\ar[rrr]_E
&&&E(\mathfrak B)
\ar[ddr]^{g}&
\\
&e(\mathfrak B)
\ar@{->}[rrrd]^(.6)g&&&
\\
&&&&\mathfrak V
}
$$
\end{lemma}

\begin{proof} 
Before proving the equivalences, recall that by 
Theorem~\ref{ordinary-models-as-new-generalized},
the ordinary models $e(\mathfrak B)$ and 
$E(\mathfrak B)$ are identified with ultrafilter 
models in the wide sense having the principal 
interpretations, and the limits of the principal 
ultrafilters over the sets of operations and 
relations are the operations and relations that 
generate them. The formulations of (i) and~(ii) 
imply such an identification.

The equivalence of items (i) and~(ii) requires 
no special assumptions about $\mathfrak B$, 
$\mathfrak V$, and~$g$; it is immediate from 
the following: our definition of homomorphisms 
of ultrafilter models via their limits, the 
assumption $\mathfrak U=i(\mathfrak B)$, and 
the equality $e(\mathfrak B)=\lim i(\mathfrak B)$ 
stated in Theorem~\ref{e-as-lim-of-i}.

To prove that item~(iii) under the additional 
assumption is also equivalent to (i) and~(ii), 
we repeat the part of the proof of 
Theorem~\ref{abstract-ET} that deduces 
the case of~($\beta$) from the case of~($\alpha$), 
taking into account Theorem~\ref{e-and-E-topology} 
stating that all relations in $E(\mathfrak B)$ are 
regular closed in the product topology on~$(\scc X)^n$.

The proof is complete. 
\end{proof}


Now we are ready to formulate a~version of the 
Second Extension Theorem for ultrafilter models 
in the wide sense.

\begin{theorem}\label{SET-generalized}
Let 
$
\mathfrak U=
(\scc X,\mathfrak f,\ldots,\mathfrak r,\ldots)
$ 
and 
$
\mathfrak V=
(Y,\mathfrak g,\ldots,\mathfrak s,\ldots)
$ 
be two ultrafilter models in the wide sense, 
of the same signature, let $\scc X$~carry 
its standard topology and $Y$ a~compact Hausdorff 
topology, let $h:X\to Y$, and suppose that 
\begin{enumerate}
\item[(a)] 
$\mathfrak U$~coincides, up to the identification 
map~$i$, with an ultrafilter model~$\mathfrak B$ 
in the narrow sense, and the interpretation 
in~$\mathfrak B$ is pseudo-principal on functional 
symbols with $\mathfrak A$~the principal submodel
(having the universe~$X$),
\item[(b)] 
the interpretation in $\mathfrak V$ takes 
all $n$-ary functional symbols to ultrafilters 
having limits in $RC_{h\image X}(Y^n,Y)$, and 
all $n$-ary predicate symbols to ultrafilters 
having limits in $\{S\subseteq Y^n:S$~is 
right closed w.r.t.~$h\image X\}$, 
for any $n<\omega$. 
\end{enumerate} 
Then the following are equivalent: 
\begin{enumerate}
\item[(i)] 
$h$~is a~homomorphism of $\mathfrak A$ 
into $\mathfrak V$, 
\item[(ii)] 
$\wt{h}$~is a~homomorphism of $\mathfrak U$ 
into~$\mathfrak V$,
\item[(iii)] 
$\wt{h}$~is a~homomorphism of $\wt{\,\mathfrak A\,}$
into $\mathfrak V$.  
\end{enumerate}
If moreover, the interpretation in $\mathfrak V$ 
takes $n$-ary predicate symbols to ultrafilters 
having limits in 
$\{S\subseteq Y^n:S$~is closed\,$\}$
where $Y^n$~carries the product topology, 
then the following item:
\begin{enumerate}
\item[(iv)] 
$\wt{h}$~is a~homomorphism of $\mathfrak A^*$
into $\mathfrak V$, 
\end{enumerate} 
also is equivalent to each of items (i)--(iii):
$$
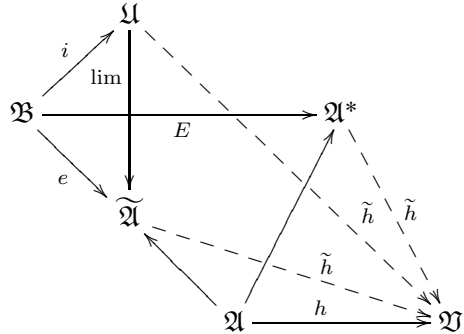
\xymatrix{
&\mathfrak U
\ar[dd]_(.3){\lim}
\ar@{-->}[rrrddd]^(.7){\wt{h}}&&& 
\\
\mathfrak B
\ar[ur]^i
\ar[dr]_e
\ar[rrr]_E
&&&\mathfrak A^\ast
\ar@{-->}[ddr]^{\wt h}&
\\
&\wt{\,\mathfrak A\,}
\ar@{-->}[rrrd]^(.6){\wt{h}}&&&\\
&&\mathfrak A
\ar[rr]^(.4)h
\ar[ul]
\ar[uur]
&&\mathfrak V
}
$$
\end{theorem}

\begin{proof} 
The assumptions about $\mathfrak U$, $\mathfrak V$, 
and $\wt{h}$ of this theorem repeat the assumptions 
about $\mathfrak U$, $\mathfrak V$, and $g$ of 
Lemma~\ref{pre-SET-generalized} with an extra 
requirement stating that the interpretation in 
$\mathfrak B$ is pseudo-principal on functional 
symbols. So as the principal submodel 
of~$\mathfrak B$ is~$\mathfrak A$, we have 
$e(\mathfrak B)=\wt{\,\mathfrak A\,}$ 
and $E(\mathfrak B)=\mathfrak A^*$
by Theorem~\ref{e-and-E-as-ultraextensions}. 
Thus by Lemma~\ref{pre-SET-generalized}, 
items (ii) and~(iii) are equivalent, and under 
the additional assumption about~$\mathfrak V$, 
item~(iv) is also equivalent to each of them.

Let us now prove that (i) and~(ii) are equivalent. 
It suffices to show that the models $\mathfrak U$ 
and~$\mathfrak V$ satisfy the conditions of 
Theorem~\ref{abstract-ET-generalized}($\alpha$). 
For~$\mathfrak V$, this is true by the assumption 
of~(b). As for~$\mathfrak U$, by the assumption 
of~(a) we have $\mathfrak U=i(\mathfrak B)$ with 
$\mathfrak B$~an ultrafilter model in the narrow sense. 
Furthermore, $i(\mathfrak B)$ is an ultrafilter model 
in the wide sense, and by Lemma~\ref{i-is-1-1}, 
the interpretation of~$i(\mathfrak B)$
takes functional symbols to ultrafilters
concentrated on $RC_{X}((\scc X)^n,\scc X)$, 
and relational symbols to ultrafilters concentrated on
$\{Q\subseteq(\scc X)^n:Q$~is right clopen w.r.t.~$X\}$. 
Therefore, the ultrafilters have limits in these sets 
endowed with the $X$-pointwise convergence topologies 
as the latter are compact Hausdorff by
Lemma~\ref{partial-pointwise-convergence}. 
Thus $\mathfrak U$ also satisfies the conditions 
of Theorem~\ref{abstract-ET-generalized}($\alpha$), 
with the principal submodel~$\mathfrak A$ here as 
the submodel~$\mathfrak D$ from that theorem (again 
by identifying ordinary models with ultrafilter  
models having the principal interpretations). 
This shows the equivalence of (i) and~(ii), 
thus completing the proof. 
\end{proof}


Finally, by changing $i$ with~$I$, we obtain 
the counterparts of Lemma~\ref{pre-SET-generalized}
and Theorem~\ref{SET-generalized}:

\begin{lemma}\label{pre-SET-generalized-I-version}
Let $\mathfrak U$ and $\mathfrak V$ be two 
ultrafilter models in the wide sense, of the 
same signature, whose universes $\scc X$ and~$Y$, 
respectively, both carry topologies where  
the topology on~$\scc X$ is standard, 
let $g:\scc X\to Y$, and suppose that 
\begin{enumerate}
\item[(a)] 
$\mathfrak U=I(\mathfrak B)$ for some ultrafilter 
model~$\mathfrak B$ in the narrow sense (also having 
the universe~$\scc X$), 
\item[(b)] 
the interpretation in $\mathfrak V$ takes 
all $n$-ary functional symbols to ultrafilters 
having limits in $RC_{g\image X}(Y^n,Y)$, and 
all $n$-ary predicate symbols to ultrafilters 
having limits in $\{S\subseteq S^n:R$~is closed\,$\}$
where $Y^n$~carries the product topology. 
\end{enumerate}
Then the following are equivalent: 
\begin{enumerate}
\item[(i)] 
$g$~is a~homomorphism of $\mathfrak U$ 
into~$\mathfrak V$,
\item[(ii)] 
$g$~is a~homomorphism of $e(\mathfrak B)$
into $\mathfrak V$, 
\item[(iii)] 
$g$~is a~homomorphism of $E(\mathfrak B)$
into $\mathfrak V$: 
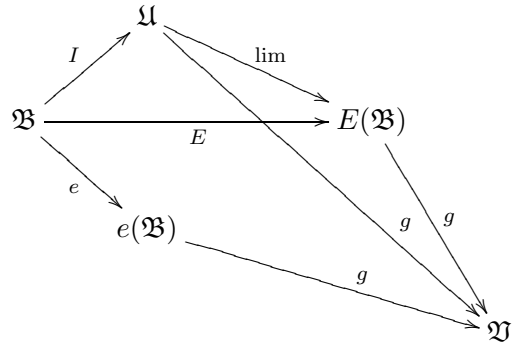
$$
\xymatrix{
&\mathfrak U
\ar[drr]^{\lim}
\ar@{->}[rrrddd]^(.7)g&&& 
\\
\mathfrak B
\ar[ur]^I
\ar[dr]_e
\ar[rrr]_E
&&&E(\mathfrak B)
\ar[ddr]^{g}&
\\
&e(\mathfrak B)
\ar@{->}[rrrd]^(.6)g&&&
\\
&&&&\mathfrak V
}
$$
\end{enumerate} 
\end{lemma}

\begin{proof} 
Again, by using 
Theorem~\ref{ordinary-models-as-new-generalized},
we identify the ordinary models $e(\mathfrak B)$ 
and $E(\mathfrak B)$ in (ii) and~(iii) with the 
corresponding ultrafilter models in the wide sense 
having the principal interpretations (and thus having 
$e(\mathfrak B)$ and $E(\mathfrak B)$ as their limits). 
The equivalence of items (i) and~(iii) is immediate 
from the following: our definition of homomorphisms 
of ultrafilter models via their limits, 
the assumption $\mathfrak U=I(\mathfrak B)$, and the 
equality $e(\mathfrak B)=\lim I(\mathfrak B)$ stated 
in Theorem~\ref{E-as-lim-of-I}. But then item~(ii) is 
also equivalent to each of items (i) and~(iii) since 
the identity map on~$\scc X$ is a~homomorphism of 
$e(\mathfrak B)$ onto $E(\mathfrak B)$ by 
Theorem~\ref{homo-of-e-to-E}. The proof is complete. 
\end{proof}


\begin{theorem}\label{SET-generalized-I-version}
Let 
$
\mathfrak U=
(\scc X,\mathfrak f,\ldots,\mathfrak r,\ldots)
$ 
and 
$
\mathfrak V=
(Y,\mathfrak g,\ldots,\mathfrak s,\ldots)
$ 
be two ultrafilter models in the wide sense, 
of the same signature, let $\scc X$~carry its 
standard topology and $Y$ a~compact Hausdorff 
topology, let $h:X\to Y$, and suppose that
\begin{enumerate}
\item[(a)] 
$\mathfrak U=I(\mathfrak B)$ for some ultrafilter 
model~$\mathfrak B$ in the narrow sense, 
and the interpretation in~$\mathfrak B$ is 
pseudo-principal on functional symbols 
with $\mathfrak A$~the principal submodel
(having the universe~$X$),
\item[(b)] 
the interpretation in $\mathfrak V$ takes 
all $n$-ary functional symbols to ultrafilters 
having limits in $RC_{h\image X}(Y^n,Y)$, and 
all $n$-ary predicate symbols to ultrafilters 
having limits in 
$\{S\subseteq Y^n:S$~is closed\,$\}$
where $Y^n$~carries the product topology. 
\end{enumerate} 
Then the following are equivalent: 
\begin{enumerate}
\item[(i)] 
$h$~is a~homomorphism of $\mathfrak A$ 
into $\mathfrak V$, 
\item[(ii)] 
$\wt{h}$~is a~homomorphism of $\mathfrak U$ 
into~$\mathfrak V$,
\item[(iii)] 
$\wt{h}$~is a~homomorphism of $\wt{\,\mathfrak A\,}$
into $\mathfrak V$,
\item[(iv)] 
$\wt{h}$~is a~homomorphism of $\mathfrak A^*$
into $\mathfrak V$:
\end{enumerate} 
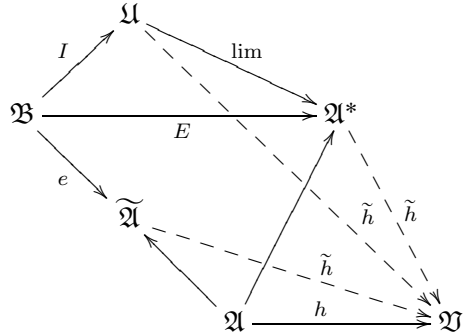
$$
\xymatrix{
&\mathfrak U
\ar[drr]^{\lim}
\ar@{-->}[rrrddd]^(.7){\wt{h}}&&& 
\\
\mathfrak B
\ar[ur]^I
\ar[dr]_e
\ar[rrr]_E
&&&\mathfrak A^\ast
\ar@{-->}[ddr]^{\wt h}&
\\
&\wt{\,\mathfrak A\,}
\ar@{-->}[rrrd]^(.6){\wt{h}}&&&
\\
&&\mathfrak A
\ar[rr]^(.4)h
\ar[ul]
\ar[uur]
&&\mathfrak V
}
$$
\end{theorem}

\begin{proof}
The assumptions about $\mathfrak U$, $\mathfrak V$, 
and $\wt{h}$ of this theorem repeat the assumptions 
about $\mathfrak U$, $\mathfrak V$, and $g$ of 
Lemma~\ref{pre-SET-generalized-I-version} with an 
extra requirement stating that the interpretation 
in $\mathfrak B$ is pseudo-principal on functional 
symbols. So as the principal submodel of~$\mathfrak B$ 
is~$\mathfrak A$, we have 
$e(\mathfrak B)=\wt{\,\mathfrak A\,}$ 
and $E(\mathfrak B)=\mathfrak A^*$
by Theorem~\ref{e-and-E-as-ultraextensions}. 
Thus by Lemma~\ref{pre-SET-generalized-I-version}, 
items (ii), (iii), and~(iv) all are equivalent.

Let us now prove that (i) and~(ii) are equivalent. 
It suffices to show that the models $\mathfrak U$ 
and~$\mathfrak V$ satisfy the conditions of 
Theorem~\ref{abstract-ET-generalized}($\beta$). 
For~$\mathfrak V$, this is true by the assumption 
of~(b). As for~$\mathfrak U$, by the assumption 
of~(a) we have $\mathfrak U=I(\mathfrak B)$ with 
$\mathfrak B$~an ultrafilter model in the narrow sense. 
Furthermore, $I(\mathfrak B)$ is an ultrafilter model 
in the wide sense, and by Lemma~\ref{I-is-1-1}, 
the interpretation of~$I(\mathfrak B)$
takes functional symbols to ultrafilters
concentrated on $RC_{X}((\scc X)^n,\scc X)$, and 
relational symbols to ultrafilters concentrated on
$\{Q\subseteq(\scc X)^n:Q$~is regular closed\,$\}$. 
Therefore, the ultrafilters have limits in these sets 
endowed with the corresponding compact Hausdorff 
topologies described above. 
Thus $\mathfrak U$ also satisfies the conditions 
of Theorem~\ref{abstract-ET-generalized}($\beta$), 
with the principal submodel~$\mathfrak A$ here as 
the submodel~$\mathfrak D$ from that theorem (again 
by identifying ordinary models with ultrafilter 
models having the principal interpretations). 
This shows the equivalence of (i) and~(ii), 
thus completing the proof. 
\end{proof}


\begin{remark}\label{rmk: variants of SET}
Theorems 
\ref{abstract-ET}--\ref{SET-generalized-I-version}
admits some variants and generalizations. 
E.g.~they remain true for epimorphisms (since 
for any compact Hausdorff~$Y$, if $h:X\to Y$ 
is such that $h\image X$ is dense in~$Y$, then 
$\wt{h}:\scc X\to Y$ is surjective), as well 
as for homotopies and isotopies (in sense of
\cite{Saveliev},~\cite{Saveliev(inftyproc)}), 
which can be defined for ultrafilter models 
in the wide sense in the same way as this was 
done for homomorphisms and embeddings. Also 
versions for multi-sorted models (having 
rather many universes $X_1,\ldots,X_n$ than 
one universe~$X$) can be easily stated. 
\end{remark}


\section{Problems}

This section contains a~list of questions and 
tasks, including all ones posed in the text 
above. Some of them are rather technical
(Problems \ref{prb: 1}, \ref{prb: 3}, 
\ref{prb: 4}, \ref{prb: 7})
while others are more program.

\begin{problem}\label{prb: 1}
Does Lemma~\ref{extensions-are-dense} remain true 
for the space 
$RC_{X_1,\ldots,X_{n-1}}(\scc X_1,\ldots,\scc X_n,Y)$ 
(or moreover, the space 
$Y^{\scc X_1\times\ldots\times\scc X_n}$)
endowed with the full pointwise convergence topology? 
i.e. given discrete spaces $X_1,\ldots,X_n$,
a~compact Hausdorff space~$Y$, and 
a~dense subset $S$ of~$Y$, is the set 
$$
\ext\image S^{X_1\times\ldots\times X_n}
=
\bigl\{
\wt{f}:f\in S^{X_1\times\ldots\times X_n}
\bigr\}
$$ 
dense in this space?
It can be seen that the answer is affirmative for 
unary maps, i.e.~the set $\{\wt{f}:f\in S^X\}$ is 
dense in $C(\scc X,Y)$. What happens for binary maps?
\end{problem}

\begin{problem}\label{prb: 2}
Given discrete $X_1,\ldots,X_n$ and 
compact Hausdorff~$Y$, let $\ext$~be a~map of 
$Y^{X_1\times\ldots\times X_n}$ endowed with 
the discrete topology into 
$Y^{\scc X_1\times\ldots\times\scc X_n}$ 
endowed with the usual product topology
(or equivalently, the usual pointwise convergence
topology). As the range is a~compact Hausdorff space,
the map~$\ext$ continuously extends to~$\wt\ext$:
$$
\xymatrix{
\;\quad\scc\bigl(Y^{X_1\times\ldots\times 
X_n}\bigr)\qquad\ar@{-->}^{\wt{\ext}}[dr]&&
\\
Y^{X_1\times\ldots\times X_n}\,
\ar[u]\ar[r]^{\;\;\;\ext}&&%
\!\!\!\!\!\!\!\!\!\!\!\!\!\!\!\!%
Y^{\,\scc X_1\times\ldots\times\scc X_n}
}
$$
Can this alternative version of self-applying 
of the map~$\ext$ lead to some interesting 
possibilities, including variants of the theory 
of ultrafilter models?

Note that now 
$\wt\ext\image\scc(Y^{X_1\times\ldots\times X_n})$ 
does not coincide with
$\ext\image Y^{X_1\times\ldots\times X_n}$ 
(unlike our previous situation); however, 
the latter set is still dense in the former: 
$$
\ext\image\,Y^{X_1\times\ldots\times X_n}
\subset
\wt{\ext}\image\,
\scc\bigl(Y^{X_1\times\ldots\times X_n}\bigr)
=
\cl_{Y^{\,\scc X_1\times\ldots\times\scc X_n}}
\bigl(
\ext\image\,Y^{X_1\times\ldots\times X_n}
\bigr). 
$$
Also, is this version of~$\wt\ext$ surjective?
This would be the case if the previous question 
in its stronger form, i.e.~for the space
$Y^{\,\scc X_1\times\ldots\times\scc X_n}$, 
had the affirmative answer. 
\end{problem}

\begin{problem}\label{prb: 3}
For which compact Hausdorff spaces~$Y$, 
instead of $\scc Y$ with a~discrete~$Y$, 
does Lemma~\ref{appviaext} remain true, 
i.e.~for any discrete $X_1,\ldots,X_n$ 
and the map~$\wt\app$ defined as in 
the remark in the beginning of Section~3:
$$
\xymatrix{
\scc X_1\times\ldots\times\scc X_n\times
\scc(Y^{X_1\times\ldots\times X_n})
\qquad\qquad\ar@{-->}^{\qquad\;\wt\app}[dr]&&
\!\!\!\!\!\!\!\!\!\!\!\!\!\!\! 
\\
\;\;\;{X_1\times\ldots\times X_n\times 
Y^{X_1\times\ldots\times X_n}}\,
\ar[u]\ar[r]^{\qquad\qquad\app}& 
Y
}
$$
the statements
\begin{align*}
e(\mathfrak f)(\mathfrak u_1,\ldots,\mathfrak u_n)
&=
\wt\app(\mathfrak u_1,\ldots,\mathfrak u_n,\mathfrak f),
\\
e(\mathfrak r)(\mathfrak u_1,\ldots,\mathfrak u_n)
&\;\;\text{iff}\;\;\,
\wt\inn
(\mathfrak u_1,\ldots,\mathfrak u_n,\mathfrak r)
\end{align*}
hold for all\,
$
\mathfrak f\in
\scc(Y^{X_1\times\ldots\times X_n}), 
$
$
\mathfrak r\in
\scc\,\mathcal{P}(X_1\times\ldots\times X_n),
$ 
and\, 
$
\mathfrak u_1\in
\scc X_1,\ldots,\mathfrak u_n\in\scc X_n?
$ 
Does this hold at least for all compact Hausdorff 
spaces~$Y$ that are zero-dimensional, 
or extremally disconnected? 
\end{problem}

\begin{problem}\label{prb: 4}
What are topological properties of the subset of 
the space $\scc(Y^{X_1\times\ldots\times X_n})$ 
consisting of pseudo-principal ultrafilters?
Of the preimage of this set under~$e$, i.e.~the 
set $\{\wt{f}:f\in Y^{X_1\times\ldots\times X_n}\}$,
in the space
$
RC_{X_1,\ldots,X_{n-1}}
(\scc X_1,\ldots,\scc X_n,\scc Y)
$ 
with the $(X_1,\ldots,X_n)$-pointwise convergence 
topology (except for the fact that it is dense there, 
as stated in Lemma~\ref{extensions-are-dense}), 
or with the (usual) pointwise convergence topology? 
in the space
$(\scc Y)^{\,\scc X_1\times\ldots\times\scc X_n}$ 
with the pointwise convergence topology?

Often objects naturally defined in terms 
of ultrafilter extensions have rather hardly 
definable topological properties, as shown in 
\cite{Hindman Strauss topol-properties, 
Hindman Strauss non-Borel}.
\end{problem}

In the next two problems, we wonder about 
variants of the definition of the satisfiability 
in ultrafilter models in the narrow sense.

\begin{problem}\label{prb: 5}
Define an alternative satisfaction relation~$\VDash$ 
by using rather $\inn^{\!\!*}$ than~$\wt\inn$; 
i.e.~if $R(t_1,\ldots,t_n)$ is an atomic formula 
in which $R$~is not the equality predicate, let
$$
\mathfrak A\VDash
R(t_1,\ldots,t_n)\;[v]
\;\;\text{iff}\;\;
\inn^{\!\!*}
(v_\imath(t_1),\ldots,v_\imath(t_n),\imath(P)).
$$
Does this give a~$E$-counterpart of the semantics  
of ultrafilter models in the narrow sense? 
More precisely, is the following $E$-counterpart 
of Theorem~\ref{e-preserves-formulas} true: 
If $\mathfrak A$ is an ultrafilter model in 
the narrow sense, then for all formulas~$\varphi$ 
and elements $\mathfrak u_1,\ldots,\mathfrak u_n$ 
of the universe of~$\mathfrak A$, 
$$
\mathfrak{A}\VDash\varphi\,
[\mathfrak u_1,\ldots,\mathfrak u_n]
\;\;\text{iff}\;\;
E(\mathfrak{A})\vDash\varphi\,
[\mathfrak u_1,\ldots,\mathfrak u_n]
$$
(where $\VDash$ has this new meaning)?
\end{problem}

\begin{problem}\label{prb: 6}
Another way to vary the definition~$\VDash$ 
is by letting
$$
\mathfrak A\VDash 
R(t_1,\ldots,t_n)\;[v]
\;\;\text{iff}\;\;
\bigl(\forall^{\,\imath(R)}Q\bigr)\;
\wt Q(v_\imath(t_1),\ldots,v_\imath(t_n)).
$$
This version looks less smooth. Does this, 
nevertheless, give something interesting?
\end{problem}

\begin{problem}\label{prb: 7}
To define the map~$I$, we considered the set 
$\RegCl$ of regular closed subsets of the space 
$\scc X_1\times\ldots\times\scc X_n$ with 
a~topology turning it into a~space homeomorphic 
to the usual product space
$2^{X_1\times\ldots\times X_n}$ with the discrete
space~$2$. Redefine this topology on~$\RegCl$ 
as a~restricted version of Vietoris topology 
(in an analogy with the restricted version of 
pointwise convergence topology turning out 
$\RC$ into a~compact Hausdorff space homeomorphic 
to the product space
$Y^{X_1\times\ldots\times X_n}$). Note that 
in the usual Vietoris topology, the space
$\{Q\subseteq\scc X_1\times\ldots\times\scc X_n:
Q$~is closed$\}$ is compact Hausdorff but 
$\RegCl$ is not a~closed subspace of~it.
\end{problem}

\begin{problem}\label{prb: 8}
Investigate filter extensions of 
first-order models (as was started in 
\cite{Goranko,Saveliev(2 concepts)}) 
and the corresponding concepts of filter 
interpretations and filter models. 
\end{problem}

\begin{problem}\label{prb: 9}
Isolate and investigate other possible types 
of ultrafilter extensions (in the sense of 
Definition~\ref{def: abstract u e}), besides 
the $\wt{\;}$- and ${}^*$-extensions,  
establish special features of the two canonical 
extensions among others (as was proposed 
at the end of~\cite{Saveliev(2 concepts)}).
\end{problem}

\begin{problem}\label{prb: 10}
Investigate ultrafilter extensions of syntax 
(including those of languages, of valuation and 
interpretation maps, of the satisfaction relation).
\end{problem}

\begin{problem}\label{prb: 11}
Investigate iterations of ultrafilter extensions 
(taking unions at limit steps). 
\end{problem}

\begin{problem}\label{prb:12 }
Investigate higher-order and infinitary analogs 
of ultrafilter extensions and ultrafilter 
interpretations, more generally, analogs 
for model-theoretic languages (in the sense 
of~\cite{Barwise Feferman}).
\end{problem}

\begin{problem}\label{prb: 13}
The ${}^*$-extensions play a~special role 
in modal propositional logic; if $\mathfrak A$ is 
a~model of a~relational language, all canonical 
modal formulas are preserved under passing from 
$\mathfrak A$ to~$\mathfrak A^*$, provided both 
first-order models are considered as Kripke frames 
(see \cite{van Benthem,Blackburn et al}). What is 
a~(non-classical) propositional logic with a~similar  
property w.r.t.~the $\,\wt{\;}\,$-extensions? 
(Perhaps, this connects to Shelah's theorem on 
fragments of second-order logic, see 
\cite{Shelah,Barwise Feferman}.)
\end{problem}

\begin{problem}\label{prb: 14}
Do the concepts of ultrafilter interpretations 
and ultrafilter models have any interesting 
applications? e.g.~combinatorial (Ramsey-theoretic) 
applications in model theory? 
\end{problem}


Also a~list of problems related to 
ultrafilter extensions can be found in 
Section~5 of~\cite{Saveliev(inftyproc)}.


\hide{
\begin{acknowledgment}
\hide{
Professor Robert I.~Goldblatt provided us 
some useful historical information about  
the ${}^*$-extension of relations, and 
Professor Neil Hindman, about the emergence of 
studies in ultrafilter extensions of semigroups. 
We would like to express our gratitude to them.
}
We would like to express our gratitude to 
Professors Robert I.~Goldblatt and Neil Hindman 
who provided us some useful historical information.
We are also indebted two anonymous referees 
for some critical remarks and suggestions.
\end{acknowledgment}
}

\vskip+1em
\noindent
\textbf{\textit{Acknowledgement.}}
We would like to express our gratitude to 
Professors Robert I.~Goldblatt and Neil Hindman 
who provided us some useful historical information.
We are also indebted to two anonymous referees 
for some critical remarks and suggestions.

\end{document}